\documentclass[A4paper,11pt]{article}
\usepackage{amsmath}
\usepackage{amsfonts}
\usepackage{amssymb}
\usepackage{amsthm}
\usepackage{graphicx}
\usepackage{subfigure}
\usepackage{fontenc}
\usepackage{wrapfig}
\usepackage{float}
\usepackage{color}
\usepackage{eepic}
\usepackage[hidelinks]{hyperref}
\usepackage[toc,page,title,titletoc,header]{appendix}

\usepackage{geometry}
\newtheorem{thm}{Theorem}[section]
\newtheorem{prop}[thm]{Proposition}
\newtheorem{lem}[thm]{Lemma}
\newtheorem{sublem}[thm]{Sublemma}
\newtheorem{coro}[thm]{Corollary}

\newtheorem*{thm*}{Theorem}
\newtheorem*{prop*}{Proposition}
\newtheorem*{coro*}{Corollary}
\theoremstyle{definition}

\theoremstyle{remark}
\newtheorem{rmq}[thm]{Remark}
\newtheorem*{rmq*}{Remark}

\begin{document}

\title{Existence of  periodic points near  an isolated fixed point with Lefschetz index $1$ and zero rotation for  area preserving surface homeomorphisms}
\author{ YAN Jingzhi\footnote{ Institut de Math\'ematiques de Jussieu-Paris Rive Gauche,
4 place Jussieu,
75252 PARIS CEDEX,
jingzhi.yan@imj-prg.fr
}}
\maketitle

\begin{abstract}
Let $f$ be an orientation and area preserving diffeomorphism of an oriented surface $M$ with an isolated degenerate fixed point $z_0$ with Lefschetz index one. Le Roux conjectured that $z_0$ is accumulated by periodic orbits. In this article, we will approach Le Roux's conjecture by proving that if $f$ is isotopic to the identity by an isotopy fixing $z_0$ and if the area of $M$ is finite, then $z_0$ is accumulated not only by periodic points, but also by periodic orbits in the measure sense. More precisely, the Dirac measure at $z_0$ is the limit  in weak-star topology of a sequence of invariant probability measures supported on periodic orbits. Our proof is purely topological and will works for homeomorphisms and is related to the notion of local rotation set.
\end{abstract}



%

\section{Introduction}\label{S: Introduction}

The goal of this article is to give a result of existence of periodic orbits for area preserving homeomorphisms of  surfaces. We will begin this introduction by explaining our result in the more general case of homeomorphisms, then will explain what does it mean in the case of diffeomorphisms and will conclude by giving its significance in the symplectic formalism.

Let $f$ be a homeomorphism of an oriented surface $M$ that is isotopic to the identity. We say that $f$ is \emph{area preserving} if it preserves a Borel measure without atom such that the measure of each open set is positive and that the measure of each compact set is finite. We call $I=(f_t)_{t\in[0,1]}$ an \emph{identity isotopy of $f$} if  it is an isotopy from the identity to $f$, and call a fixed point of $f$  a \emph{contractible fixed point}  associated to $I$ if its trajectory along $I$ is a loop homotopic to zero in $M$. We say that a fixed point  of $f$ is a \emph{fixed point of $I$} if it is fixed along the isotopy, and denote by $\mathrm{Fix}(I)$ the set of fixed points of $I$.

Suppose that $f$ is an area preserving homeomorphism of $M$, that $I$ is an identity isotopy of $f$, and that $z$ is a fixed point of $I$. We say that $f$ can be \emph{blown-up} at $z$ if we can replace $z$ with a circle  and extend $f$ continuously to this circle. In particular, when $f$ is a diffeomorphism near $z$,  the blow-up can be induced by $Df(z)$. We can define the \emph{blow-up rotation number} $\rho(I,z)$ to be a representative of the Poincar\'e rotation number of the homeomorphism on the added circle. (The precise definitions can be found in Section \ref{S: pre-local rotation set}.)

Suppose that $f$ can be blown-up at $z\in\mathrm{Fix}(I)$, and that  $M$ (resp. $M\setminus (\mathrm{Fix}(I)\setminus\{z\})$) is neither a sphere nor a plane. When the blow-up rotation number $\rho(I,z)$ is different from $0$, we lift $I$ (resp. $I|_{M\setminus(\mathrm{Fix}(I)\setminus\{z\})}$) to the universal covering space and get an identity isotopy of a lift of $f$ (resp. $f|_{M\setminus(\mathrm{Fix}(I)\setminus\{z\})}$). If we fix a pre-image of $z$, we can blow up this point and get an area preserving homeomorphism $\widetilde{f}$ of the half-open annulus. Every other pre-image of $z$  in the annulus is a fixed point of $\widetilde{f}$ with rotation number $0$, and the rotation number of $\widetilde{f}$ on the boundary is different from $0$. (A formal definition of the rotation number can be found in Section \ref{S: pre-Poincare-Birkhoof}.)  By a generalization of Poincar\'e-Birkhoff Theorem, we can deduce that $\widetilde{f}$ has infinitely many periodic points that correspond to different contractible periodic points of $f$. So, we can prove the existence of infinitely many contractible periodic points.

When the blow-up rotation number $\rho(I,z)$ is  $0$, the problem is much more difficult. We will be interested in this article in the case where $\rho(I,z)=0$ and where the Lefschetz index is equal to $1$. It must be noticed that this situation does not occur in the case of a diffeomorphism with no degenerate\footnote{Here, degenerate means that $1$ is an eigenvalue of the Jacobian matrix of $f$ at $z_0$.} fixed point. It is a critical case, but an interesting one for the following reason: there are many situations where existence of a fixed point of index one can be proven. For example, every orientation and area preserving homeomorphism of the sphere with finitely many fixed point has at least two fixed points of Lefschetz index $1$. It is a consequence of Lefschetz formula and of the fact that the Lefschetz index of an orientation and   area preserving homeomorphism at an isolated fixed point is always not bigger that $1$ (see \cite{Slaminka} and \cite{lecalvezindicesup1}). Existence of at least one fixed point of Lefschetz index $1$ can be proven for an area preserving homeomorphism $f$ of a closed surface of positive genus  in the case where $f$ is isotopic to the identity, $f$ has finitely many fixed points,  and  the mean rotation vector vanishes. (See \cite{Franksrotationvector} in the case of a diffeomorphism and \cite{Matsumoto} in the more general case.)

\bigskip

More precisely, suppose that $f:M\rightarrow M$  is an area preserving homeomorphism of an oriented surface $M$,  that  $z_0$ is an isolated  fixed point of $f$ with a Lefschetz index $i(f,z_0)=1$, and that $I$ is an identity isotopy of $f$ fixing $z_0$.  The homeomorphism $f$  can not always be blown-up at $z_0$, nevertheless, Fr\'ed\'eric Le Roux \cite{lerouxrotation} generalized the rotation number and  defined a \emph{local rotation set} $\rho_s(I,z_0)$ (see Section \ref{S: pre-local rotation set}). In particular, when $f$ can be blown-up at $z_0$, the local rotation set is just reduced to the rotation number. We will prove that  if the total area of $M$ is finite and if $\rho_s(I,z_0)$ is reduced to $0$, then $z_0$ is accumulated by contractible periodic points of $f$. More generally, the result is still valid if we relax the condition that $\rho_s(I,z_0)$ is reduced to $0$ to the condition that $\rho_s(I,z_0)$ is reduced to an integer $k$.
 We will prove a stronger result: the Dirac measure $\delta_{z_0}$ at the fixed point $z_0$ is a limit, in the weak-star topology, of a sequence of  invariant probability measures  supported on periodic orbits. We can be more precise. Let us say that a contractible periodic orbit of  period $q$ has \emph{type $(p,q)$ }  if its trajectory  along  the isotopy  is homotopic to $p\Gamma$ in $M\setminus\text{Fix}(I)$, where  $\Gamma$ is the boundary of a sufficiently small Jordan domain containing $z_0$. Then, there exists an open interval $L$ containing an integer $k$ in its boundary such that for all irreducible  $p/q\in L$ there exists a contractible periodic orbit $O_{p/q}$ of type $(p,q)$ and such that $\delta_{z_0}$ is the limit, in the weak-star topology, of any sequence  $(\mu_{O_{p_n/q_n}})_{n\ge 1 }$ such that $\lim_{n\rightarrow \infty}\frac{p_n}{q_n}=k$, where $\mu_{O_{p/q}}$ is the invariant probability measure supported on $O_{p/q}$.

 Formally, we have the following theorem, which is the main result of this article:

\begin{thm}\label{T: main-index 1 and rotation 0 implies  accumulated by periodic points}
Let $f:M\rightarrow M$  be an area preserving homeomorphism of an oriented surface $M$,   $z_0$ be an isolated  fixed point of $f$ with a Lefschetz index $i(f,z_0)=1$, and $I$ be an identity isotopy of $f$ fixing $z_0$ and satisfying  $\rho_s(I,z_0)=\{k\}$. Suppose that one of the following situations occurs,
\begin{itemize}
\item[i)] $M$ is a plane,  $f$ has only one fixed point $z_0$ and  has a periodic orbit besides $z_0$;
\item[ii)] the total area of $M$ is finite.
\end{itemize}
Then, $z_0$ is accumulated by periodic points. More precisely, the following property holds:

 \textbf{P)}: There exists $\varepsilon >0$, such that either for all irreducible  $p/q\in(k,k+\varepsilon)$, or for all irreducible  $p/q\in(k-\varepsilon, k)$, there exists a contractible periodic orbit $O_{p/q}$ of type $(p,q)$, such that  $\mu_{O_{p/q}}\rightarrow \delta_{z_0}$ as $p/q\rightarrow k$, in the weak-star topology, where $\mu_{O_{p/q}}$ is the invariant probability measure  supported on $O_{p/q}$,
\end{thm}

\begin{rmq}
 Le Roux \cite{lerouxrotation} gave the following  conjecture: if $f: (W, z_0)\rightarrow (W', z_0)$ is an orientation and area preserving homeomorphism between two neighborhoods of $z_0\in M$, and if $z_0$ is an isolated fixed point of $f$ such that $i(f,z_0)$ is equal to $1$ and that $\rho_s(I,z_0)$ is reduced to $0$ for a local isotopy $I$ of $f$, then $z_0$ is accumulated by periodic orbits of $f$. Although we can not give a complete answer to this conjecture in this article, we approach it by the previous theorem.
\end{rmq}

\bigskip

When $f$ is a diffeomorphism, we will give several versions of the theorem whose conditions are more easy to understand.

The rotation set of a local isotopy (see Section \ref{S: pre-local rotation set}) at a degenerate fixed point of an orientation preserving diffeomorphism is reduced to an integer. So, given an area-preserving diffeomorphism $f$ of a surface $M$ with finite area that is isotopic to the identity, if $z_0$ is a degenerate fixed point whose Lefschetz index is equal to $1$, the assumptions of  the previous theorem are satisfied, and hence $z_0$ is accumulated by contractible periodic points. Formally, we  have the following corollary:

\begin{coro}\label{C: main-degenerate and index 1 implies accumlated by periodic points}
Let $f$ be an orientation and  area preserving diffeomorphism  of an oriented  surface $M$ with finite total area, and $z_0$ be a degenerate isolated fixed point such that $i(f,z_0)=1$. If $f$ is isotopic to the identity by an isotopy $I$ that fixes $z_0$\footnote{If there exists an identity isotopy of $f$ such that the trajectory of $z_0$ along the isotopy is homotopic to zero, there always exists an identity isotopy of $f$ fixing $z_0$.},   then $z_0$ is accumulated by contractible periodic points. Moreover, the property \textbf{P)} holds.
\end{coro}

Let  $f$ be a $\mathcal{C}^1$ diffeomorphism of $\mathbb{R}^2$.  A function $g:\mathbb{R}^2\rightarrow \mathbb{R}$ of class $\mathcal{C}^2$ is called  a \emph{generating function} of  $f$  if $\partial_{12}^2g<1$, and
\begin{equation*}
 f(x,y)=(X,Y)\Leftrightarrow\left\{\begin{aligned} X-x & = & \partial_2 g(X,y),\\ Y-y & = &  -\partial_1 g(X,y) . \end{aligned}\right.
\end{equation*}
We know that the previous diffeomorphism $f$ is orientation and area preserving  by a direct computation.

Generating functions are usual objects in symplectic geometry. We will give the following version of our result  whose conditions are  described by  generating functions.
\begin{coro}\label{C: main-generating function with a degenrate maximal point means accumlated by periodic points}
Let $f$ be an orientation and area preserving diffeomorphism of an oriented  surface $M$ with finite area. We suppose that $f$ is isotopic to the identity by an isotopy that fixes $z_0$. Suppose that  in a neighborhood of $z_0$, $f$ is  conjugate to a local  diffeomorphism at $0$ that is generated by a generating function $g$,   that $0$ is a local extremum of $g$, and that the Hessian matrix of $g$ at $0$ is degenerate. Then $z_0$ is accumulated by periodic points, and the property \textbf{P)} holds.
\end{coro}

Let us explain now what is our result in the symplectic formalism.

A time-dependent vector field $(X_t)_{t\in\mathbb{R}}$ is called a \emph{Hamiltonian vector field} if it is defined by the equation:
$$dH_t=\omega(X_t,\cdot),$$
where $(M,\omega)$ is a symplectic manifold, and $H:\mathbb{R}\times M\rightarrow \mathbb{R}$ is a smooth function. The Hamiltonian vector field induces a flow $(\varphi_t)_{t\in\mathbb{R}}$ on $M$, which is the solution of  the following equation
$$\frac{\partial}{\partial t}\varphi_t(z)=X_t(\varphi_t(z)).$$
 We say that a diffeomorphism $f$ of $M$ is a \emph{Hamiltonian diffeomorphism} if it is the time-$1$ map of a Hamiltonian flow.

In particular, a Hamiltonian diffeomorphism $f$ of the torus $\mathbb{T}^{2}$ that  is close to the identity in $\mathcal{C}^1$ topology can be lifted to the plane $\mathbb{R}^2$, and the lifted diffeomorphism  can be defined by a generating function $g$. If $z_0$ is a local maximum of $g$, and if  the Hessian of $g$ at $z_0$ is degenerate, we are in the case of the previous corollary, and the image of $z_0$ in $\mathbb{T}^2$ is a fixed point of $f$ that is accumulated by contractible periodic points.


\bigskip

In 1984, Conley conjectured that a Hamiltonian diffeomorphism $f$ of the standard symplectic torus $(T^{2d},\omega)$ has infinitely many contractible periodic points. The result was proved later in the case where $f$ has no degenerate fixed point, by Conley and Zehnder\cite{cznondegenarate}, in the weakly non-degenerate case by Salamon and Zehnder \cite{szweaklynondegenerate}, and finally in the general case by Hinston\cite{hinstondegenerate}. Recently, Mazzucchelli\cite{mazzuchelli} gave a simpler argument based on generating functions for the second part of the proof of Hinston, and noticed that the existence of a symplectical degenerate extremum, that will be defined in the section \ref{S: Discrete symplectic actions and symplectically degenerate extremum}, implies the  existence of infinitely many other periodic points. In the same paper, he asked  whether a symplectical degenerate extremum actually corresponds to a fixed point accumulated by periodic points. In this article, we will give a positive answer in the case where $d=1$. More precisely, we have the following result.

\begin{thm}\label{T: main-symplectic degenerate maxima implies accumlated by periodic points}
Let $f: \mathbb{T}^2\rightarrow \mathbb{T}^2$ be a Hamiltonian diffeomorphism, and $z_0$ be a symplectical degenerate  extremum. Then $z_0$  is accumulated by periodic points, and the property \textbf{P)} holds.
\end{thm}

Now, we will give a plan of this article. In Section \ref{S: Preliminaries}, we will recall some definitions and results that we will use in the proofs of our results. In Section \ref{S: chapter3-homeomorphism case},  we will prove Theorem \ref{T: main-index 1 and rotation 0 implies  accumulated by periodic points}, which is the main result of this paper. In Section \ref{S: chapter3-diffeomorphism case},  we will study a particular  case where $f$ is a diffeomorphism, and will give several versions of our results whose conditions are described in different ways:  Corollary \ref{C: main-generating function with a degenrate maximal point means accumlated by periodic points} and Theorem \ref{T: main-symplectic degenerate maxima implies accumlated by periodic points}.

\section{Preliminaries}\label{S: Preliminaries}

\subsection{A classification of isolated fixed points}\label{S: pre-a classification of isolated fixed points}

In this section, we will give a classification of  isolated fixed points. More details can be found in  \cite{lecalvezindices}.

Let $f:(W,z_0)\rightarrow (W',z_0)$ be a local  homeomorphism with an isolated fixed point $z_0$.
We say that  $z_0$ is an \emph{accumulated point} if every neighborhood of $z_0$ contains a periodic orbit besides $z_0$. Otherwise, we say that $z_0$ is a \emph{non-accumulated point}.

We define a \emph{Jordan domain} to be a bounded domain whose boundary is a simple closed curve. We say that $z_0$ is \emph{indifferent} if there exists a neighborhood $V\subset\overline {V}\subset W$ of $z_0$ such that for every Jordan domain $U\subset V$ containing $z_0$, the connected component of $\cap_{k\in \mathbb{Z}}f^{-k}(\overline{U})$ containing $z_0$ intersects the boundary of $U$.

%

We say that  $z_0$ is \emph{dissipative} if there exists a fundamental system $\{U_{\alpha}\}_{\alpha\in J}$ of the neighborhood of $z_0$
such that each $U_{\alpha}$ is a Jordan domain and that  $f(\partial U_\alpha)\cap \partial U_\alpha=\emptyset$.

We say that  $z_0$ is  a \emph{saddle} point if it is neither  indifferent nor dissipative.

Note that if $f$ is area preserving, an isolated fixed point of $f$ is either an indifferent point  or a saddle point.


\subsection{Lefschetz  index}\label{S: pre-Lefschetz index}

Let $f:(W,0)\rightarrow (W',0)$ be an orientation  preserving local homeomorphism at an isolated fixed point $0\in\mathbb{R}^2$. Denote by $S^1$ the unit circle. If $C\subset W$ is a simple closed curve which contains no fixed point of $f$, then we can define the \emph{ index}  $i(f,C)$ of $f$ along the curve $C$ to be the Brouwer degree of the map
\begin{eqnarray*}
\varphi:S^1&\rightarrow &S^1\\
 t&\mapsto&\frac{f(\gamma(t))-\gamma(t)}{||f(\gamma(t))-\gamma(t)||},
\end{eqnarray*}
where $\gamma: S^1\rightarrow C$ is a parametrization compatible with the orientation, and $\|\cdot\|$ is the usual Euclidean norm.  Let $U$ be a Jordan domain containing $0$ and contained in a  sufficiently small neighborhood of $0$. We define the \emph{Lefschetz index} of $f$ at $0$ to be  $i(f,\partial U)$, which is independent of the choice of  $U$. We denote it by $i(f,0)$.

 More generally, if $f:(W,z_0)\rightarrow (W',z_0)$ is an orientation preserving local homeomorphism at a fixed point $z_0$ on a surface $M$, we can conjugate it topologically to an orientation preserving  local homeomorphism $g$ at $0$ and define the  \emph{Lefschetz index} of $f$ at $z_0$ to be $i(g,0)$, which is independent of the choice of the conjugation. We denote it by $i(f,z_0)$.

 \subsection{Local isotopies and the index of local isotopies}\label{S: pre-local isotopies and index of isotopies}

Let $f: (W,z_0)\rightarrow (W',z_0)$  be an orientation preserving local homeomorphism at $z_0\in M$. A \emph{local isotopy} $I$ of $f$ at $z_0$ is a family of homeomorphisms $(f_t)_{t\in[0,1]}$   such that
 \begin{itemize}
\item[-] every $f_t$ is a homeomorphism between the neighborhoods $V_t\subset W$ and $V'_t\subset W'$ of $z_0$, and $f_0=\mathrm{Id}_{V_0}$, $f_1=f|_{V_1}$;
\item[-] for all $t$, one has $f_t(z_0)=z_0$;
\item[-]the sets $\{(z,t)\in M\times[0,1] : z\in V_t\}$ and  $\{(z,t)\in M\times[0,1] : z\in V'_t\}$ are both open in $M\times [0,1]$;
\item[-] the maps $(z,t)\mapsto f_t(z)$   and $(z,t)\mapsto f^{-1}_t(z)$ are both continuous.
\end{itemize}

We say that two local isotopies of $f$ are \emph{equivalent} if they are locally homotopic.

Let us  introduce the index of a local isotopy which was defined by  Le Roux \cite{lerouxrotation} and  Le Calvez \cite{lecalveztourner}.

 Let $f:(W,0)\rightarrow (W',0)$ be an orientation preserving local homeomorphism at $0\in\mathbb{R}^2$, and $I=(f_t)_{i\in[0,1]}$ be a local isotopy of $f$. We denote by $D_r$ the disk with radius $r$  and centered at $0$. Then the isotopy $f_t$ is well defined in the disk $D_r$ if $r$ is sufficiently small. Let
\begin{eqnarray*}
\pi: \mathbb{R}\times (-\infty, 0)&\rightarrow& \mathbb{C}\setminus\{0\}\simeq \mathbb{R}^2\setminus\{0\}\\
(\theta,y) &\mapsto& -ye^{i2\pi \theta},
 \end{eqnarray*}
 be the universal covering projection, and  $\widetilde{I}=(\widetilde{f}_t)_{t\in[0,1]}$ be the lift of $I|_{D_r\setminus\{0\}}$ to $\mathbb{R}\times (-r,0)$ such that $f_0$ is  the identity. Let $\widetilde{\gamma}:[0,1]\rightarrow \mathbb{R}\times(-r, 0)$ be a path  from $\widetilde{z'}\in\mathbb{R}\times(-r, 0)$ to $\widetilde{z'}+(1,0)$. The map
$$t\mapsto\frac{\widetilde{f}_1(\widetilde{\gamma}(t))-\widetilde{\gamma}(t)}{||\widetilde{f}_1(\widetilde{\gamma}(t))-\widetilde{\gamma}(t)||}$$
takes the same value at both $0$ and $1$,  and hence descends to a continuous map $\varphi: [0,1]/_{0\sim 1}\rightarrow S^1$. We define the \emph{ index of the isotopy} $I$ at $0$   to be the Brouwer degree of $\varphi$, which does not depend on the choice of $\widetilde{\gamma}$ when $r$ is sufficiently small. We denote it by $i(I,0)$.

%
%

 More generally, we consider an orientation preserving local homeomorphism on an oriented surface. Let $f:(W, z_0)\rightarrow (W', z_0)$ be an orientation preserving local homeomorphism at a fixed point $z_0$ in a surface $M$. Let $h: (U,z_0)\rightarrow (U',0)$ be a  local homeomorphism. Then  $h\circ I\circ h^{-1}=(h\circ f_t\circ h^{-1})_{t\in[0,1]}$ is a local isotopy at $0$, and we define the \emph{index of} $I$ at $z_0$ to be $i(h\circ I\circ h^{-1},0)$, which is independent of the choice of  $h$. We denote it by $i(I,z_0)$. 

Let $I=(f_t)_{t\in [0,1]}$ and $I'=(g_t)_{t\in[0,1]}$ be two identity isotopies (resp. local isotopies). We denote by $I^{-1}$ the  isotopy (resp. local isotopy) $(f^{-1}_t)_{t\in [0,1]}$, by $I'I$ the  isotopy (resp. local isotopy) $(\varphi_t)_{t\in[0,1]}$ such that
\begin{eqnarray*}
\varphi_t=\left\{\begin{array}{ll} f_{2t} &\mbox{for $t\in[0,\frac{1}{2}]$},\\ g_{2t-1}\circ f &\mbox{for $t\in[\frac{1}{2},0]$},\end{array}\right.
\end{eqnarray*}
and by $I^n$ the isotopy (resp. local isotopy) $\underbrace{I\cdots I}_{n \textrm{ times}}$ for every $n\ge 1$.


The Lefschetz index at an isolated fixed point  and the indices of the local isotopies are related. We have the following result:

\begin{prop}(\cite{lecalveztourner}\cite{lerouxrotation})\label{P: pre-relation of indices between isotopy and homeo}
Let $f: W\rightarrow W'$ be an orientation preserving homeomorphism  with an isolated fixed point $z$. Then, we have the following results:
 \begin{itemize}
\item[-] if $i(f,z)\neq 1$, there exists a unique homotopy class of local isotopies  such that $i(I,z)=i(f,z)-1$ for every local isotopy $I$ in this class, and the indices of the other local isotopies are equal to $0$;
\item[-] if $i(f,z)= 1$, the indices of all the local isotopies are equal to $0$.
\end{itemize}
\end{prop}

\subsection{Brouwer plane translation theorem}\label{S: pre-Brouwer plane translation theorem}

In this section, we will recall the Brouwer plane translation theorem. More details can be found in \cite{Brouwer}, \cite{Guillou} and \cite{FranksBrouwertranslationtheorem}.

Let $f$ be an orientation preserving homeomorphism of  $\mathbb{R}^2$. If $f$ does not have any fixed point, the Brouwer plane translation theorem asserts that every $z\in\mathbb{R}^2$ is contained in a translation domain for $f$, i.e. an open connected set of $\mathbb{R}^2$ whose boundary is $L\cup f(L)$, where $L$ is the image of a proper embedding of $\mathbb{R}$ in $\mathbb{R}^2$ such that $L$ separates $f(L)$ and $f^{-1}(L)$.

As an immediate corollary, one knows that if $f$ is an orientation and area preserving homeomorphism of a plane\footnote{Here, a plane  is an open set homeomorphic to $\mathbb{R}^2$.} with finite area, it has at least one fixed point.

\subsection{Transverse foliations and its index  at an isolated end}\label{S: pre-index of foliation}

In this section, we will introduce the index of a foliation at an isolated end. More details can be found in \cite{lecalveztourner}.

Let $M$ be an oriented surface and $\mathcal{F}$ be an oriented topological foliation on $M$. For every point $z$, there is a neighborhood $V$ of $z$ and a homeomorphism $h: V\rightarrow (0,1)^2$ preserving the orientation such that the images of the leaves of $\mathcal{F}|_{V}$ are the vertical lines oriented upward. We call $V$  a  \emph{trivialization neighborhood} of $z$, and $h$ a \emph{trivialization chart}.

 Let $z_0$ be an isolated end of $M$.  We choose a small annulus $U\subset M$ such that $z_0$ is an end of $U$.  Let   $h: U\rightarrow \mathbb{D}\setminus\{0\}$ be a homeomorphism which sends $z_0$ to $0$ and preserves the orientation. Let $\gamma: \mathbb{T}^1\rightarrow \mathbb{D}\setminus \{0\}$ be a simple closed  curve homotopic to $\partial\mathbb{D}$. We can cover the curve by finite trivialization  neighborhoods $\{V_i\}_{1\leq i\leq n}$ of the foliation $\mathcal{F}_h$, where $\mathcal{F}_{h}$ is the image of  $\mathcal{F}|_{U}$. For every $z\in V_i$, we denote by $\phi_{V_i, z}^{+}$ the positive half leaf of the the leaf in $V_i$ containing $z$. Then we can construct a continuous map $\psi$ from the curve $\gamma$ to $\mathbb{D}\setminus \{0\}$, such that $\psi(z)\in \phi_{V_i,z}^{+}$ for all $0\le i\le n$ and for all $z\in V_i$.  We define the  \emph{index}  $i(\mathcal{F},z_0)$ of $\mathcal{F}$ at  $z_0$ to be the Brouwer degree of the application
$$\theta\mapsto \frac{\psi(\gamma(\theta))-\gamma(\theta)}{\|\psi(\gamma(\theta))-\gamma(\theta)\|},$$  which  depends neither on the choice of $\psi$, nor on the choice of $V_i$, nor on the choice of $\gamma$, nor on the choice of $h$.

We say that a path $\gamma:[0,1]\rightarrow M$ is \emph{positively transverse} to $\mathcal{F}$, if for every $t_0\in[0,1]$, there exists a trivialization neighborhood $V$ of $\gamma(t_0)$ and $\varepsilon>0$ such that $\gamma([t_0-\varepsilon,t_0+\varepsilon]\cap[0,1])\subset V$ and $h\circ \gamma|_{[t_0-\varepsilon,t_0+\varepsilon]\cap[0,1]}$ intersects the vertical lines  from left to right, where $h: V\rightarrow (0,1)^2$ is the trivialization chart.

Let $f$ be a homeomorphism on $M$ isotopic to the identity, and $I=(f_t)_{t\in[0,1]}$ be an identity isotopy of $f$.  We say that an oriented foliation $\mathcal{F}$ on $M$ is a \emph{transverse foliation}  of $I$ if for every $z\in M$, there is a path  that is homotopic to the trajectory $t\rightarrow f_t(z)$ of $z$ along $I$ and  is positively transverse to   $\mathcal{F}$.

Suppose that $I=(f_t)_{t\in[0,1]}$ is a local isotopy at $z_0$. We say that $\mathcal{F}$  is \emph{locally transverse} to $I$  if for every sufficiently small neighborhood $U$ of $z_0$, there exists a neighborhood $V\subset U$ such that for all $z\in V\setminus \{z_0\}$, there exists a path  in $U\setminus\{z_0\}$ that is homotopic to the trajectory $t\mapsto f_t(z)$  of $z$ along  $I$  and  is positively transverse to $\mathcal{F}$.

\begin{prop}\cite{lecalveztourner}
Suppose that $I$ is an identity isotopy  on a surface $M$ with an isolated end $z$ and $\mathcal{F}$ is a transverse foliation of $I$. If $M$ is not a plane, $\mathcal{F}$ is also locally transverse to the local isotopy $I$ at $z$.
\end{prop}

\begin{prop}\cite{lecalveztourner}\label{P: pre-relations of indices between  foliation and the others}
Let  $f: (W,z_0)\rightarrow (W',z_0)$ be an orientation preserving local homeomorphism at an isolated fixed point $z_0$, $I$ be a local isotopy of $f$ at $z_0$, and $\mathcal{F}$ be a  foliation that is locally transverse to $I$, then
\begin{itemize}
\item[-]   $i(\mathcal{F}, z_0)=i (I,z_0)+1$;
\item[-] $i(f,z_0)=i(\mathcal{F}, z_0)$  if $i(\mathcal{F}, z_0)\ne 1$.
\end{itemize}
\end{prop}

\subsection{Existence of a transverse foliation and Jaulent's  preorder}\label{S: pre-Jaulent's preorder}
Let $f$ be a homeomorphism of $M$ isotopic to the identity, and $I=(f_t)_{t\in[0,1]}$ be an identity isotopy of $f$. A \emph{contractible fixed point} $z$ of $f$ associated to $I$ is a  fixed point of $f$ such that  the trajectory of $z$ along $I$, that is the path  $t\mapsto f_t(z)$, is a  loop homotopic to zero in $M$. One has the following generalization of Brouwer's translation theorem.

\begin{thm}\cite{lecalvezfeuilletage}\label{T: pre-feuilletage}
Let $M$ be an oriented surface. If $I=(f_t)_{t\in[0,1]}$ is an identity isotopy of a homeomorphism $f$ of $M$ such that there exists no contractible fixed point of $f$ associated to $I$, then there exists a transverse  foliation $\mathcal{F}$ of $I$.
\end{thm}

One can extend this result to the case where there exist contractible fixed points by defining the following preorder of Jaulent \cite{Jaulent}.

Let us denote  by $\mathrm{Fix}(f)$  the set of fixed points of $f$,  and for every identity isotopy $I=(f_t)_{t\in[0,1]}$  of  $f$,  by $\mathrm{Fix}(I)=\cap_{t\in[0,1]}\mathrm{Fix}(f_t)$ the set of fixed points of $I$. Let  $X$ be a closed subset of $\mathrm{Fix}(f)$. We denote by $(X, I_X)$ the couple that consists of a closed subset $X\subset\mathrm{Fix}(f)$ such that $f|_{M\setminus X}$ is isotopic to the identity and an identity  isotopy $I_X$  of $f|_{M\setminus X}$.

Let $\pi_X:\widetilde{M}_X\rightarrow M\setminus X$ be the universal cover, and $\widetilde{I}_X=(\widetilde{f}_t)_{t\in[0,1]}$ be the identity isotopy  that  lifts  $I_X$.  We say that $\widetilde{f}_X=\widetilde{f}_1$ is \emph{the lift of $f$ associated to $I_X$}. We say that a path $\gamma:[0,1]\rightarrow  M\setminus X$ from  $z$ to $f(z)$ is \emph{associated} to $I_X$ if there exists a path $\widetilde{\gamma}:[0,1]\rightarrow \widetilde{M}_X$  that is the lift of $\gamma$ and satisfies $\widetilde{f}_X(\widetilde{\gamma}(0))=\widetilde{\gamma}(1)$.
We write $(X, I_X)\precsim (Y, I_Y)$, if
\begin{itemize}
\item[-] $X\subset Y\subset (X\cup\pi_X(\mathrm{Fix}(\widetilde{f}_X)))$;
\item[-]all the paths in $M\setminus Y$ associated to $I_Y$ are also associated to $I_X$.
\end{itemize}
The preorder $\precsim$ is well defined. Moreover, if one has $(X, I_X)\precsim (Y, I_Y)$ and $(Y, I_Y)\precsim (X, I_X)$, then one knows that $X=Y$ and that   $I_X$ is homotopic to  $I_Y$. In this case, we will write $(X,I_X)\sim(Y, I_Y)$. Jaulent proved the following result:

\begin{thm}\cite{Jaulent}\label{T: pre-exist max isotopy}
Let $M$ be an oriented surface and $I$ be an identity isotopy of a homeomorphism $f$ on $M$.  Then, there exists a maximal $(X,I_X)\in\mathcal{I}$ such that $(\mathrm{Fix}(I),I)\precsim(X,I_X)$. Moreover, $f|_{M\setminus X}$  has no contractible fixed point associated to $I_X$, and there exists a  transverse foliation $\mathcal{F}$ of $I_X$ on $M\setminus X$.
\end{thm}

 \begin{rmq}
 Here, we can also consider the previous foliation $\mathcal{F}$ to be a singular foliation on $M$ whose  singularities are the points in $X$.  In particular, if $I_X$ is the restriction to $M\setminus X$  of an identity isotopy $I'$ on $M$, we will say that  $\mathcal{F}$ a transverse (singular) foliation of $I'$.
\end{rmq}

 We call $(Y,I_Y)\in\mathcal{I}$  a \emph{maximal extension} of $(X,I_X)$ if $(X,I_X)\precsim (Y, I_Y)$ and if $(Y,I_Y)$ is maximal in Jaulent's preorder; we call  $(Y,I_Y)$  a \emph{ maximal extension} of $I$ if $(Y,I_Y)$ is a maximal extension of $(\mathrm{Fix}(I),I)$.

\subsection{Dynamics of an oriented foliation in a neighborhood of an isolated singularity}\label{S: pre-dynamics of foliation}
In this section, we  consider  singular foliations. A \emph{sink}  (resp. a\emph{ source}) of $\mathcal{F}$ is an isolated singular point of $\mathcal{F}$ such that  there is a homeomorphism $h: U\rightarrow \mathbb{D}$ which sends  $z_0$ to $0$ and sends the restricted foliation $\mathcal{F}|_{U\setminus\{z_0\}}$ to the radial foliation of $\mathbb{D}\setminus\{0\}$ with the leaves toward (resp. backward) $0$, where $U$ is a neighborhood of $z_0$ and $\mathbb{D}$ is the unit disk. A \emph{petal} of $\mathcal{F}$ is a closed topological disk whose boundary is the union of a leaf and a singularity. Let $\mathcal{F}_0$ be the foliation on $\mathbb{R}^2\setminus\{0\}$ whose leaves are the horizontal lines except the $x-$axis which is cut  into two leaves. Let $S_0=\{y\ge 0:x^2+y^2\le 1\}$ be the half-disk.  We call a closed topological disk $S$ a \emph{hyperbolic sector} if there exist
\begin{itemize}
\item[-] a closed set $K\subset S$ such that $K\cap\partial S$ is reduced to a singularity $z_0$ and $K\setminus\{z_0\}$ is the union of the leaves of $\mathcal{F}$ that are contained in $S$,
\item[-] a continuous map $\phi: S\rightarrow S_0$ that maps $K$ to $0$ and the leaves of $\mathcal{F}|_{S\setminus K}$ to the leaves of $\mathcal{F}_0|_{S_0}$.
\end{itemize}

\begin{figure}[h]
  \subfigure[the hyperbolic sector model $S_0$]{
    \begin{minipage}[t]{0.25\linewidth}
      \centering
   \includegraphics[width=3cm]{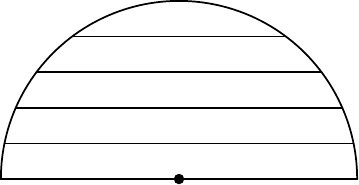}
    \end{minipage}}
  \hspace{0.05\linewidth}
  \subfigure[a pure hyperbolic sector]{
    \begin{minipage}[t]{0.25\linewidth}
      \centering
     \includegraphics[width=2cm]{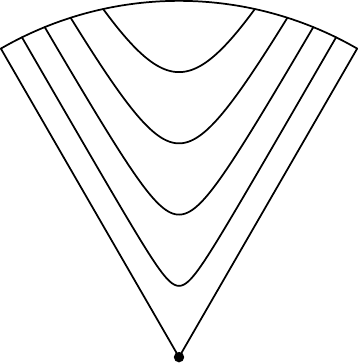}
    \end{minipage}}
    \hspace{0.05\linewidth}
    \subfigure[a strange hyperbolic sector]{
    \begin{minipage}[t]{0.25\linewidth}
      \centering
   \includegraphics[width=2cm]{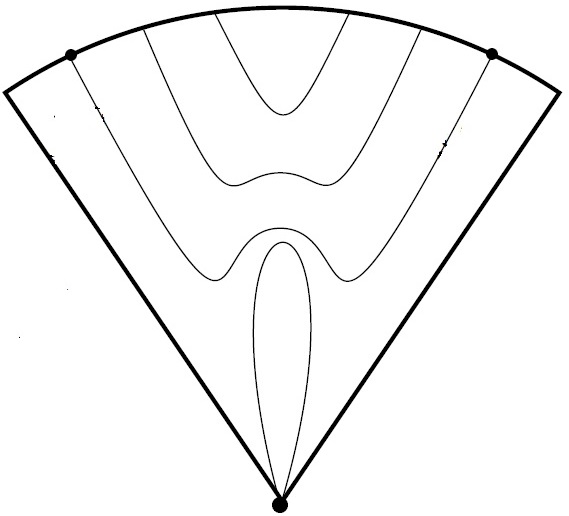}
    \end{minipage}}
  \caption{The hyperbolic sectors}
\end{figure}

 Le Roux gives a description of  the  dynamics of an oriented foliation $\mathcal{F}$ at an isolated singularity $z_0$.

 \begin{prop}\cite{lerouxrotation}\label{P: pre-dynamics of a foliation}
 We have one of the following cases:
 \begin{itemize}
 \item[i)](sink or source) there exists a neighborhood of $z_0$ that contains neither a closed leaf, nor a petal, nor a hyperbolic sector;
 \item[ii)](cycle) every neighborhood of $z_0$ contains a closed leaf;
 \item[iii)](petal) every neighborhood of $z_0$ contains a petal, and does not contain any hyperbolic sector;
 \item[iv)](saddle) every neighborhood of $z_0$ contains a hyperbolic sector,  and does not contain any petal;
 \item[v)](mixed) every neighborhood of $z_0$ contains both a petal and a hyperbolic sector.
 \end{itemize}
Moreover, $i(\mathcal{F},z_0)$ is equal to $1$ in the first two cases,  is strictly bigger than $1$ in the petal case, and is strictly smaller than $1$ in the saddle case.
 \end{prop}

\begin{rmq}
In particular, let $f:(W,z_0)\rightarrow (W',z_0)$ be an orientation preserving local homeomorphism at $z_0$, $I$ be a local isotopy of $f$, $\mathcal{F}$ be an oriented foliation that is locally transverse to $I$, and $z_0$ be an isolated singularity of $\mathcal{F}$. If $P$ is a petal in a small neighborhood of $z_0$ and $\Phi$ is the leaf in $\partial P$, then $\Phi\cup\{z_0\}$ divides $M$ into two parts. We denote by $L(\Phi)$ the one to the left and $R(\Phi)$ the one to the right. By  definition, $P$ contains the positive orbit of $R(\Phi)\cap L(f(\Phi))$  or the negative orbit  of $ L(\Phi)\cap R(f^{-1}(\Phi))$. Then, a petal in a small neighborhood of $z_0$ contains the positive or the negative orbit of a wandering open set. So does the topological disk whose boundary is a closed leaf in a small neighborhood of $z_0$. Therefore, if $f$ is area preserving, then $z_0$ is either a sink,  a source, or a saddle of $\mathcal{F}$.
\end{rmq}
%

\subsection{The local rotation type of a local isotopy}\label{S: pre-local rotation type}

In this section, suppose that  $f:(W,z_0)\rightarrow(W',z_0)$ is an orientation and area preserving local homeomorphism at an isolated fixed point $z_0$, and that $I$ is a local isotopy of $f$. We say that $I$ has a \emph{positive rotation type} (resp. \emph{negative rotation type}) if there exists a locally transverse foliation $\mathcal{F}$ of $I$ such that $z_0$ is a sink (resp. source) of $\mathcal{F}$. Shigenori Matsumoto \cite{Matsumoto} proved the following result:

\begin{prop}\cite{Matsumoto}\label{P: pre-rotation type is unique}
If $i(f,z_0)$ is equal to $1$, $I$ has unique one of the two kinds of  rotation types.
\end{prop}

\begin{rmq}
By considering the index of foliation, one  deduces the following corollary:
\textit{if $\mathcal{F}$ and $\mathcal{F}'$ are two locally transverse foliations of $I$, and if $0$ is a sink (resp. source) of $\mathcal{F}$, then $0$ is a sink (resp. a source) of $\mathcal{F}'$.}
\end{rmq}

\subsection{Prime-ends compactification and  rotation number}\label{S: pre-prime-ends Compactification}
In this section, we first recall some facts and definitions from Carath\'eodory's prime-ends theory, and then give the definition of the prime-ends rotation number. More details can be found in \cite{milnor} and \cite{Lecalvezprimeendsrotationnumberandperiodicpoints}.

Let $U\varsubsetneq\mathbb{R}^2$ be a simply connected domain, then there exists a natural compactification of $U$  by adding a circle, that can be defined in different ways.  One explanation is the following: we can identify $\mathbb{R}^2$ with $\mathbb{C}$ and consider a conformal diffeomorphism
$$h: U\rightarrow\mathbb{D},$$
where $\mathbb{D}$ is the unit disk. We endow $U\sqcup S^1$ with the topology of the pre-image of the natural topology of $\overline{\mathbb{D}}$ by the application
$$\overline{h}:U\sqcup S^1\rightarrow \overline{\mathbb{D}},$$
whose restriction is $h$ on $U$ and the identity on $S^1$ .

Any arc in $U$ which lands at a point $z$ of $\partial U$ corresponds, under $h$, to an arc in $\mathbb{D}$ which lands at a point of $S^1$, and arcs which land at distinct points of $\partial U$ necessarily correspond to arcs which land at distinct points of $S^1$. We define an \emph{end-cut} to be the image of a simple arc $\gamma:[0,1)\rightarrow U$ with a limit point in $\partial U$. Its image by $h$ has a limit point in $S^1$. We say that two end-cuts are \emph{equivalent} if their images have the same limit point in $S^1$. We say that a point $z\in\partial U$ is \emph{accessible} if there is an end-cut that lands at $z$.  Then the set of points of $S^1$ that are limit points of an end-cut is dense in $S^1$, and accessible points of $\partial U$ are dense in $\partial U$.
We define a \emph{cross-cut}  by the image of a simple arc $\gamma:(0,1)\rightarrow U$ which extends to an arc $\overline{\gamma}:[0,1]\rightarrow \overline{U}$ joining two points of $\partial U$ and such that each of the two components of $U\setminus \gamma$ has a boundary point in $\partial U\setminus\overline{\gamma}$.

Let $f$ be an orientation preserving homeomorphism of $U$. We can extend $f$ to  a homeomorphism of the prime-ends compactification $U\sqcup S^1$, and denote it by $\overline{f}$. In fact, for a point $z\in S^1$ which is a limit point of an end-cut $\gamma$, we can naturally define $\overline{f}(z)$ to be the limit point of $f\circ\gamma$. Then we  can define the \emph{prime-ends rotation number} $\rho(f, U)\in \mathbb{T}^1=\mathbb{R}/\mathbb{Z}$ to be  the Poincar\'e's rotation number of  $\overline{f}|_{S^1}$. In particular, if $f$ fixes every point in $\partial U$, $\rho(f,U)=0$.

\bigskip

Let $K\subset \mathbb{R}^2$ be a continuum, and $U_K$ be the unbounded component of $\mathbb{R}^2\setminus K$. Then, $U_K$ is an annulus and becomes a simply connected domain of the Riemann sphere if we identify $\mathbb{R}^2$ with $\mathbb{C}$ and add a point at infinity. The prime-ends compactification also gives us a compactification of the end of $U_K$ corresponding to $K$ by  adding the circle of prime-ends. We can define  \emph{end-cuts} and \emph{cross-cuts} similarly.

Let $f:(W,0)\rightarrow (W',0)$ be an orientation preserving local homeomorphism at $0\in\mathbb{R}^2$,  and  $K\subset W$ be an invariant continuum containing $0$.  Similarly, we can naturally extend $f|_{U_K\cap W}$ to  a homeomorphism $f_K: U_K\cap W\cup S^1\rightarrow U_K\cap W'\cup S^1$ , and define the \emph{rotation number} $\rho(f,K)\in \mathbb{R}/\mathbb{Z}$ to be  the Poincar\'e's rotation number of  $f_K|_{S^1}$.

Furthermore, if $I=(f_t)_{t\in[0,1]}$ is a local isotopy of $f$ at $0$,  we consider the universal covering projections
\begin{eqnarray*}
\pi: \mathbb{R}\times (-\infty,0)&\rightarrow&\mathbb{C}\setminus\{0\}\simeq\mathbb{R}^2\setminus \{0\}\\
(\theta,y)&\mapsto&-ye^{i 2\pi \theta}
\end{eqnarray*}
and
\begin{eqnarray*}
\pi': \mathbb{R}&\rightarrow& S^1\\
\theta &\mapsto&e^{i 2\pi \theta}.
\end{eqnarray*}
Let $\widetilde{U}_K=\pi^{-1}(U_K)$, $\widetilde{W}=\pi^{-1}(W)$, and $\widetilde{W}'=\pi^{-1}(W')$.  Let
$$\pi_K: \widetilde{U}_K\sqcup \mathbb{R}\rightarrow U_K\sqcup \mathbb{T}^1$$
be the map such that $\pi_K=\pi$ in $\widetilde{U}_K$ and $\pi_K=\pi'$ on $\mathbb{R}$.  We endow the topology on $\widetilde{U}_K\sqcup R$ such that  $\pi_K$ is a universal cover.  Let $\widetilde{I}=(\widetilde{f}_t)_{t\in[0,1]}$ be  the lift of $(f_t|_{V\setminus\{0\}})_{t\in[0,1]}$ such that $\widetilde{f}_0$ is the identity, where $V$ is a small neighborhood of $z$. Let  $\widetilde{f}: \widetilde{W}\rightarrow\widetilde{W}'$ be the lift of $f|_{W\setminus\{0\}}$  such that $\widetilde{f}=\widetilde{f}_1$ in $\pi^{-1}(V)$, we call it the \emph{lift of $f$ associated to $I$}. Let $\widetilde{f}_K:(\widetilde{W}\cap\widetilde{U}_K)\sqcup\mathbb{R}\rightarrow(\widetilde{W}'\cap\widetilde{U}_K)\sqcup \mathbb{R}$ be the lift of $f_K$ such that $\widetilde{f}_K=\widetilde{f}$ in $\widetilde{W}\cap\widetilde{U}_K$,  we call it \emph{the lift of $f_K$ associated to $I$}.  We define the \emph{rotation number} $\displaystyle\rho (I,K)=\lim_{n\rightarrow \infty}\frac{\widetilde{f}^n_K(\theta)-\theta}{n}$ which is a  real number that does not depend on the choice of $\theta$.  We know that $\rho(I, K)$ is a representative of $\rho(f,K)$ in $\mathbb{R}$.

We have the following property:

%

\begin{prop}\cite{lecalvezindices}\label{P: pre-existence of rotation number of indifferent point}
Let $f:(W,0)\rightarrow (W',0)$ be an orientation preserving local homeomorphism at a non-accumulated indifferent point $0$. Let $U\subset \overline{U}\subset W$ be a Jordan domain such that $\overline{U}$ does not contain any periodic orbit except $0$, and that for all $V\subset U$, the connected component of $\cap_{n\in \mathbb{Z}} f^{-n}(\overline{V})$ containing $0$ intersects the boundary of  $V$. Let $K_0$ be the connected component of $\cap_{n\in \mathbb{Z}}f^{-n}(\overline{U})$ containing $0$. Then  for every local isotopy $I$ of $f$, and  for every invariant continuum $K\subset \overline{U}$ containing $0$, one has $\rho(I,K)=\rho(I,K_0)$.
\end{prop}

This proposition  implies that if $f:(W,0)\rightarrow (W',0)$ is an orientation preserving local homeomorphism at a non-accumulated indifferent point $0$, we can define the \emph{rotation number}  $\rho(I,0)$ for every local isotopy $I$ of $f$ at $0$, by writing $\rho(I,0)=\rho(I,K)$ where $K$ is a non-trivial invariant continuum  sufficiently close to $0$.


More generally, if $f:(W,z_0)\rightarrow (W',z_0)$ is an orientation preserving local homeomorphism at a non-accumulated indifferent point $z_0\in M$, we can conjugate it to a local homeomorphism at $0$, and get the previous definitions and results similarly.

\subsection{The local rotation set}\label{S: pre-local rotation set}

In this section, we will give a definition of the local rotation set and  will describe the relations between the rotation set and the rotation number. More details can be found in \cite{lerouxrotation}.

Let $f: (W,0)\rightarrow (W',0)$ be an orientation preserving local homeomorphism at $0\in\mathbb{R}^2$, and $I=(f_t)_{t\in[0,1]}$ be a local isotopy of $f$. Given two neighborhoods $V\subset U$ of $0$ and an integer $n\ge 1$, we define
$$E(U,V,n)=\{z\in U: z\notin V,f^{n}(z)\notin V, f^i(z)\in U \textrm{ for all } 1\le i\le n\}.$$
We define the \emph{rotation set} of $I$ relative to $U$ and $V$ by
$$\rho_{U,V}(I)=\cap_{m\ge 1}\overline{\cup_{n\ge m}\{\rho_n(z),z\in E(U,V,n)\}}\subset[-\infty,\infty],$$
where $\rho_n(z)$ is the average change of  angular coordinate along the trajectory of $z$. More precisely, let
\begin{eqnarray*}
\pi: \mathbb{R}\times (-\infty,0)&\rightarrow&\mathbb{C}\setminus\{0\}\simeq\mathbb{R}^2\setminus \{0\}\\
(\theta,y)&\mapsto&-ye^{i 2\pi \theta}
\end{eqnarray*}
be the universal covering projection,  $\widetilde{f}:\pi^{-1}(W)\rightarrow \pi^{-1}(W')$ be the lift of $f$ associated to $I$, and $p_1:\mathbb{R}\times (-\infty,0)\rightarrow \mathbb{R}$ be the projection onto the first factor. We define $$\rho_n(z)=\frac{p_1(\widetilde{f}^n(\widetilde{z})-\widetilde{z})}{n},$$
where $\widetilde{z}$ is any point in $\pi^{-1}\{z\}$.

We define the \emph{local rotation set} of $I$ to be
$$\rho_s(I,0)=\cap_U\overline{\cup_V\rho_{U,V}(I)}\subset[-\infty,\infty],$$
where $V\subset U\subset W$ are neighborhoods of $0$.


We say that $f$ can be \emph{blown-up} at $0$ if there exists an orientation preserving homeomorphism $\Phi: \mathbb{R}^2\setminus\{0\}\rightarrow \mathbb{T}^1\times(-\infty, 0)$, such that $\Phi f \Phi^{-1}$ can be extended continuously to $\mathbb{T}^1\times\{0\}$. We denote this extension by $h$.  Suppose that  $f$ is not conjugate a contraction  or an expansion. We define the \emph{blow-up rotation number} $\rho(f, 0)$ of $f$ at $0$ to be the Poincar\'e rotation number of $h|_{\mathbb{T}^1}$.  Let $I=(f_t)_{t\in[0,1]}$ be a local isotopy of $f$,  $(\widetilde{h}_t)$ be the natural lift of $\Phi|_{\mathbb{T}^1\times(0,r)}\circ f_t|_{D_r\setminus\{0\}}\circ \Phi^{-1}|_{\mathbb{T}^1\times(0,r)}$, where $D_r$ is a sufficiently small disk with radius $r$ and centered at $0$, and $\widetilde{h}$ be the lift of $h$ such that $\widetilde{h}=\widetilde{h}_1$ in a neighborhood of $\mathbb{R}\times\{0\}$. We define the \emph{blow-up rotation number $\rho(I,0)$ of $I$ at $0$} to be  the rotation number of $h|_{\mathbb{T}^1}$ associated to the lift $\widetilde{h}|_{\mathbb{R}\times\{0\}}$, which is a representative of $\rho(f, 0)$ on $\mathbb{R}$.  Jean-Marc Gambaudo, Le Calvez, and Elisabeth P\'ecou \cite{LecalvezthmNaishul} proved that  neither $\rho(f,0)$ nor $\rho(I,0)$   depend on the choice of $\Phi$, which generalizes a previous result of Na\u{\i}shul$'$ \cite{Naishul}.
In particular, if $f$ is a diffeomorphism, $f$ can be blown-up at $0$ and the extension of $f$ on $\mathbb{T}^1$  is induced by the map
 $$v\mapsto \frac{Df(0) v}{\|Df(0)v\|}$$
 on the space of unit tangent vectors.

More generally, if $f:(W,z_0)\rightarrow (W',z_0)$ is an orientation preserving local homeomorphism at $z_0$ that is not conjugate to the contraction or the expansion, we can give  the previous definitions for $f$ by conjugate it to an orientation preserving local homeomorphism at $0\in\mathbb{R}^2$.

The local rotation set can be empty. However, due to  Le Roux  \cite{lerouxparabolic}, we know that the rotation set is not empty if $f$ is area preserving. 
%
%

We say that $z$ is a \emph{contractible} fixed point of $f$ associated to a local isotopy $I=(f_t)_{t\in[0,1]}$ if the trajectory $t\mapsto f_t(z)$  of $z$ along $I$ is a loop homotopic to zero in $W\setminus\{z_0\}$.


The local rotation set satisfies the following properties:

\begin{prop}\cite{lerouxrotation} \label{P: pre-rotation set}
Let $f:(W,z_0)\rightarrow (W',z_0)$ be an orientation preserving local homeomorphism at $z_0$, and $I$ be a local isotopy of $f$ at $z_0$. One has the following results:
\begin{itemize}
\item[i)] For all integer $p,q$, $\rho_s(J^pI^q,z_0)=q\rho_s(I,z_0)+p$, where $J$ is a local isotopy of the identity such that $\rho(J,z_0)=1$.
\item[ii)] If $z_0$ is accumulated by  contractible  fixed points of $f$ associated to $I$, then $0\in \rho_s(I,z_0)$.
\item[iii)] If $\rho_s(I,z_0)$ is a non-empty set that is contained in $(0,+\infty]$ (resp. $[-\infty, 0)$), then $I$ has a positive (resp.negative) rotation type.
\item[iv)] If $f$ can be blown-up at $z_0$, and if $\rho_s(I,z_0)$  is not empty, then  $\rho_s(I,z_0)$ is  reduced to the single real number $\rho(I,z_0)$.
\item[v)] If $z_0$ is a non-accumulated indifferent point, $\rho_s(I,z_0)$ is  reduced to $\rho(I,z_0)$ (the rotation number defined in Section \ref{S: pre-prime-ends Compactification}).
\end{itemize}
\end{prop}


\begin{rmq}\label{R: pre-blow-up}
Le Roux also gives several criteria implying that $f$ can be blown-up at $z_0$.  The one we need in this article is  due to  B\'eguin, Crovisier and  Le Roux \cite{lerouxrotation}：

\textit{If there exists an arc $\gamma$ at $z_0$ whose germ is disjoint with the germs of $f^{n}(\gamma)$ for all $n\ne 0$, then $f$ can be blown-up at $z_0$.}

In particular, if there exists a leaf $\gamma^+$ from $z_0$ and a leaf $\gamma^-$ toward $z_0$ (we are in this case if $z_0$ is a petal, a saddle, or a mixed singularity of $\mathcal{F}$), we can choose a sector $U$ as in the picture.
\begin{figure}[h]
\centering
 \includegraphics[width=3cm]{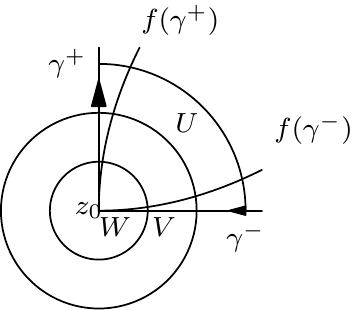}
\end{figure}
Let $V$ be a small neighborhood of $z_0$. There exists a neighborhood $W\subset V$ of $z_0$ such that
$$f(\overline{U}\cap W)\subset(\mathrm{Int}(U)\cap V)\cup\{z_0\}.$$
 So, the germs at $z_0$ of $f^{n}(\gamma^+)$  are pairwise disjoint, and hence $f$ can be blown-up at $z_0$. Moreover, $\rho(I,z_0)$ is equal to $0$ in this case.
\end{rmq}

 Le Roux also studied the dynamics near  a non-accumulated saddle point, and  proved the following result:

\begin{prop}\cite{lerouxrotation}\label{P: pre-rotation number of a saddle point}
If $z_0$ is a non-accumulated saddle point, then $f$ can be blown-up at $z_0$ and $\rho_s(I,z_0)$ is reduced to a rational number. Moreover, if $i(f,z_0)$ is equal to $1$, this rational number is not an integer.
\end{prop}

\subsection{Some generalizations of Poincar\'e-Birkhoff theorem}\label{S: pre-Poincare-Birkhoof}

In this section, we will introduce several generalizations of Poincar\'e-Birkhoff theorem.  An \emph{essential} loop in the annulus is a loop that is not homotopic to zero.

We first consider the homeomorphisms of closed annuli. Let $f$ be a homeomorphism of $\mathbb{T}^1\times[0,1]$ isotopic to  the identity, $I=(f_t)_{t\in[0,1]}$ be an identity isotopy of $f$. Let $\pi: \mathbb{R}\times[0,1]\rightarrow \mathbb{T}^1\times\mathbb{R}$ be the universal cover, $\widetilde{I}=(\widetilde{f}_t)_{t\in[0,1]}$ be the identity isotopy that lifts $I$,  $\widetilde{f}=\widetilde{f}_1$ be the lift of $f$ associated to $I$, and $p_1:\mathbb{R}^2\rightarrow \mathbb{R}$ be the projection on the first factor. The limits
$$\lim_{n\rightarrow \infty}\frac{p_1\circ\widetilde{f}^n(x,0)-x}{n}\quad \text{and } \lim_{n\rightarrow \infty}\frac{p_1\circ\widetilde{f}^n(x,1)-x}{n}$$
exists for all $x\in\mathbb{R}$, and do not depend on the choice of $x$. We  define the \emph{rotation number} of $f$ on  each boundary to be the corresponding limits. We define the \emph{rotation number} of  $z\in \mathbb{T}^1\times \mathbb{R}$ associated to  $I$ to be the limit
$$\lim_{n\rightarrow +\infty}\frac{p_1 (\widetilde{f}^n(\widetilde{z})-\widetilde{z})}{n}\in [-\infty,\infty],$$
 if this limit exists. We say that $f$ satisfies the \emph{intersection property} if $f\circ\gamma$ intersects $\gamma$, for every  simple essential loop $\gamma\subset\mathbb{T}^1\times(0,1)$. We have the following   generalizations of Poincar\'e-Birkhoff theorem:

\begin{prop} \cite{birkhoff}\label{P: pre-poincare-birkhoof-closed annulus}
Let $f$ be a homeomorphism of $\mathbb{T}^1\times[0,1]$ isotopic to  the identity and satisfying the intersection property. If the rotation number of $f$ on the two boundaries are different, then there exists a $q$-periodic orbit of rotation number $p/q$ for all irreducible  rational $p/q\in(\rho_1, \rho_2)$, where $\rho_1$ and $\rho_2$ are the rotation numbers of $f$ on the boundaries.
\end{prop}

\bigskip

We also consider  homeomorphisms of open annuli. Let $f:\mathbb{T}^1\times \mathbb{R}\rightarrow\mathbb{T}^1\times\mathbb{R}$ be a homeomorphism  isotopic to the identity, and $I=(f_t)_{t\in[0,1]}$ be an identity isotopy of $f$. Let $\pi: \mathbb{R}\times\mathbb{R}\rightarrow \mathbb{T}^2\times\mathbb{R}$ be the universal cover, $\widetilde{I}=(\widetilde{f}_t)_{t\in[0,1]}$ be the identity isotopy that lifts $I$,  $\widetilde{f}=\widetilde{f}_1$ be the lift of $f$ associated to $I$, and $p_1:\mathbb{R}^2\rightarrow \mathbb{R}$ be the projection on the first factor. Similarly, we define the \emph{rotation number} of a recurrent point $z\in \mathbb{T}^1\times \mathbb{R}$ associated to  $I$ to be the limit
$$\lim_{n\rightarrow +\infty}\frac{p_1 (\widetilde{f}^n(\widetilde{z})-\widetilde{z})}{n}\in [-\infty,\infty],$$
 if this limit exists. We say that $f$ satisfies the \emph{intersection property} if $f\circ\gamma$ intersects $\gamma$, for every  simple essential loop $\gamma\subset\mathbb{T}^1\times\mathbb{R}$.  Then, we have the following  generalization of Poincar\'e-Birkhoff theorem:
\begin{prop}[\cite{Franks}, \cite{lecalvezfeuilletage}]\label{P: pre-Le Calvez's generalization of poincare-birkhoof}
Let $f:\mathbb{T}^1\times\mathbb{R}\rightarrow\mathbb{T}^1\times\mathbb{R}$ be a homeomorphism isotopic to  the identity and satisfying the intersection property. If there exist two recurrent points of rotation numbers $\rho_1,\rho_2\in [-\infty,+\infty]$ respectively such that $\rho_1<\rho_2$, then there exists a $q$-periodic orbit of rotation number $p/q$ for all irreducible  rational $p/q\in(\rho_1, \rho_2)$.
\end{prop}

\begin{rmq}
 The result is also true for  area preserving homeomorphisms of the closed or half closed annulus by considering a symmetry.
\end{rmq}

\subsection{Topologically monotone periodic orbits for annulus homeomorphisms}\label{S: pre-topologically monotone periodic orbits}

In this section, we will recall the braid type of a periodic orbit and  the existence of the topologically monotone periodic orbits for  annulus homeomorphisms under some conditions. More details can be found in \cite{Boyland}.

 Denote by $\mathbb{A}$ the closed annulus $\mathbb{T}^1\times[0,1]$.  Let $f$ be a homeomorphism of ${\mathbb{A}}$ that preserves the orientation and each boundary circle, and $\widetilde{f}$ be a lift of $f$ to the universal cover $\widetilde{\mathbb{A}}=\mathbb{R}\times[0,1]$. Given $\widetilde{z}\in \widetilde{\mathbb{A}}$,  we define its rotation number under $\widetilde{f}$ as
 $$\rho(\widetilde{z},\widetilde{f})=\lim_{n\rightarrow\infty}\frac{p_1(\widetilde{f}^n(\widetilde{z}))-p_1(\widetilde{z})}{n},$$
if this limit exists, where $p_1$ is the projection onto the first factor. We define the rotation set of $\widetilde{f}$ to be
$$\rho(\widetilde{f})=\{\rho(\widetilde{z},\widetilde{f}), \widetilde{z}\in\widetilde{A}\}.$$
In particular, if $I$ is an identity isotopy of $f$ and $\widetilde{f}$ is the lift of $f$ associated to $I$, this definition of the rotation number coincides with the definition of the rotation number in Section \ref{S: pre-Poincare-Birkhoof}.

 Fix a copy of the closed annulus minus $n$ interior points, and denote it by ${\mathbb{A}}_n$. Let $G_n$ be the group of isotopy classes of orientation preserving homeomorphism of $\mathbb{A}_n$.  If $O$ is an $n$-periodic orbit of $f$ in the interior of ${\mathbb{A}}$, then there is an orientation preserving homeomorphism $h:{\mathbb{A}}\setminus O\rightarrow {\mathbb{A}}_n$. Philip Boyland defined the \emph{braid type} $bt(O,f)$ to be the isotopy class of $h\circ f|_{{\mathbb{A}}\setminus O}\circ h^{-1}$ in $G_n$, this isotopy class is independent of the choice of $h$. If $O$ is an $n$-periodic orbit of $f$ contained in a boundary circle of ${\mathbb{A}}$, he extends $f$ near this boundary and gets a homeomorphism $\overline{f}$ also on a closed annulus. Then $O$ is in the interior of this new annulus. The braid type $bt(O,\overline{f})$ is independent of the choice of the extension, and Boyland defined $bt(O,f)=bt(O,\overline{f})$.

Let $p/q$ be an irreducible  positive rational, and  $\widetilde{T}_{p/q}$ be the homeomorphism of $\widetilde{\mathbb{A}}$ defined by $(x,y)\mapsto(x+p/q,y)$. It descends to a homeomorphism $T_{p/q}$ of $\mathbb{A}$. We denote by $\alpha_{p/q}$ the braid type $bt(O,T_{p/q})$, where $O$ is any periodic orbit of $T_{p/q}$. We say that a $q$-periodic orbit $O$ of $f$ is a $(p,q)$-periodic orbit if $\rho(\widetilde{z},\widetilde{f})=p/q$ for  any  $\widetilde{z}$  in the  lift of $O$. We say that a $(p,q)$-periodic orbit $O$ is \emph{topologically monotone} if $bt(O,f)=\alpha_{p/q}$.
We define the \emph{Farey interval} $I(p/q)$ of $p/q$ to be the closed interval
$$[\max\{r/s: r/s<p/q, 0<s<q, \text{ and } (r,s)=1\}, \min\{r/s: r/s>p/q, 0<s<q, \text{ and } (r,s)=1\}].$$
 In particular, the Farey interval of $1/q$ is equal to $[0,1/(q-1)]$.

Boyland proved the following result:

\begin{prop}[\cite{Boyland}]\label{P: pre-boyland closed annulus}
If $f$ is an orientation and boundary preserving homeomorphism of the closed annulus, and $p/q\in\rho(\widetilde{f})$ is an irreducible positive rational, then $f$ has a $(p,q)$-topologically monotone periodic orbit. If $f$ has a $(p,q)$-orbit that is not topologically monotone, then $I(p/q)\subset \rho(\widetilde{f})$.
\end{prop}

\subsection{Annulus covering projection}\label{S: pre-annulus covering projection}

Let $M$ be an oriented surface, $X_0\subset M$ be a closed set, and $z_0\in M\setminus X_0$.  Denote by $M_0$ the connected component of $M\setminus X_0$ containing $z_0$.  Let $V\subset U\subset M_0$ be two small Jordan domains containing $z_0$. Write $\dot{U}=U\setminus \{z_0\}$ and $\dot{V}=V\setminus\{z_0\}$. Fix $z_1\in \dot{V}$. Let $\gamma\subset \dot{V}$ be a simple loop at $z_1$ such that the homotopic class $[\gamma]$ of $\gamma$ in $\dot{V}$ generates $\pi_1(\dot{V}, z_1)$. Let $i:\dot{U}\rightarrow M_0\setminus\{z_0\}$ be the inclusion, then $i_{*}\pi_1(\dot{U},z_1)$ is a subgroup of $\pi_1(M_0\setminus\{z_0\},z_1)$. Then, there exists a covering projection $\pi:(\widetilde{M},\widetilde{z}_1)\rightarrow (M_0\setminus\{z_0\},z_1)$ such that $\pi_{*}\pi_1(\widetilde{M},\widetilde{z}_1)=i_{*}\pi_1(\dot{U},z_1)$ by Theorem 2.13 in \cite{spanier}. Moreover, the fundamental group of $\widetilde{M}$ is isomorphic to $\mathbb{Z}$, so $\widetilde{M}$ is an annulus.

Let $\widetilde{\dot{U}}$ be the component of $\pi^{-1}(\dot{U})$ containing $\widetilde{z}_1$. Then  $\pi_{*}\pi_1(\widetilde{\dot{U}},\widetilde{z}_1)=\pi_1(\dot{U},z_1)$ and the restriction of $\pi$ to $\widetilde{\dot{U}}$ is a homeomorphism between $\widetilde{\dot{U}}$ and $\dot{U}$ by Theorem 2.9 in \cite{spanier}. Consider the ideal-ends compactification of $\widetilde{M}$, and denote by $\star$ the end in $\widetilde{\dot{U}}$. Then $\pi|_{\widetilde{\dot{U}}}$ can be extended continuously to a homeomorphism between $\widetilde{\dot{U}}\cup\{\star\}$ and $U$. We denote it by $h$.

\bigskip

If $f$ is an orientation preserving homeomorphism of  $M_0$, and  $z_0$ is a fixed point of $f$. By choosing sufficiently small $V$, we can suppose that $f(V)\subset U$. We know that $(f\circ\pi)_{*}\pi_1(\widetilde{M},\widetilde{z}_1)=i_{*}\pi_1(\dot{U},f(z_1))=\pi_{*}\pi_1(\widetilde{M},h^{-1}(f(z_1)))$, then  we deduce by Theorem 2.5 of \cite{spanier} that there is a lift $\widetilde{f}$ of $f$ to $\widetilde{M}$ that sends $\widetilde{z}_1$ to $h^{-1}(f(z_1))$. This map $\widetilde{f}$ is an homeomorphism because $\widetilde{f}_{*}\pi_1(\widetilde{M},\widetilde{z}_1)=\pi_1(\widetilde{M},h^{-1}(f(z_1))$ (see Corollary 2.7 in \cite{spanier}). Moreover, $\widetilde{f}$ can be extend continuously to a homeomorphism of $\widetilde{M}\cup\{\star\}$ that fixes $\star$.

In particular, if $f$ is isotopic to the identity, and if $I=(f_t)_{t\in[0,1]}$ is an identity isotopy of $f$ fixing $z_0$, then  there exists a lift $\widetilde{f}_{\cdot}( \cdot ): I\times\widetilde{M}\rightarrow \widetilde{M}$ of the continuous map $(t,\widetilde{z})\mapsto f_t(\pi(\widetilde{z}))$ such that $\widetilde{f}_0$ is equal to the identity, because $\pi$ is a covering projection. Moreover, by choosing $V$ small enough, we know that $\widetilde{f}_t|_{\widetilde{\dot{V}}}$ is conjugate to $f_t|_{\dot{V}}$ for $t\in[0,1]$, where $\widetilde{\dot{V}}$ is the component of $\pi^{-1}(\dot{V})$ containing $\widetilde{z}_1$. Then $(\widetilde{f}_{t})_{*}\pi_1(\widetilde{M},\widetilde{z}_1)=\pi_1(\widetilde{M},h^{-1}(f_t(z_1))$, therefore $\widetilde{f}_t$ is a homeomorphism by Corollary 2.7 in \cite{spanier}. We have indeed lifted  $I$ to an identity isotopy $\widetilde{I}=(\widetilde{f}_t)_{t\in[0,1]}$. Moreover, $\widetilde{f}_t$ can be extended continuously to a homeomorphism of $\widetilde{M}\cup\{\star\}$ that fixes $\star$, and we get an isotopy on $\widetilde{M}\cup\{\star\}$ that fixes $\star$. We still denote by $\widetilde{f}_t$ the homeomorphism of $\widetilde{M}\cup\{\star\}$ and  by $\widetilde{I}$ the identity isotopy on $\widetilde{M}\cup\{\star\}$ when there is no ambiguity. We call $\widetilde{I}$ the natural lift of $I$ to $\widetilde{M}\cup\{\star\}$, and $\widetilde{f}=\widetilde{f}_1$ the lift of $f$ to $\widetilde{M}\cup\{\star\}$ associated to $I$.

 Moreover, if $I$ is a maximal isotopy, $\widetilde{f}$ has no contractible fixed point associated to $\widetilde{I}$ on $\widetilde{M}$ and $\widetilde{I}$ is also a maximal isotopy. Recall that $\pi_{*}\pi_1(\widetilde{M},\widetilde{z}_1)=i_{*}\pi_1(\dot{U},z_1)$. So, $\pi(O)$ is a periodic orbit of type $(p,q)$ associated to $I$ at $z_0$ for  all periodic orbit $O$ of type $(p,q)$ associated to $\widetilde{I}$ at $\star$, where $\frac{p}{q}\in\mathbb{Q}$ is irreducible.

\bigskip

 Let $\mathcal{F}$ be an oriented foliation on $M_0$ such that $z_0$ is a sink (resp. source). Then there exists a lift $\widetilde{\mathcal{F}}$ of $\mathcal{F}|_{M_0\setminus\{z_0\}}$ to $\widetilde{M}$, and $\star$ is a sink (resp. source) of $\widetilde{\mathcal{F}}$.  Denote by  $W$   the attracting (resp. repelling) basin of $z_0$ for  $\mathcal{F}$, and by $\widetilde{W}$ the attracting (resp. repelling) basin of $\star$ for $\widetilde{\mathcal{F}}$. Write $\dot{W}=W\setminus\{z_0\}$,  $\dot{\widetilde{W}}=\widetilde{W}\setminus\{\star\}$. Let $\widetilde{z}_1\in\dot{\widetilde{W}}$ be a point sufficient close to $\star$.  Then $(\pi|_{\dot{\widetilde{W}}})_{*}\pi_1(\dot{\widetilde{W}},\widetilde{z}_1)=\pi_1(\dot{W},\pi(z_1))$, and hence $\pi|_{\dot{\widetilde{W}}}$ is a homeomorphism between $\dot{\widetilde{W}}$ and $\dot{W}$ by Corollary 2.7 in \cite{spanier}, and can be extended continuously to a homeomorphism between $\widetilde{W}$ and $W$.

\subsection{Extend  lifts of a homeomorphism to the boundary}\label{S: pre-extend lifts to the boundary}
In this section, let $M$ be a plane,  $f$ be an orientation preserving homeomorphism of $M$, and $X$ be a invariant, discrete, closed subset of $M$ with at least $2$ points.

We consider the Poincar\'e's disk model for the hyperbolic plane $H$, in which model, $H$ is identified with the interior of the unit disk and the geodesics are segments of Euclidean circles and straight lines that meet the boundary perpendicularly. A choice of hyperbolic structure on $M\setminus X$ provides an identification of the universal cover of $M\setminus X$ with $H$. A detailed description of the hyperbolic structures can be found in  \cite{Casson}. The compactification of the interior of the unit disk by the unit circle induces a compactification of $H$ by the circle $S_{\infty}$. Let $\pi: H\rightarrow M\setminus X$ be the universal cover. Then, $f|_{M\setminus X}$ can be lifted to homeomorphisms of $H$. Moreover, we have the following result:

\begin{prop}\cite{Handel}
Each lift $\widehat{f}$ of $f|_{M\setminus X}$ extends uniquely to a homeomorphism of $H\cup S_{\infty}$.
\end{prop}

 \begin{rmq} When $X$ has infinitely many points, Michael Handel gave a proof  in Section 3 of \cite{Handel}; when $X$ has finitely many points, the situation is easier and Handel's proof still works.
\end{rmq}

\bigskip

In particular, suppose that $z_0$ is an isolated point in $X$ and is a fixed point of $f$. Let  $\gamma$ be a sufficiently small circle near $z_0$ whose lifts to $H$ are horocycles. Fix one lift $\widehat{\gamma}$ of $\gamma$. Denote by $P$ the end point of $\widehat{\gamma}$ in $S_{\infty}$.    Fix $z_1\in\gamma$ and a lift $\widehat{z}_1$ of $z_1$ in $\widehat{\gamma}$. Let $\Gamma$ be the group of parabolic covering translations  that fix $\widehat{\gamma}$, and $T$ be the parabolic covering translations  that generates $\Gamma$. Then, $\pi$ descends to a annulus cover $\pi': (H/\Gamma, \widetilde{z}_1)\rightarrow (M\setminus X, z_1)$, where $\widetilde{z}_1=\{T^n(\widehat{z}_1):n\in\mathbb{Z}\}$. Also, $\widehat{z}\mapsto\{T^n(\widehat{z}):n\in\mathbb{Z}\}$ defines a universal cover $\pi'': H\rightarrow H/\Gamma$.

\begin{figure}[h]
    \begin{minipage}[t]{0.45\linewidth}
      \centering
   \includegraphics[width=4cm]{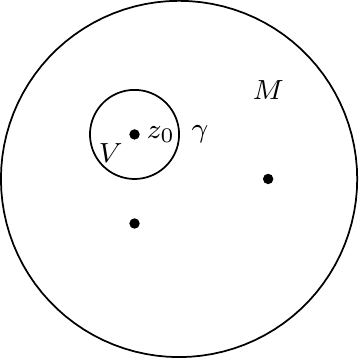}
    \end{minipage}
  \hspace{0.1\linewidth}
 \begin{minipage}[t]{0.45\linewidth}
      \centering
     \includegraphics[width=3.8cm]{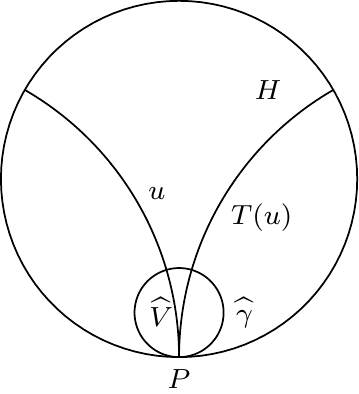}
    \end{minipage}
\end{figure}

Let $V$ be the disk containing $z_0$ and bounded by $\gamma$, $\widehat{V}$ be the disk bounded by $\widehat{\gamma}$ which is a component of $\pi^{-1}(V\setminus \{z_0\})$. We know that $\pi''(\widehat{V})$ is an annulus with $\pi''(\widehat{\gamma})$ as one of its boundary. We add a point $\star$ at the other end, and get a disk $ \displaystyle{\widetilde{V}=\pi''(\widehat{V})\cup\{\star\}}$.  As in the previous section, $\pi'|_{\pi''(\widehat{V})}$ extends continuously to a homeomorphism between $\displaystyle{\widetilde{V}}$ and $V$, and $f$ can be lifted to  a homeomorphism $\widetilde{f}$ of $H/\Gamma\cup\{\star\}$ fixing $\star$. Let $\widehat{f}$ be a lift of $\widetilde{f}|_{H/\Gamma}$ to $H$, it is also a lift of $f|_{M\setminus X}$ and satisfies $\widehat{f}\circ T=T\circ\widehat{f}$. Moreover, both $\widehat{f}$  and $T$ extend continuously to  homeomorphisms of $H\cup S_{\infty}$ fixing $P$. We denote still by $\widehat{f}$  and $T$ the two extensions respectively. The formula $\widehat{f}\circ T=T\circ\widehat{f}$ is still satisfied. So, $\widehat{f}|_{H\cup S_{\infty}\setminus\{P\}}$ descends to a homeomorphism of $(H\cup S_{\infty}\setminus\{P\})/\Gamma$. Because $(H\cup S_{\infty}\setminus\{P\})/\Gamma$ is homeomorphic to a compactification of $H/\Gamma\cup \{\star\}$ by adding a circle at infinity $S_{\infty}$, one knows that $\widetilde{f}$ extends continuously to a homeomorphism of $H/\Gamma\cup\{\star\}\cup S_{\infty}$.

\section{Proof of the main theorem}\label{S: chapter3-homeomorphism case}

Let $M$ be an oriented surface,  $f:M\rightarrow M$ be an area preserving homeomorphism of $M$ isotopic to the identity, and $z_0$ be an isolated fixed point of $f$ such that $i(f,z_0)=1$.  Let $I$ be an identity isotopy of $f$ fixing $z_0$ such that its rotation set, which was defined in section \ref{S: pre-local rotation set}, is reduced to an integer $k$.

We say that the property \textbf{P)} holds for $(f,I,z_0)$ if there exists $\varepsilon >0$, such that either for all irreducible  $p/q\in(k,k+\varepsilon)$, or for all irreducible  $p/q\in(k-\varepsilon, k)$, there exists a contractible periodic orbit $O_{p/q}$ of type $(p,q)$ associated to $I$ at $z_0$, such that  $\mu_{O_{p/q}}\rightarrow \delta_{z_0}$ as $p/q\rightarrow k$, in the weak-star topology, where $\mu_{O_{p/q}}$ is the invariant probability measure  supported on $O_{p/q}$,

Our aim of this section is to prove the following result:
\begin{thm}[Theorem \ref{T: main-index 1 and rotation 0 implies  accumulated by periodic points}]
Under the previous assumptions, if one of the following situations occurs,
\begin{itemize}
\item[a)]
$M$ is a plane, $f$ has only one fixed point $z_0$, and  has another periodic orbit besides $z_0$;
 \item[b)] the total area of $M$ is finite,
 \end{itemize}
then the property \textbf{P)} holds for $(f,I,z_0)$.
\end{thm}

\begin{rmq}\label{R: asume positive}
Let $I'$ be a local isotopy of $f$ at $z_0$ such that $\rho_s(I',z_0)$ is reduced to $0$. Since $f$ is area preserving and $i(f,z_0)=1$, by Proposition \ref{P: pre-rotation type is unique}, $I'$ has either a positive or a negative rotation type. Let $\mathcal{F}'$ be a locally transverse foliation of $I'$.  If $I'$ has a positive rotation type,  then $z_0$ is a sink of $\mathcal{F}'$ and  the interval in Property \textbf{P)} is  $(k,k+\varepsilon)$;   if $I'$ has a negative rotation type, then $z_0$ is a source and  the interval in Property \textbf{P)} is $(k-\varepsilon, k)$.

We suppose that $I'$ has a positive rotation type in  this section, the other case can be treated similarly.
\end{rmq}

\begin{rmq}\label{R: nonaccumulated and rotation 0 means indefferent}
 If $z_0$ is not accumulated by periodic orbits, since the rotation set is reduced to  an integer and $i(f,z_0)=1$, $z_0$ is an indifferent fixed  point  by Proposition \ref{P: pre-rotation number of a saddle point}. Then, by the assertion v) of Proposition \ref{P: pre-rotation set}, one deduces that $\rho(I,z_0)$ is equal to this integer.
\end{rmq}

 We will prove the theorem in several cases.

\subsection{The case where $M$ is a plane}\label{S: M is a plane}

In this section, we suppose that $M$ is a plane, and that  $I$ is a maximal identity isotopy of $f$ such  that $\mathrm{Fix}(I)$ is reduced to $z_0$.  We will prove the following result in this section and get the proof of the first part of Theorem \ref{T: main-index 1 and rotation 0 implies  accumulated by periodic points} as a corollary.

\begin{thm}\label{T: chapter3-plane case}
 Under the previous assumption, if $\rho_s(I,z_0)$ is reduced to $0$, and if $f$ has another periodic orbit besides $z_0$, then the property \textbf{P)} holds for $(f,I,z_0)$.
\end{thm}

This result is an important one in the proof of Theorem \ref{T: main-index 1 and rotation 0 implies  accumulated by periodic points}. In the latter cases, we will always reduce the problem to this case and get the result as a corollary.
Before proving this result, we first prove the first case of  Theorem \ref{T: main-index 1 and rotation 0 implies  accumulated by periodic points} as a corollary.

\begin{proof}[Proof of the first case of  Theorem \ref{T: main-index 1 and rotation 0 implies  accumulated by periodic points}]
 We only need to deal with the  case where $\rho_s(I,z_0)$ is reduced to a non-zero integer $k$. Let $J$ be an identity isotopy of the identity fixing $z_0$ such that the blow-up rotation number $\rho(J,z_0)$ is equal to $1$. Write $I'=J^{-k}I$. By the first assertion of Proposition \ref{P: pre-rotation set}, $\rho_s(I',z_0)$ is reduced to $0$.   Since $f$ has exactly one fixed point, $I'$  is maximal and the property \textbf{P)} holds for $(f, I', z_0)$. A periodic orbit in the annulus $M\setminus \{z_0\}$ with rotation number $p/q$ associated to $I'$ is  a periodic orbit with rotation number $k+p/q$ associated to $I$. Therefore,  the property \textbf{P)} holds for $(f, I, z_0)$.
\end{proof}

Now we begin the proof of Theorem \ref{T: chapter3-plane case} by some lemmas.

\begin{lem}\label{L: compact to converge}
Let $g$ be a homeomorphism of $\mathbb{R}^2$,  $I'$ be a maximal identity isotopy of $g$ that fixes $z_0$, and $\mathcal{F}'$ be a transverse foliation of $I'$.  Suppose that $z_0$ is an isolated fixed point of $g$ and a sink of $\mathcal{F}'$. Let  $W'$ be the attacting basin of $z_0$ for  $\mathcal{F}'$. Suppose that either $W'$ is equal to $\mathbb{R}^2$ or  $W'$ is a proper subset of $\mathbb{R}^2$ whose boundary is the union of some proper leaves of $\mathcal{F}'$. Let  $U$ be  a Jordan domain  containing $z_0$  that   satisfies  $U\subset W'$ and  $g(U)\subset W'$.

 If there exist a compact subset $K\subset U$ and  $\varepsilon>0$ such that $K$ contains a $q$-periodic orbit $O_{p/q}$ with  rotation number $p/q$ in the annulus $\mathbb{R}^2\setminus\{z_0\}$ for all irreducible $p/q\in(0,\varepsilon)$, then $\mu_{O_{p/q}}$ converges, in the weak-star topology, to the Dirac measure $\delta_{z_0}$ as $p/q\rightarrow 0$, where $\mu_{O_{p/q}}$ is  the invariant probability measure supported on $O_{p/q}$.
\end{lem}

\begin{proof}
We only need to prove that for every continuous  function $\varphi:W'\rightarrow \mathbb{R}$, for every $\eta>0$, there exists $\delta >0$, such that for every $q-$periodic orbit $O\subset K$ with irreducible rotation number $p/q<\delta$, we have
$$|\int\varphi d\mu_O-\varphi(z_0)|<\eta,$$
where $\mu_O$ is the invariant probability measure supported on $O$.

Let $V$ be a neighborhood of $z_0$ such that $|\varphi(z)-\varphi(z_0)|<\eta/2$ for all $z\in V$.  Let $\pi:\mathbb{R}\times (-\infty,0)\rightarrow W'\setminus\{z_0\}$ be the universal cover which sends the vertical lines upward to the leaves of $\mathcal{F}'$, and $p_1:\mathbb{R}\times(-\infty,0)\rightarrow \mathbb{R}$ be the projection onto the first coordinate. Let $\widetilde{U}=\pi^{-1}(U\setminus\{z_0\})$, $\widetilde{K}=\pi^{-1}(K\setminus\{z_0\})$     and $\widetilde{g}$ be the lift of $g$ to $\widetilde{U}$ associated to $I'$.  By the assumptions about $W'$, we know that   any arc that is positively transverse to $\mathcal{F}'$ cannot come back into $W'$ once it leaves $W'$. So
$$p_1(\widetilde{g}(z))-p_1(z)>0, \textrm{ for all } z\in\widetilde{K}.$$
Therefore there exists   $\eta_1>0$  such that for all $z\in\pi^{-1}(K\setminus V)$, one has
$$p_1(\widetilde{g}(z))-p_1(z)>\eta_1.$$
One deduces that for all $q-$periodic orbit $O\subset K$ with irreducible rotation number   $p/q<\delta = \frac{\eta\eta_1}{4|\sup_K\varphi|}$,
$$\frac{\#(O\setminus V)\eta_1}{q}<\frac{p}{q},$$
hence,
$$|\int\varphi d\mu_O-\varphi(z_0)|< \eta/2+2\sup_K|\varphi|\frac{\#(O\setminus V)}{q}<\eta.$$
\end{proof}

\begin{rmq}
 In  this lemma, the homeomorphism $g$ do not need to be area preserving. The assumptions about $W'$ prohibit the following bad situation:
\end{rmq}
\begin{figure}[h]
      \centering
   \includegraphics{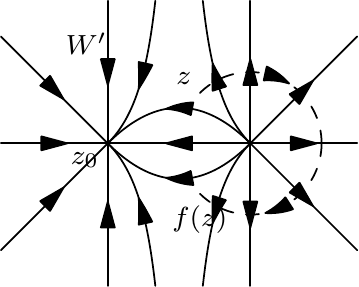}
 \end{figure}

\begin{lem}\label{L: chapter3-mainproof1-plane1}
If  $\rho_s(I, z_0)$ is reduced to $0$, and if $f$ can be blown-up at $\infty$ such that the blow-up rotation number at $\infty$, that is defined in Section \ref{S: pre-local rotation set}, is different from $0$,   then the property \textbf{P)} holds for $(f, I, z_0)$.
\end{lem}
In order to to prove this lemma, we need the following sublemma:
\begin{sublem}\label{L: chapter3-mainproof1-plane1-sublemma}
Under the conditions of the previous  Lemma,   $\rho(I,\infty)$ is negative, and there exists $\varepsilon>0$ such that for all irreducible $p/q\in(0,\varepsilon)$, there exists a $q$-periodic orbit with rotation number $p/q$ in the annulus $M\setminus\{z_0\}$.
\end{sublem}
\begin{proof}
We consider a transverse foliation $\mathcal{F}$ of $I$. It has exactly two singularities $z_0$ and $\infty$. Since $f$ is area preserving,  $f|_{M\setminus\{z_0\}}$ satisfies the intersection property, and using the remark that follows Proposition \ref{P: pre-dynamics of a foliation}, one can deduce that $\mathcal{F}$ does not have any closed leaf. Because $f$ can be blown-up at $\infty$ and the blow-up rotation number $\rho(I,\infty)$ is different from $0$,  we deduce that $\infty$ is either a sink or a source. By the assumption in Remark \ref{R: asume positive},  $z_0$ is a sink of $\mathcal{F}$,  so $\infty$ is a source of $\mathcal{F}$, and hence $\rho(I,\infty)$ is smaller than $0$. Write $\rho=-\rho(I,\infty)$. We denote by $S_{\infty}$ the circle  added at $\infty$ when blowing-up $f$ at $\infty$, and still by $f$ the extension of $f$ to $M\sqcup S_{\infty}$. One has to consider the following two cases:
\begin{itemize}
\item[-]Suppose that $z_0$ is accumulated by periodic orbits. Let $z_1$ be a periodic point of $f$ in the annulus $M\setminus\{z_0\}$. Its rotation number is strictly positive. We denote by $\varepsilon$ this number. Because the rotation set $\rho_s(I,z_0)$ is equal to $0$, the rotation number of a periodic orbit tends to $0$ as the periodic orbit tend to $z_0$.   Hence for all irreducible $p/q\in(0,\varepsilon)$, there exists a periodic orbit near $z_0$ with rotation number $r/s \in(0,p/q)$.  The restriction of the homeomorphism $f$  to the annulus $M\setminus \{z_0\}$ satisfies the intersection property,  then by  Proposition \ref{P: pre-Le Calvez's generalization of poincare-birkhoof}, there exists a $q$-periodic orbit with rotation number $p/q$ in the annulus for all irreducible $p/q\in(0,\varepsilon)$.

\item[-]Suppose that  $z_0$ is not accumulated by periodic orbits. Then,  $z_0$ is an indifferent fixed point by Proposition \ref{P: pre-rotation number of a saddle point}, and $\rho(I,z_0)$, which was defined in Section \ref{S: pre-prime-ends Compactification}, is equal to $0$. Let $K_0$ be a small enough invariant continuum at $z_0$ such that $\rho(I,K_0)=0$ ( see Section \ref{S: pre-prime-ends Compactification}). We denote by $(M\setminus K_0)\sqcup \mathbb{T}^1\sqcup S_{\infty}$ the prime-ends compactification at the ends $K_0$ and the compactification at $\infty$, which is an annulus. We can extend $f$ to both boundaries and get a homeomorphism of the closed annulus satisfying the intersection condition. Moreover, the rotation number of $f$ on the  upper boundary $\mathbb{T}^1$ is equal to $0$, and on the lower boundary $S_{\infty}$ is equal to  $\rho$. So, by  Proposition \ref{P: pre-poincare-birkhoof-closed annulus}, for all irreducible $p/q$ between $0$ and $\rho$, there exists a periodic orbit in the annulus with rotation number $p/q$.
    \end{itemize}
\end{proof}

\begin{rmq}
In the first case of the proof, it is natural to think that we can prove by a generalization of Poincar\'e-Birkhoff theorem that there exists a periodic orbit in the annulus with rotation number $p/q$ for all irreducible $p/q$ between $0$ and $\rho$. But in fact, the annulus in this case is half-open, and   we do not know whether there exists such a generalization of Poincar\'e-Birkhoff theorem. So, we choose another periodic orbit to avoid treating the half-open annulus.
\end{rmq}

\begin{proof}[Proof of Lemma \ref{L: chapter3-mainproof1-plane1}]
Paste two copies of the closed disk by $S_{\infty}$. We get a sphere $S$ and a  homeomorphism $f'$ that equals to $f$ on each copy and has two fixed points $z_0$ and $\sigma(z_0)$, where $\sigma$ is the natural involution. Let $I'$ be an identity isotopy that fixes $z_0$ and $\sigma(z_0)$ and satisfies  $\rho_s(I', z_0)=\{0\}$. Because $I$ is a maximal isotopy,  $f|_{M\setminus\{z_0\}}$ has no contractible fixed point associated to $I$. Because the blow-up rotation number $\rho(I,\infty)$ is different from $0$, the extension of $f$ to $S_{\infty}$ does not have any  fixed point with rotation number $0$ (associated to $I$). So, $f'|_{S\setminus\{z_0,\sigma(z_0)\}}$ has no contractible fixed point associated to $I'|_{S\setminus\{z_0,\sigma(z_0)\}}$. Therefore, $I'$ is a maximal isotopy, and  one knows $\mathrm{Fix}(I')=\{z_0,\sigma(z_0)\}$. Let $\mathcal{F}'$ be a transverse foliation of $I'$. Then $\mathcal{F}'$ has exactly two singularities $z_0$ and $\sigma(z_0)$. By the assumption, $z_0$ is a sink of $\mathcal{F}'$. Since the involution $\sigma$ is orientation reversing, $\rho(I',\sigma(z_0))=0$ and $I'$ has a negative rotation type at $\sigma(z_0)$. So $\sigma(z_0)$ is a source of $\mathcal{F}'$, and hence  $\mathcal{F}'$ does not have any petal. One has to consider the following two cases:

\begin{itemize}
\item[-]Suppose that all the leaves of $\mathcal{F}'$ are curves from $\sigma(z_0)$ to $z_0$. The compact set $M\sqcup S_{\infty}$ satisfies the conditions of Lemma \ref{L: compact to converge}, and we can deduce the result.
\begin{figure}[h]
      \centering
   \includegraphics{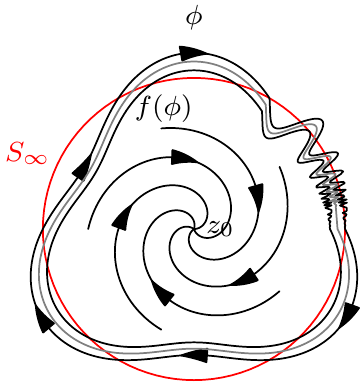}
   \end{figure}
\item[-]Suppose that  there exists a closed leaf in $\mathcal{F}'$. Since $f$ is area preserving,  similarly to the remark that follows Proposition \ref{P: pre-dynamics of a foliation}, one can deduce that there does not exist any closed  leaf in $M$ or in $\sigma(M)$.    So,  each closed leaf intersects $S_{\infty}$. Let $W'$ be the basin of $z_0$ for  $\mathcal{F}'$. Then $\partial W'$ is  a closed leaf, and hence  intersects $S_{\infty}$. Denote this leaf by $\phi$. We suppose  that $z_0$ is to the right of $\phi$, the other case can be treated similarly.
Denote by $R(\phi)$ (resp. $L(\phi)$) the component of $S\setminus \phi$ to the right (resp. left) of $\phi$. Since $f'(\phi)$ is included in $R(\phi)$, we know that both $R(\phi)\cap M$ and $R(\phi)\setminus (M\cup S_{\infty})$ are not empty.
 Choose a suitable essential curve $\displaystyle{\Gamma\subset \left(R(\phi)\cap L(f(\phi))\right)\subset W'}$ that transversely intersects $S_{\infty}$ at only finitely many points (see the gray curve between $\phi$ and $f(\phi)$ in the picture). Then, $(L(\Gamma)\cap M)$ has finitely many connected components, and so does $(L(\Gamma)\cap (M\cup S_{\infty}))$. Moreover, each component of  $(L(\Gamma)\cap (M\cup S_{\infty}))$ contains a segment of $S_{\infty}$.

 Since both $M$ and $S_{\infty}$ are invariant by $f'$, one knows that $f'^{-1}(L(\Gamma)\cap (M\cup S_{\infty}))$ is included in $L(\Gamma)\cap (M\cup S_{\infty})$. So, if $V$ is a component of  $(L(\Gamma)\cap (M\cup S_{\infty}))$,  there exists $n>0$ such that $f'^{-n}(V)\subset V$. Moreover, one knows that $f'^{-n}(V\cap S_{\infty})\subset V\cap S_{\infty}$ and that the rotation number of each point in $S_{\infty}$ is equal to $\rho$, so there exists $m>0$ such that $\rho=m/n$ and the rotation number of every periodic point of $f'$ in $V$ is equal to $\rho$.
  Therefore, the rotation number of every periodic point $z\in(L(\Gamma)\cap M)$  of $f'$ is equal to the rotation number of $S_{\infty}$.
So, all the periodic orbits in the annulus $M\setminus\{z_0\}$ with rotation number in $(0,\rho)$ is contained in $R(\Gamma)\cap M$. We find  a compact set $\overline{R(\Gamma)\cap M}$  that satisfies the conditions of Lemma \ref{L: compact to converge}, and can deduce the result.
\end{itemize}
\end{proof}

\begin{lem}
If $\rho_s(I,z_0)$ is reduced to $0$, and if $O\subset M\setminus \{z_0\}$ is a periodic orbit of $f$, then the rotation number of $O$ (associated to I) is positive.
\end{lem}

\begin{proof}
Let $\mathcal{F}$ be a transverse foliation of $I$. Then $\mathcal{F}$ has only one singularity $z_0$, and $z_0$ is a sink of $\mathcal{F}$ by the assumption in Remark \ref{R: asume positive}. Since $f$ is area preserving, by the remark that follows \ref{P: pre-dynamics of a foliation} one knows that $\mathcal{F}$ does not have any closed leaf. Let $W$ be the attracting basin of $z_0$ for $\mathcal{F}$. It is either  $M$ or a  proper subset of $M$ whose boundary is the union of some proper leaves. In the first case, any periodic orbit of $f|_{M\setminus\{z_0\}}$ has a positive rotation number associated to $I$, and the proof is finished. In the second case,  note that each  connected component of $M\setminus \overline{W}$ is a disk foliated by proper leaves, and hence does not contain any loop that is transverse to $\mathcal{F}$. Moreover, any loop transverse to $\mathcal{F}$ can not meet a boundary leaf of $W$, and hence is contained in $W$. One deduces that every periodic orbit of $f$ distinct from $\{z_0\}$ is contained in  $W$, and its trajectory along the isotopy is homotopic to a transverse loop in $W$. So, its rotation number is positive.
\end{proof}

\begin{lem}
 If $\rho_s(I,z_0)$  is reduced to $0$, and if $f$ has
another periodic orbit besides $z_0$, then there exist an interger $q\ge 1$ and a $q$-periodic orbit $O$ with rotation number $1/q$ (associated to $I$) such that  $f|_{M\setminus O}$ is isotopic to a homeomorphism $R_{1/q}$ satisfying $R^q_{1/q}=\mathrm{Id}$.
\end{lem}

\begin{proof}
Let $O_0$ be a periodic orbit of $f$ distinct from $\{z_0\}$. By the previous lemma, the rotation number $\rho$ of $O_0$ in the annulus $M\setminus\{z_0\}$ associated to $I$ is positive. Similarly to the proof of Sublemma \ref{L: chapter3-mainproof1-plane1-sublemma},  there exists a $q$-periodic orbit with rotation number $p/q$ in the annulus $M\setminus\{z_0\}$ for all irreducible $p/q\in(0,\rho)$.  Let $\mathcal{F}$ be a transverse foliation of $I$. One knows by the assumption in Remark \ref{R: asume positive} that $z_0$ is a sink of $\mathcal{F}$. Let $W$ be the attracting basin of $z_0$ for $\mathcal{F}$. One has to consider the following three cases:
\begin{itemize}
\item[i)] Suppose that $W$ is equal to $M$.

 Let $T:(x,y)\mapsto (x+1,y)$ be the translation of $\mathbb{R}^2$. It induces a universal covering map
$ \pi: \mathbb{R}^2\rightarrow \mathbb{R}^2/T\simeq \mathbb{T}^1\times \mathbb{R}$. Let $h:M\setminus \{z_0\}\rightarrow \mathbb{T}^1\times\mathbb{R}$ be an orientation preserving map  that maps  the leaves of $\mathcal{F}$ to the vertical lines $\{\pi(\{x\}\times\mathbb{R}): x\in\mathbb{R}\}$ upward. Write $I'=(h\circ f_t\circ h^{-1})_{t\in[0,1]}$, and $f'=h\circ f\circ h^{-1}$. We will prove that there exists a positive integer $q$, and a $q$-periodic orbit $O$ of $f'$ with rotation number $1/q$ (associated to $I'$) such that  $f'|_{(\mathbb{T}^1\times\mathbb{R})\setminus O}$ is isotopic to a homeomorphism $R_{1/q}$ satisfying that $R^q_{1/q}=\mathrm{Id}$, and hence $h^{-1}(O)$ is a $q$-periodic orbit of $f$ with rotation number $1/q$ (associated to $I$) such that $f|_{M\setminus h^{-1}(O)}$ is isotopic to $h^{-1}\circ R_{1/q}\circ h$.

Fix a  $q$-periodic orbit $O$ of $f'$ with rotation number $1/q$ in the annulus $M\setminus\{z_0\}$ for  $1/q\in(0,\rho)$.  Choose  $0<M_1<M_2$ such that
$$O\subset \mathbb{T}^1\times(-M_1, M_1), \quad \text{and } (\bigcup_{t\in[0,1]}f'_t(\mathbb{T}^1\times[-M_1,M_1]))\subset \mathbb{T}^1\times(-M_2,M_2).$$
 Let $\widetilde{f}$ be the lift of $f'$ associated to $I'$. One knows that
    $$p_1(\widetilde{f}(\widetilde{z}))-p_1(\widetilde{z})>0 \quad \text{for all } \widetilde{z}\in\mathbb{R}^2, $$
where $p_1$ is the projection to the first factor. Let $\varphi_1$ be the  homeomorphism of $\mathbb{T}^1\times \mathbb{R}$ whose lift to $\mathbb{R}^2$ is defined by
$$\widetilde{\varphi}_1(x,y)=\left\{\begin{array}{ll}(x,y), \quad &\text{for } |y|\le M_2,\\
(x+|y|-M_2,y),\quad &\text{for } |y|>M_2.\end{array}\right.$$
We know that $\eta(y)=\sup_{x\in \mathbb{R}}|p_2(\widetilde{f'}(x,y))-y|$ is a continuous function, where $p_2$ is the projection onto the second factor. So,  there exist $M_3>M_2$ and a   homeomorphism $\varphi_2$ of $\mathbb{T}^1\times \mathbb{R}$ whose lift $\widetilde{\varphi}_2$  to $\mathbb{R}^2$  satisfies $p_1\circ\widetilde{\varphi}_2=\mathrm{Id}$ and
$$\widetilde{\varphi}_2(x,y)=\left\{\begin{array}{ll}(x,y), \quad &\text{for } |y|\le M_2,\\
(x,y+\text{sign}(y)(\eta(y)+1)),\quad &\text{for } |y|\ge M_3.\end{array}\right.$$
Let $f''=\varphi_2\circ\varphi_1\circ f'$. It is a contraction near each end and hence can be blown-up at each end by adding a circle. Moreover, by choosing suitable blow-up, the rotation numbers  at the boundary can be any real number, and we get a homeomorphism $\overline{f''}$ of closed annulus and a lift $\widetilde{f''}$ of $\overline{f''}$  such that $O$ is a $(1,q)$-periodic orbit and $\rho(\widetilde{f''})$ (see Section \ref{S: pre-topologically monotone periodic orbits} for the definition) is a closed interval in $(0,\infty)$.  One deduces by Proposition \ref{P: pre-boyland closed annulus} that $O$ is topologically monotone (Otherwise $I(1/q)=[0,1/(q-1)]\subset \rho(\widetilde{f''})$). Therefore,  $f''|_{(\mathbb{T}^1\times\mathbb{R})\setminus O}$  is isotopic to a homeomorphism $R_{1/q}$ satisfying $R^q_{1/q}=\mathrm{Id}$, and so is $f'|_{(\mathbb{T}^1\times\mathbb{R})\setminus O}$. The lemma is proved.

\item[ii)] Suppose that $W$ is a proper subset of M whose boundary is the union of some proper leaves, and that $z_0$ is not accumulated by periodic orbits.

    In this case, one knows by Remark \ref{R: pre-blow-up} that $f$ can be blown-up at $\infty$, and that the blow-up rotation number $\rho(I,\infty)$ is equal to $0$. One knows  by Remark \ref{R: nonaccumulated and rotation 0 means indefferent} that $z_0$ is an non-accumulated indifferent point, and that $\rho(I,z_0)$ is equal to $0$.

    Recall that there exists a $q$-periodic orbit with rotation number $p/q$ in the annulus $M\setminus\{z_0\}$ for all irreducible $p/q\in(0,\rho)$. We  fix a  $q$-periodic orbit $O$ of $f$ with rotation number $1/q$ in the annulus $M\setminus\{z_0\}$ for  $1/q\in(0,\rho)$. Let $\gamma_1$ be a simple closed curve that separates $O$ and $z_0$. Denote by $U^-$ the component of $M\setminus \gamma_1$ containing $O$.  We deduce by the assertions i) and ii) of Proposition \ref{P: pre-rotation set} that there exists a neighborhood  of $z_0$ that does not contain any $q$-periodic point of $f$ with rotation number $1/q$. So, by choosing $\gamma_1$ sufficiently close to $z_0$, we can suppose that all the $q$-periodic points of $f$ with rotation number $1/q$ are contained in $U^-$. Let $\gamma_2$ be a simple closed curve that separate $\gamma_1$ and $z_0$ such that $\cup_{t\in[0,1]}f_t(\overline{U^-})$ is in the component of $M\setminus \gamma_2$ containing $\gamma_1$. Denote by $U$ the component of $M\setminus \gamma_2$ containing $z_0$. Let  $V\subset U$ be a small Jordan domain containing $z_0$ such that   $\cup_{t\in[0,1]}f_t(\overline{V})\subset U$, and $K\subset V$ be a  sufficiently small invariant continuum at $z_0$ such that $\rho(I,K)=0$.  Let $M\setminus K\cup S_\infty\cup S^1$ be a compactification of  $M\setminus K$, where $S_{\infty}$ is the circle added when blowing $f$ at $\infty$ and $S^1$ is the circle added when blowing $f|_{M\setminus K}$ at the end $K$. It is a closed annulus, and $f|_{M\setminus K}$ extends continuously to a homeomorphism $\overline{f}$ of $M\setminus K\cup S_\infty\cup S^1$.   The homeomorphism $\overline{f}$ has a $(1,q)$ periodic orbit $O$ and hence by Proposition \ref{P: pre-boyland closed annulus} has a $(1,q)$ topologically monotone periodic orbit $O'$ (It could be equal to $O$ or different from $O$). Since the rotation number of $\overline{f}$ at both boundary is equal to $0$, $O'$ is included in $M\setminus K$ and hence in $U^-$. So,   $f|_{M\setminus K}$ is isotopic to a homeomorphism  $R_{1/q}$ satisfying that $R^q_{1/q}=\mathrm{Id}$.

     Let $h: M\setminus K\rightarrow M\setminus\{z_0\}$ be a homeomorphism whose restriction to $M\setminus U$ is equal to the identity. Then, $f'=h\circ f|_{M\setminus K}\circ h^{-1}$ is a homeomorphism of $M\setminus\{z_0\}$ which coincides $f$ in $M\setminus U$. The restriction of $f'\circ f^{-1}$ to $M\setminus U$ is equal to the identity, using Alexander's trick one deduces that $f'\circ f^{-1}|_{M\setminus O'}$ is isotopic to the identity. So, $f'|_{M\setminus O'}$ and  $f|_{M\setminus O'}$ are isotopic. Therefore, $f|_{M\setminus O'}$ is isotopic to   $R'_{1/q}=h|_{M\setminus O'} \circ R_{1/q}\circ h^{-1}|_{M\setminus O'}$ which satisfies of course $R'^q_{1/q}=\mathrm{Id}$.

\item[iii)] Suppose that $W$ is a proper subset of M whose boundary is the union of some proper leaves, and that $z_0$ is  accumulated by periodic orbits.

    As in case ii), $f$ can be blown-up at $\infty$ and the  blow-up rotation number $\rho(I,\infty)$ is equal to $0$.  Recall that there exists a $q$-periodic orbit with rotation number $p/q$ in the annulus $M\setminus\{z_0\}$ for all irreducible $p/q\in(0,\rho)$.  Fix two prime integers $q_1$ and $ q_2$ such that $1/q_2<1/q_1<\rho$. Choose  a $q_1$-periodic orbit $O_1$ and a $q_2$-periodic orbit $O_2$ in $M\setminus \{z_0\}$ with rotation number (associated to $I$)  $1/q_1$ and $1/q_2$ respectively. Recall that the rotation number of every periodic orbit in $M\setminus\{z_0\}$ is positive and that $\rho_s(I,z_0)$ is reduced to $0$. One deduces by the assertion i) and ii) that for any given integer $q> 1$, there is a neighborhood of $z_0$  that does not contain any $q$-periodic point of $f$. So, there exists  a Jordan domain $U$ containing $z_0$ that does not contain any periodic point of $f$  with period not bigger that $q_2$ except $z_0$.    Let $\gamma_1\subset U$ be a simple closed curve that  separates $z_0$ and $O_1\cup O_2$. Denote by $U^-_1$ the component of $M\setminus \gamma_1$ containing $O_1\cup O_2$. Let $\gamma_2$ be a simple closed curve that separates $\gamma_1$ and $z_0$ such that the trajectory of each $z\in \overline {U^-_1}$ along $I^{q_2}$ is in the component $U^-_2$ of $M\setminus \gamma_2$ containing $\gamma_1$. Let $\gamma_3$ be a simple closed curve that separates $\gamma_2$ and $z_0$ such that the trajectory of each $z\in \overline {U^-_2}$ along $I^{q_2}$ is in the component $U^-_3$ of $M\setminus \gamma_3$ containing $\gamma_2$. Since $\gamma_3\subset U$, there does not exist any periodic points of $f$ with periodic not bigger that $q_2$ in $\gamma_3$.  We can perturb $f$ in $M\setminus (\overline{U^-_3}\cup\{z_0\})$ and get a homeomorphism $f'$ such that $f'$  has  finitely many periodic points with periods not bigger than $q_2$ in $M\setminus U^-_{3}$.

    Let $X$ be the union of periodic orbits of $f'$ with periodic not bigger than $q_2$ that intersects $M\setminus \overline{U^-_3}$. It is a finite set containing $z_0$. We consider the annulus covering $\pi:\widetilde{M}\rightarrow M\setminus X$  such that the restriction of $\pi$ to a sufficiently small annulus near one end is a homeomorphism between this annulus and a small annulus near $\infty$ in $M\setminus X$.  As in Section \ref{S: pre-annulus covering projection}, we  add a point $\star$ at this end of $\widetilde{M}$. Let $\widetilde{U^-_i}$ be the component of $\pi^{-1}(U^-_i)$ that has an end $\star$ and $\widetilde{O}_i'$ be the lift of $O_i'$ in $\widetilde{U^-_2}$ for $i=1,2$.  Let $\widetilde{f}'$ be the lift of $f'|_{M\setminus X}$. It extends continuously to a homeomorphism of $\widetilde{M}\cup\{\star\}$, and the dynamics of $\widetilde{f}'$ near $\star$ is conjugate to the dynamics of $f'$ near $\infty$. So, $\widetilde{f}'$ can be blown-up at $\star$, and by choosing a suitable isotopy $\widetilde{I}'$ of $\widetilde{f}'$, the blow-up rotation number $\rho(\widetilde{I}',\star)$ is equal to $0$. Moreover, $\widetilde{O}_i'$ is a $q_i$-periodic orbit of $\widetilde{f}'$ with rotation number $1/q_i$ (associated to $\widetilde{I}'$), for $i=1,2$.  Referring to Section \ref{S: pre-extend lifts to the boundary}, one knows that $\widetilde{f}'$ can be blown-up at the other end.

     We blow-up $\widetilde{f}$ at both ends and get a homeomorphism $\overline{\widetilde{f}'}$ of a closed annulus. For $i=1,2$, the homeomorphism $\overline{\widetilde{f}'}$ has a $(1,q_i)$ periodic orbit,  so one can deduce by Proposition \ref{P: pre-boyland closed annulus} that $\overline{\widetilde{f}'}$ has a $(1,q_i)$ topologically monotone periodic orbit $\widetilde{O}_i''$. The circle we added at $\star$ does not contain any periodic points with rotation number different from  $0$, so it does not contain $\widetilde{O}_1''$ or $\widetilde{O}_2''$.  The rotation number of $\overline{\widetilde{f}'}$ at the circle we added at the other end is different from $1/q_1$ or $1/q_2$. Suppose that it is different from $1/q_1$, the other case can be treated similarly. Then,  $\widetilde{O}''_1$ is included in $ \widetilde{M}$, and $\pi(\widetilde{O}''_1)$ is a periodic orbit of $f'$ of period not bigger than $q_1$. So, $\pi(\widetilde{O}''_1)$ is included in $\overline{U^-_3}$, and hence is a periodic orbit of $f$ in $U^-_1$.
     \begin{figure}[h]
     \centering
      \includegraphics[width=4cm]{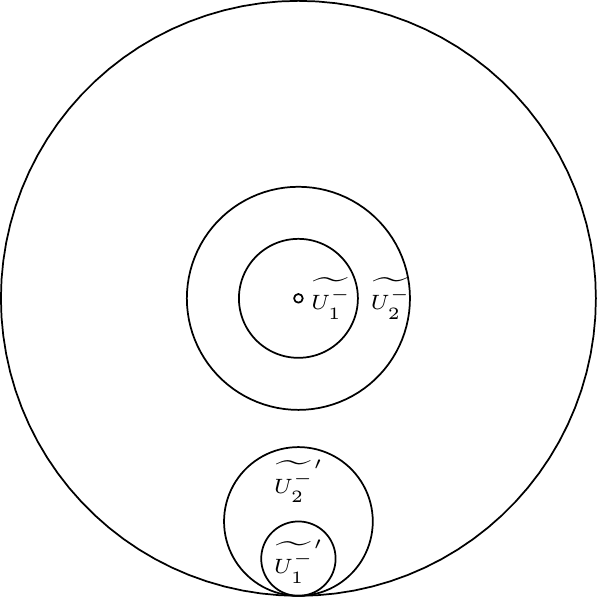}
      \end{figure}

We will prove by contradiction that $\widetilde{O}''_1$ is included in $\widetilde{U^-_1}$. Otherwise, suppose that there exists $\widetilde{z}\in\widetilde{O}''_1$ in another component $\widetilde{U^-_1}'$ of $\pi^{-1}(U^-_1)$. Then, $\widetilde{z}$ is a fixed point of $\widetilde{f}^{q_1}$. Let $\widetilde{U^-_2}'$ be the component of $\pi^{-1}(U^-_2)$ containing $\widetilde{U^-_1}'$.  Since $f'^{q_1}(U^-_1)\subset U^-_2$,  one deduces that $\widetilde{f}'^{q_1}(\widetilde{U^-_1}')\subset \widetilde{U^-_2}'$. Recall that the rotation number of $\widetilde{O}_1''$ is $1/q_1$.  So, the rotation number of $\overline{\widetilde{f}'}$ at the outer boundary is $1/q_1$, which contradicts our assumption.

Let $h:\widetilde{M}\rightarrow M\setminus \{z_0\}$ be a homeomorphism whose restriction to $\widetilde{U^-_2}$ is equal to $\pi$. As in the end of case ii), we deduce that $f|_{M\setminus \pi(\widetilde{O}''_1)}$ is isotopic to a homeomorphism  $R_{1/q_1}$ satisfying that $R^{q_1}_{1/q_1}=\mathrm{Id}$. The lemma is proved.
\end{itemize}
\end{proof}

\begin{proof}[Proof of Theorem \ref{T: chapter3-plane case}]
By the previous lemma, there exists  a  $q'$-periodic orbit $O$ with rotation number $1/q'> 0$ (associated to $I$) such that $f|_{M\setminus O}$ is isotopic to a homeomorphism  $R_{1/q'}$ satisfying  $R^{q'}_{1/q'}=\mathrm{Id}$. Let $I'=(\varphi_t)_{t\in[0,1]}$ be  an identity isotopy  of $f^{q'}$ that fixes every point in $O\cup \{z_0\}$.
Since the rotation number of $O$ associated to $I$ is $1/q'$, each point in $O$ is a fixed point of $f^{q'}$ and its rotation number associated to $I^{q'}$ is $1$. Because $I'$ fixes $O\cup\{z_0\}$, $I'|_{M\setminus\{z_0\}}$ is homotopic to $J^{-1}_{z_0}I^{q'}|_{M\setminus\{z_0\}}$, where $J_{z_0}$ is an identity isotopy of the identity fixing $z_0$ such that $\rho(J_{z_0},z_0)=1$. By the first assertion of Proposition \ref{P: pre-rotation set}, one knows that  $\rho_s(I',z_0)$ is reduced to $-1$.

Let $\pi':\widehat{M}\rightarrow M\setminus O$ be the universal cover.  Since $M\setminus O$ is a surface of finite type, we can endow it a hyperbolic structure, and $\widehat{M}$ can be viewed to be the hyperbolic plane.
Fix  $\widehat{z}_0\in\pi'^{-1}(z_0)$. Let $\widehat{f}$ be the lift of $f|_{M\setminus O}$ that fixes $\widehat{z}_0$.
 Then, $\widehat{f}$ can be blown-up at $\infty$.

 Let $\widehat{I}'=(\widehat{\varphi}_t)_{t\in[0,1]}$ be the identity isotopy of $\widehat{f}^{q'}$ that lifts $I'$. Then, $\rho_s(\widehat{I}',\widehat{z}_0)$ is reduced to $-1$.
On the other hand, $\infty$ is accumulated by the points of $\pi'^{-1}\{z_0\}$ which are fixed points of $\widehat{I}'$, so by the assertion ii) of Proposition \ref{P: pre-rotation set}, one knows that $0$ is belong to $\rho_s(\widehat{I}',\infty)$. But $\widehat{f}$ can be blown-up at $\infty$, by the assertion iv) of Proposition \ref{P: pre-rotation set}, we know that $\rho_s(\widehat{I}',\infty)$ is reduced to $0$.

Let $\widehat{I}_0$ be an identity isotopy of $\widehat{f}$ that fixes $\widehat{z}_0$ and satisfies $\rho_s(\widehat{I}_0,\widehat{z}_0)=\{0\}$. Then $\rho_s(\widehat{I}_0^{q'},\widehat{z}_0)$ is reduced to $0$, and hence $\widehat{I}_0^{q'}|_{\widehat{M}\setminus\{\widehat{z}_0\}}$ is homotopic to  $ J_{\widehat{z}_0}\widehat{I}'|_{\widehat{M}\setminus\{\widehat{z}_0\}}$. So, $\rho_s(\widehat{I}^{q'}_0,\infty)$ is reduced to $-1$, and by the assertion i) of Proposition \ref{P: pre-rotation set}, we deduce that $\rho_s(\widehat{I}_0,\infty)$ is reduced to $-1/q'$. Since $\widehat{f}$ can be blown-up at $\infty$, by the assertion iv) of Proposition \ref{P: pre-rotation set},  one knows that the blow-up rotation number $\rho(\widehat{I}_0,\infty)$ is equal to $-1/q'$.

Every $\widehat{z}'_0\in\pi'^{-1}\{z_0\}\setminus\{\widehat{z}_0\}$ is a contractible fixed point of $\widehat{f}^{q'}|_{\widehat{M}\setminus\{\widehat{z}_0\}}$ associated to $\widehat{I}'|_{\widehat{M}\setminus\{\widehat{z}_0\}}$, so it is not a contractible fixed point of $\widehat{f}|_{\widehat{M}\setminus\{\widehat{z}_0\}}$ associated to $\widehat{I}_0|_{\widehat{M}\setminus\{\widehat{z}_0\}}$.

Let $\widehat{O}'$ be a periodic orbit of $\widehat{f}$ in the annulus $\widehat{M}\setminus\{\widehat{z}_0\}$ such that $z_0\notin \pi'(\widehat{O}')$ and the rotation number of $\widehat{O}'$ associated to $\widehat{I}_0|_{\widehat{M}\setminus\{\widehat{z}_0\}}$ is $p/q$. Then $\widehat{O}'$ is a periodic orbit of $\widehat{f}^{q'}$  in the annulus $\widehat{M}\setminus\{\widehat{z}_0\}$ and the rotation number associated to $\widehat{I}'|_{\widehat{M}\setminus\{\widehat{z}_0\}}$ is $\frac{pq'}{q}-1$. So, $\pi'(\widehat{O}')$ is a periodic orbit of $f$ in the annulus $M\setminus\{z_0\}$, the rotation number associated to $I'$ is $\frac{pq'}{q}-1$, the rotation number associated to $I^{q'}$ is $\frac{pq'}{q}$, and the rotation number associated to $I$ is $p/q$. In particular, if $\widehat{z}'$ is a contractible  fixed point of $\widehat{f}|_{\widehat{M}\setminus\{\widehat{z}_0\}}$  associated to $\widehat{I}_0|_{\widehat{M}\setminus\{\widehat{z}_0\}}$, $\pi'(\widehat{z}')$ is a contractible fixed point of $f|_{M\setminus\{z_0\}}$ associated to $I|_{M\setminus\{z_0\}}$. So, $\widehat{f}|_{\widehat{M}\setminus\{\widehat{z}_0\}}$ does not have any contractible fixed point associated to $\widehat{I}_0$.

 Moreover, if $p/q$ is irreducible, and if $\widehat{O}'$ is a periodic orbit of $\widehat{f}$ of type $(p,q)$ associated to $\widehat{I}_0$ at $\widehat{z}_0$ such that $z_0\notin \pi'(\widehat{O}')$, then $\pi'(\widehat{O}')$ is a periodic orbit of $f$ of type $(p,q)$ associated to $I$ at $z_0$.

 By  Lemma \ref{L: chapter3-mainproof1-plane1}, the property \textbf{P)} holds for $(\widehat{f},\widehat{I}_0,\widehat{z}_0)$, and then holds for $(f,I,z_0)$.
\end{proof}

\subsection{The case where the total area of $M$ is finite}\label{S: M plane}

In this section, we assume that the area of $M$ is finite. Recall that $f$ is an area preserving homeomorphism of $M$, that $z_0$ is an isolated fixed point of $f$ satisfying $i(f,z_0)=1$, that $I$ is an identity isotopy of $f$ that fixes $z_0$ and satisfies $\rho_s(I,z_0)=\{k\}$.  Let $(X, I_X)$ be a maximal extension of $I$ that satisfies $\rho_s(I_X,z_0)=\rho_s(I,z_0)$. Write $X_0=X\setminus\{z_0\}$. Then, $X_0$ is a closed subset of $\mathrm{Fix}(f)$, and $I_X$ can be extended to a maximal identity isotopy on $M\setminus X_0$ that fixes $z_0$. To simplify the notation, we still denote by $I_X$ this extension.  Moreover, by definition of Jaulent's preorder, we know  that a periodic orbit of type $(p,q)$ associated to $I_X$ at $z_0$ is a periodic orbit of type $(p,q)$ associated to $I$ at $z_0$.  Let $M_0$ be the connected component of $M\setminus X_0$ that contains $z_0$.  Of course  the total area of $M_0$ is also finite. When $M$ is a sphere,  $f|_{M\setminus\{z_0\}}$ has at least one fixed point (see Section \ref{S: pre-Brouwer plane translation theorem}), and hence $X_0$ is not empty. So, $M_0$ is not a sphere. To simplify the notations, we denote by $f_0$ the restriction of $f$ to $M_0$, and  by $I_0$ the restriction of  $I_X$ to $M_0$. If the property \textbf{P)} holds for $(f_0,I_0, z_0)$, it holds  for $(f,I,z_0)$. So, we will prove the following proposition, and the second part of Theorem \ref{T: main-index 1 and rotation 0 implies  accumulated by periodic points} is also proved.

\begin{prop}\label{P: chapter3-mainproof2-property p hold if the area is finite}
Under the previous assumptions, the property \textbf{P)} holds for $(f_0,I_0, z_0)$.
\end{prop}
We will prove this proposition  in the following four cases:
\begin{itemize}
\item[-] the component $M_0$ is a plane and $\rho_s(I,z_0)$ is reduced to $0$;
\item[-] the component $M_0$ is neither a sphere nor a plane and $\rho_s(I,z_0)$ is reduced to $0$;
\item[-] the component $M_0$ is a plane and $\rho_s(I,z_0)$ is reduced to an non-zero integer $k$;
\item[-] the component $M_0$ is neither a sphere nor a plane and $\rho_s(I,z_0)$ is reduced to an non-zero integer $k$.
\end{itemize}
We will use some results that  will be deduced  in the first two cases to obtain the last two cases.

\subsubsection{The case where $M_0$ is a plane and $\rho_s(I,z_0)$ is reduced to $0$}\label{S: chapter3-plane finite area rotation 0}

 In this case, $I_0$ is a maximal identity isotopy on the plane $M_0$ that fixes only one point $z_0$ and  satisfies  $\rho_s(I_0,z_0)=\{0\}$.  The result of Proposition \ref{P: chapter3-mainproof2-property p hold if the area is finite} is just a corollary of Theorem \ref{T: chapter3-plane case} and the following lemma:

 \begin{lem}\label{L: chapter3-positive area implies periodic orbits}
  Under the  previous assumptions, there exists a periodic orbit of $f$ in the annulus $M_0\setminus \{z_0\}$.
 \end{lem}

 \begin{proof}
 Of course, we can assume that  $z_0$ is not accumulated by periodic orbits. As in  Remark \ref{R: nonaccumulated and rotation 0 means indefferent}, one knows that $z_0$ is an indifferent fixed point with rotation number $\rho(I,z_0)=0$.

Let $\mathcal{F}$ be  a transverse foliation of $I_0$. One knows that  $\mathcal{F}$ has a unique singularity $z_0$ and an end $\infty$. By the assumption in Remark \ref{R: asume positive}, $z_0$ is a sink of $\mathcal{F}$. Since $f_0$ is area preserving and the total area of $M_0$ is finite,  $\infty$ is a source of $\mathcal{F}$ and all the leaves of $\mathcal{F}$ are lines from $\infty$ to $z_0$.
Let $\pi:\mathbb{R}\times(0,1)\rightarrow M_0\setminus\{z_0\}$ be the universal cover such that the leaves of the lift $\widetilde{\mathcal{F}}$ of $\mathcal{F}$ are the vertical lines oriented upward. Let $\widetilde{f}$ be the lift of $f_0$ associated to $I_0$, and $p_1:\mathbb{R}\times(0,1)\rightarrow \mathbb{R}$ be the projection onto the first factor.  Then we know that
  $$p_1(\widetilde{f}(\widetilde{z})-\widetilde{z})>0, \text{ for all } \widetilde{z}\in \mathbb{R}\times(0,1).$$
  Let V be a small Jordan domain in the annulus $M_0\setminus\{z_0\}$ such that $f(V)\cap V=\emptyset$. Let $\widetilde{V}$ be one of the connected components of $\pi^{-1}(V)$. By choosing $V$ small enough, one can suppose that
$$|p_1(\widetilde{z})-p_1(\widetilde{z}')|<\frac{1}{2} \text{ for all } \widetilde{z},\widetilde{z}'\in \widetilde{V}. $$
  Then, for every $z\in V$ and $\widetilde{z}\in \pi^{-1}\{z\}$, we know that
$$\frac{p_1(\widetilde{f}^n(\widetilde{z})-\widetilde{z})}{n}\ge\frac{
 \left(\sum_{k=1}^{n}\chi_{V}(f^k(z))\right)-\frac{1}{2}}{n}.$$
We define $U=\cup_{k\in\mathbb{Z}}f^{k}(V)$. By Poincar\'e Recurrence Theorem, almost all points in $U$ are recurrent.
By  Birkhoff-Khinchen Theorem, for almost all $z\in U$, and every $\widetilde{z}\in \pi^{-1}\{z\}$, both of the two limits
$$\lim_{n\rightarrow \infty}\frac{p_1(\widetilde{f}^n(\widetilde{z})-\widetilde{z})}{n} \text{ and } \lim_{n\rightarrow\infty}\sum_{k=1}^{n}\frac{\chi_{V}(f^k(z))}{n}$$
exist, and there exists a non negative measurable  function $\varphi$ on  $U$ that satisfies $\varphi\circ f=\varphi$ and
 $$\lim_{n\rightarrow\infty}\sum_{k=0}^{n-1}\frac{\chi_{V}(f^k(z))}{n}=\varphi(z)\text{ for almost all } z\in U.$$
Moreover, by Lebesgue's dominated convergence theorem,
$$\int_U \varphi=\int \chi_V =\mathrm{Area}(V)>0.$$
Therefore, there exist a recurrent point $z\in V$ and $\widetilde{z}\in \pi^{-1}\{z\}$ such that the limit
$$\lim_{n\rightarrow \infty}\frac{p_1(\widetilde{f}^n(\widetilde{z})-\widetilde{z})}{n}$$
exists and is positive. So,  the rotation number of $z$ is positive. We denote it by $\rho$.

On the other hand, let $K_0$ be a small enough continuum at $z_0$ whose rotation number is $0$. We denote  by $(W\setminus K_0)\sqcup \mathbb{T}^1$ the prime-ends compactification at the end $K_0$, which is an annulus. We can extend $f$ to $\mathbb{T}^1$ and  know that the  rotation number on  $\mathbb{T}^1$ is $0$. Then, there exists a fixed point on $\mathbb{T}^1$ whose rotation number is $0$.

By the remark that follows Proposition \ref{P: pre-Le Calvez's generalization of poincare-birkhoof}, there exists a $q$-periodic orbit of rotation number $p/q$ in the annulus $(W\setminus K_0)$, for all irreducible $p/q\in(0,\rho)$.
 \end{proof}

\subsubsection{The case where $M_0$ is neither a sphere nor a plane and $\rho_s(I,z_0)$ is reduced to $0$}\label{S: rotation 0 and not plane}
Recall that $f_0$ is an area preserving homeomorphism of $M_0$,  that $z_0$ is an isolated fixed point of $f_0$ satisfying $i(f_0,z_0)=1$,  that $I_0$ is a maximal identity isotopy that fixes only one point $z_0$ and satisfies $\rho_s(I_0,z_0)=\{0\}$.

As in Section \ref{S: pre-annulus covering projection}, let  $\pi:\widetilde{M}\rightarrow M_0\setminus \{z_0\}$ be the annulus covering projection, $\widetilde{I}$ be the natural lift of $I_0$ to $\widetilde{M}\cup\{\star\}$, $\widetilde{f}$ be the lift of $f_0$ to $\widetilde{M}\cup\{\star\}$ associated to $I_0$. Then $\widetilde{I}$ is a maximal identity isotopy and $\mathrm{Fix}(\widetilde{I})$ is reduced to $\star$. For all irreducible $p/q\in \mathbb{Q}$, if $O$ is a periodic orbit of type $(p,q)$ associated to $\widetilde{I}$ at $\star$, then $\pi(O)$ is a periodic orbit of type $(p,q)$ associated to $I_0$ at $z_0$. So, if the property \textbf{P)} holds for $(\widetilde{f},\widetilde{I},\star)$, then it holds for $(f_0,I_0,z_0)$. The result of Proposition \ref{P: chapter3-mainproof2-property p hold if the area is finite} is a corollary of Theorem \ref{T: chapter3-plane case} and the following  Proposition \ref{P: chanpter3-mainproof2-exist periodic orbit in W}, which is the most difficult part of this article.

\begin{prop}\label{P: chanpter3-mainproof2-exist periodic orbit in W}
There exists a periodic orbit of $\widetilde{f}$ besides $\star$.
\end{prop}

The idea of the proof of the proposition is the following: we will first consider several simple situations such that there exists a periodic orbit  of $\widetilde{f}$ besides $\star$, then we suppose that we are not in these situations and follow the idea of Le Calvez  (see Section 11 of \cite{lecalvezfeuilletage}) to get a contradiction.

Let us begin with some necessary assumptions and lemmas. Of course, we can suppose that $\star$ is not accumulated by periodic orbits of $\widetilde{f}$. As in Remark \ref{R: nonaccumulated and rotation 0 means indefferent}, $\star$ is an  indifferent fixed point of $\widetilde{f}$ and the rotation number $\rho(\widetilde{I}, \star)$ is equal to $0$.
Let $\mathcal{F}$ be a transverse foliation of $I_0$, and $\widetilde{\mathcal{F}}$ be the lift of $\mathcal{F}$. By the assumption in Remark \ref{R: asume positive}, $z_0$ is a sink of $\mathcal{F}$, and $\star$ is a sink of $\widetilde{\mathcal{F}}$. Denote by  $W$   the attracting basin of $z_0$ for  $\mathcal{F}$, and by $\widetilde{W}$ the attracting basin of $\star$ for $\widetilde{\mathcal{F}}$. Write $\dot{W}=W\setminus\{z_0\}$ and $\dot{\widetilde{W}}=\widetilde{W}\setminus\{\star\}$. Recall that  $\pi|_{\dot{\widetilde{W}}}$ is a homeomorphism between $\dot{\widetilde{W}}$ and $\dot{W}$ and can be extended continuously to a homeomorphism between $\widetilde{W}$ and $W$. The area  on $M_0$ induces an  area on $\widetilde{M}$.  So $\widetilde{f}$ is area preserving, and the area of $\widetilde{W}$ is finite.

\begin{lem}\label{L: chapter3-mainproof2-positive area continuum implies the existence of periodic orbis}
Under the previous assumptions, if there exists an invariant continuum $K\subset\widetilde{W}$ with positive area, then there exists a periodic orbit  besides $\star$.
\end{lem}

\begin{proof}
The proof is similar to the proof of Lemma \ref{L: chapter3-positive area implies periodic orbits} except some small modifications when we try to find a recurrent point with positive rotation number. We will give a more precise description.

Since $\widetilde{W}$ is different from $\widetilde{M}\cup\{\star\}$, we can not get a  lift of $\widetilde{f}$ as in the proof of Lemma \ref{L: chapter3-positive area implies periodic orbits}. Instead, we will get a similar one by the following procedure. Let $\pi':\mathbb{R}^2\rightarrow \dot{\widetilde{W}}$ be a universal cover which sends the vertical lines upwards to the leaves of $\widetilde{F}|_{\dot{\widetilde{W}}}$. Since $K$ is an invariant subset of $\widetilde{W}$, we can lift $\widetilde{f}|_{K\setminus\{\star\}}$ to  a homeomorphism  $\widehat{f}$  of $\pi'^{-1}(K\setminus\{\star\})$ such that
$$p_1(\widehat{f}(\widehat{z})-\widehat{z})>0, \text{ for all } \widehat{z}\in \pi'^{-1}(K\setminus\{\star\}), $$
where $p_1$ is the projection onto the first factor.

Also, we should replace the small Jordan domain $V$ in the proof of Lemma \ref{L: chapter3-positive area implies periodic orbits} with $V\cap K$ by choosing suitable $V$ such that the area of  $V\cap K$ is positive, that $f(V)\cap V=\emptyset$, and that for every component $\widehat{V}$ of $\pi'^{-1}(V)$, one has
$$|p_1(\widehat{z})-p_1(\widetilde{z}')|<1/2\quad \text{ for all }\widehat{z},\widehat{z}'\in \widehat{V}.$$
We can always find such a set because the area of $K$ is positive.
\end{proof}

\begin{lem}\label{L: chapter3-mainproof2-positive rotation implies the existence of periodic orbits}
Under the previous assumptions, if there exists an invariant continuum $K\subset \widetilde{W}$ such that $\rho(\widetilde{I},K)\ne 0$, then there exists a periodic orbit in $\widetilde{M}$.
\end{lem}

\begin{proof}
  Recall that $\pi|_{\dot{\widetilde{W}}}$ is a homeomorphism between $\dot{\widetilde{W}}$ and $\dot{W}$. So, $\widetilde{W}$ is a proper subset of $\widetilde{M}\cup\{\star\}$, and  the boundary of $\widetilde{W}$ is the union of some  proper leaves. By Remark \ref{R: pre-blow-up}, one knows that $\widetilde{f}$ can be blown-up at $\infty$ and the blow-up rotation number $\rho(\widetilde{I},\infty)$ is equal to $0$.

We consider the prime-ends compactification of $\widetilde{M}\setminus K$ at the end $K$, and extend  $\widetilde{f}$  continuously to a homeomorphism of $(\widetilde{M}\setminus K)\sqcup S^1$. We get a homeomorphism $g$ of the closed annulus $S_{\infty}\sqcup(\widetilde{M}\setminus K)\sqcup S^1$ that coincides with $\widetilde{f}$ on $\widetilde{M}\setminus K$, where $S_{\infty}$ is the circle we added when blowing-up $\widetilde{f}$ at $\infty$.

Moreover, $g$ satisfies the intersection property and has different rotation numbers at each boundary, then  by Proposition \ref{P: pre-poincare-birkhoof-closed annulus}, there exists a periodic orbit in $\widetilde{M}\setminus K$, which is also a periodic orbit of $\widetilde{f}$.
\end{proof}

\begin{lem}\label{L: chapter3-mainprof2-like saddle point}
Suppose that there exists  a closed disk $D\subset \widetilde{W}$  containing $\star$ as an interior point such that the connected component  of $\bigcap_{k\in\mathbb{Z}} \widetilde{f}^{-k}(D)$ containing $\star$ is contained in the interior of $D$. Then $\widetilde{f}$ has  another  periodic orbit besides $\star$.
\end{lem}

\begin{proof}
We will proof this lemma by contradiction. Suppose that $\widetilde{f}$ does not have any other periodic orbit.  Let $K$ be the connected component of  $\bigcap_{k\in\mathbb{Z}}\widetilde{f}^{-k}(D)$ containing $\star$.  We  identify $K$ as a point $\{K\}$, and still denote by $\widetilde{f}$ the reduced homeomorphism. The fixed point $\{K\}$ is a non-accumulated saddle-point of $\widetilde{f}$ with index $i(\widetilde{f},\{K\})=i(\widetilde{f},K)=i(\widetilde{f},0)=1$. By Proposition \ref{P: pre-rotation number of a saddle point}, $\widetilde{f}$ can be blown-up at $\{K\}$ and  $\rho(\widetilde{f}, \{K\})$ is different from $0\in\mathbb{R}/\mathbb{Z}$. So, $\rho(\widetilde{I},K)$ is different from $0$. By the previous lemma, $\widetilde{f}$ has  another  periodic orbit besides $\star$, which is a contradiction.
\end{proof}

Now we begin the proof of Proposition \ref{P: chanpter3-mainproof2-exist periodic orbit in W}.

\begin{proof}[Proof of Proposition \ref{P: chanpter3-mainproof2-exist periodic orbit in W}]
We will prove this proposition by contradiction. Suppose that there does not exist any other periodic orbit except $\star$.
 Let $(D_n)_{n\in\mathbb{}N}$ be an increasing sequence of closed disks containing $\star$ as an interior point such that $D_{n}$ is contained in the interior of $D_{n+1}$  for all $n\in\mathbb{N} $ and $\cup_{n\in \mathbb{N}} D_n=\widetilde{W}$. Let $K_n$ be the connected component of $\cap_{k\in \mathbb{Z}} \widetilde{f}^{-k}(D_n)$ containing $\star$.  By Lemma \ref{L: chapter3-mainproof2-positive rotation implies the existence of periodic orbits}, we know that $\rho(\widetilde{I}, K_n)$ is equal to $0$ for every $n\in\mathbb{N}$. By Lemma \ref{L: chapter3-mainprof2-like saddle point}, each $K_n$ intersects the boundary of $D_n$.  Let $K=\overline{\cup_{n\in\mathbb{N}} K_n}\subset \widetilde{M}\cup\{\star\}$. It is an invariant set of $\widetilde{f}$. The boundary of $\widetilde{W}$ is the union of proper leaves, so for every point in $\partial \widetilde{W}$, either its image or its pre-image by $\widetilde{f}$ will leave $\widetilde{W}$. Therefore, $K$ can not touch the boundary of $\widetilde{W}$, and is included in  $\widetilde{W}$. But each $K_n$ intersects the boundary of $D_n$, so $K$ intersects every neighborhood of $\infty$.

\begin{lem}\label{L: chapter3-mainproof2-each componnet of M-K is unbounded}
 There does not exist any connected component of $\widetilde{M}\setminus K$ that is included in  $\widetilde{W}$.
\end{lem}
   \begin{proof}
   We will give a proof by contradiction.   Suppose that there exists   a component $\widetilde{U}$ of $\widetilde{M}\setminus K$ such that $\widetilde{U}\subset \widetilde{W}$.  Then  $\partial U$ is a subset of $K$, which is invariant by $\widetilde{f}$. So, $\partial (\widetilde{f}^n(U))$ is a subset of $K$ for every $n\in\mathbb{Z}$, and  one deduces that $\cup_{n\in\mathbb{Z}}\widetilde{f}^{n}(\widetilde{U})\subset \widetilde{W}$. Moreover, $\star$ is not an interior point of  $\widetilde{U}$,  and $\widetilde{U}$ is homeomorphic to a disk. We know that the area of $\widetilde{W}$ is finite, so there exists $q\in\mathbb{N}$ such that $\widetilde{f}^q(\widetilde{U})=\widetilde{U}$. Then, one knows that $\widetilde{f}^q$ has a fixed point (see Section \ref{S: pre-Brouwer plane translation theorem}), and  hence $\widetilde{f}$ has a periodic point different from $\star$. We get a contradiction.
\end{proof}

 Let $\pi':\widehat{M}\rightarrow \widetilde{M}$ be the universal cover, and $T$ be  a generator of the group of covering automorphisms. Let $\widehat{I}=(\widehat{f}_t)_{t\in[0,1]}$ be the natural lift of $\widetilde{I}$, and  $\widehat{\mathcal{F}}$ be the lift of $\widetilde{\mathcal{F}}$. Write  $\widehat{f}=\widehat{f}_1$. It is  the lift of $\widetilde{f}$ associated to $\widetilde{I}$. Write $\widehat{K}=\pi'^{-1}(K\setminus\{\star\})$, and $\widehat{W}=\pi'^{-1}(\dot{\widetilde{W}})$.

%
%

Because $K$ is connected, each connected component of $\widetilde{M}\setminus K$ is simply connected. So, if $\widetilde{U}$ is one of the connected components of $\widetilde{M}\setminus K$, and if $\widehat{U}$ is one of the components of $\pi'^{-1}(\widetilde{U})$,  then $\widehat{U}$ does not intersect $T(\widehat{U})$. Therefore, $\widehat{M}\setminus\widehat{K}$ is not connected and has infinitely many components. By Lemma \ref{L: chapter3-mainproof2-each componnet of M-K is unbounded}, each component of $\widehat{M}\setminus \widehat{K}$ contains a proper leaf in $\partial\widehat{W}$, and hence a disk bounded by this leaf. As in the following picture,
\begin{figure}[h]
    \begin{minipage}[t]{0.45\linewidth}
      \centering
   \includegraphics[width=3cm]{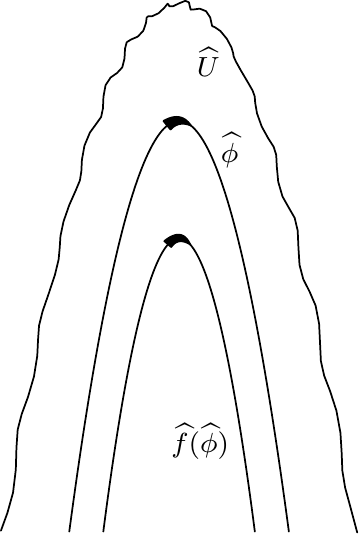}
    \end{minipage}
  \hspace{0.1\linewidth}
    \begin{minipage}[t]{0.45\linewidth}
      \centering
     \includegraphics[width=3cm]{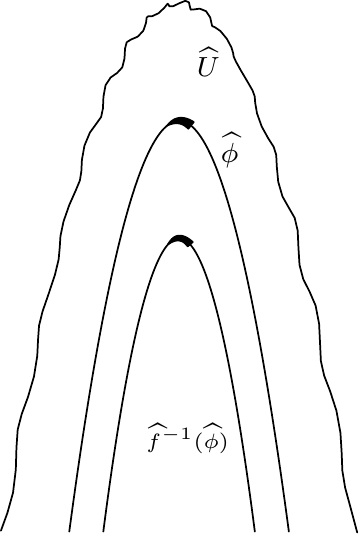}
    \end{minipage}
  \caption{Each component  $\widehat{U}$ of $\widehat{M}\setminus\widehat{K}$ is invariant by $\widehat{f}$}
\end{figure}
this disk contains the image or the pre-image of this proper leaf by $\widehat{f}$.   So,  every component of $\widehat{M}\setminus \widehat{K}$ is invariant by $\widehat{f}$.

\begin{lem}\label{L: chapter3-mainsproof2-arc in tildeW dived tildeM}
Each leaf in $\widetilde{W}$ is an arc from $\infty$ to $\star$.
\end{lem}
\begin{proof}
Recall that the area of $\widetilde{W}$ is finite. So, there exist a leaf included in $\partial\widetilde{W}$ such that $\widetilde{W}$ is to its right and a leaf included in $\partial\widetilde{W}$ such that $\widetilde{W}$ is to its left. (Otherwise, $\widetilde{W}$ contains the positive or negative orbit of a wandering open set $\widetilde{W}\setminus\widetilde{f}(\widetilde{W})$ or $\widetilde{W}\setminus\widetilde{f}^{-1}(\widetilde{W})$ respectively.)
 Therefore, a leaf in $\widetilde{W}$ is an arc from $\infty$ to $\star$.
\end{proof}

Every leaf $\widehat{\Phi}\subset \widehat{W}$ divides $\widehat{M}$ into two part.  We denote by $R(\widehat{\Phi})$ the component of $\widehat{M}\setminus\widehat{\Phi}$ to the right of $\widehat{\Phi}$ and by $L(\widehat{\Phi})$ the component to the left.

\begin{lem}\label{L: chapter3-mainproof2-K doesnot conatin a leaf in W}
 There does not exist any leaf $\widehat{\Phi}\subset\widehat{W}$ such that $\widehat{\Phi}\subset \widehat{K}$.
\end{lem}
 \begin{proof}
 We can prove this lemma by contradiction. Suppose that $\widehat{\Phi}\subset\widehat{K}$. Then a component  of $\widehat{M}\setminus\widehat{K}$ is either to the left or to the right of $\widehat{\Phi}$. Moreover, if it is to the right (resp. left) of $\widehat{\Phi}$,  it is to the right (resp. left) of $\widehat{f}(\widehat{\Phi})$. Therefore, $R(\widehat{\Phi})\cap L(\widehat{f}(\widehat{\Phi}))$ is included in $\widehat{K}$, and so the interior of $\widehat{K}$ is not empty. We deduce that  $K$ is an invariant set of $\widetilde{f}$ with non-empty interior and  finite area. By Lemma \ref{L: chapter3-mainproof2-positive area continuum implies the existence of periodic orbis}, there exists a periodic orbit of $\widetilde{f}$ in $\widetilde{M}$, which is a contradiction.
\end{proof}

 \begin{lem}\label{L: chapter3-mainproof2-each leaf intersects only one component of M-K}
Let $\widehat{\Phi}$ be a leaf in $\widehat{W}$, $t\mapsto\widehat{\Phi}(t)$ be an oriented parametrization of $\widehat{\Phi}$, and $\widehat{U}$ be a component of $\widehat{M}\setminus\widehat{K}$.  If $\widehat{\Phi}$ intersects $ \widehat{U}$, then  both the area of $L(\widehat{\Phi})\cap\widehat{U}$ and the area of $R(\widehat{\Phi})\cap\widehat{U}$ are infinite, and  there exists $t_0$ such that $\widehat{\Phi}(t)\in\widehat{U}$ for all $t\le t_0$.
 \end{lem}

  \begin{proof}
  We will first give a proof of the first statement by contradiction. We suppose that the area of $L(\widehat{\Phi})\cap\widehat{U}$ is finite, the other case can be treated similarly.  Then, $L(\widehat{\Phi})\cap R(\widehat{f}^{-1}(\widehat{\Phi}))\cap\widehat{U}$ is a wandering open set whose negative orbit  is contained in $L(\widehat{\Phi})\cap\widehat{U}$. It contradicts the fact that $\widehat{f}$ is area preserving.

 Let us prove the second statement.    We know that both the area of $L(\widehat{\Phi})\cap\widehat{U}$ and the area of $R(\widehat{\Phi})\cap\widehat{U}$ are infinite. Since $\pi'|_{\widehat{U}}$ is injective, both the area of $\pi'(L(\widehat{\Phi})\cap\widehat{U})$ and the area of $\pi'(R(\widehat{\Phi})\cap\widehat{U})$ are infinite. The area of $\widetilde{W}$ is finite, so both $\pi'(L(\widehat{\Phi})\cap\widehat{U})$ and $\pi'(R(\widehat{\Phi})\cap\widehat{U})$ intersect $\widetilde{M}\setminus\widetilde{W}$, and hence both $L(\widehat{\Phi})\cap\widehat{U}$ and $R(\widehat{\Phi})\cap\widehat{U}$ intersect $\widehat{M}\setminus \widehat{W}$.  Therefore, there exists a proper leaf $\widehat{\Phi}_1$ in $L(\widehat{\Phi})\cap\widehat{U}$ and a proper leaf $\widehat{\Phi}_2$ in $R(\widehat{\Phi})\cap\widehat{U}$. Fix  a parametrization $t\mapsto\widehat{\Phi}_1(t)$ of $\widehat{\Phi}_1$ and a parametrization $t\mapsto\widehat{\Phi}_2(t)$ of $\widehat{\Phi}_2$, and draw a path $\gamma$ in $\widehat{U}$ from a point of $\widehat{\Phi}_1$ to a point of $\widehat{\Phi}_2$.  Let $s_1=\inf\{t:\widehat{\Phi}_1(t)\in\gamma\}$, $s_2=\sup\{t:\widehat{\Phi}_2(t)\in\gamma\}$, and $\gamma'$ be the sub-path of $\gamma$ connecting $\widehat{\Phi}_1(s_1)$ to $\widehat{\Phi}_2(s_2)$. Then, as in the following picture,
\begin{figure}[h]
    \begin{minipage}[t]{0.2\linewidth}
      \centering
   \includegraphics[width=2.5cm]{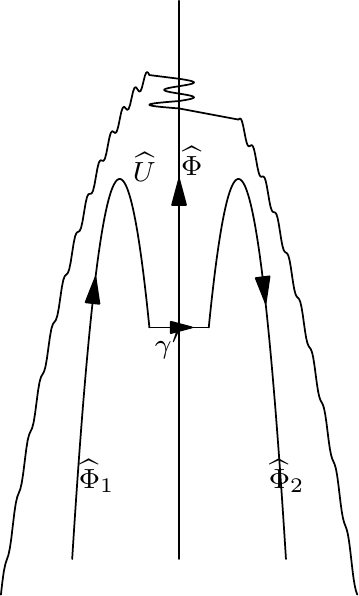}
    \end{minipage}
  \hspace{0.05\linewidth}
    \begin{minipage}[t]{0.2\linewidth}
      \centering
     \includegraphics[width=2.5cm]{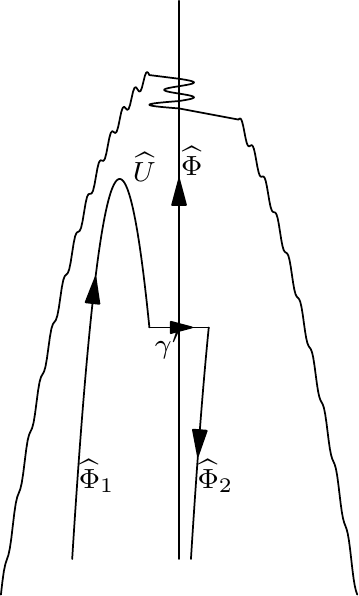}
    \end{minipage}
    \hspace{0.05\linewidth}
    \begin{minipage}[t]{0.2\linewidth}
      \centering
   \includegraphics[width=2.5cm]{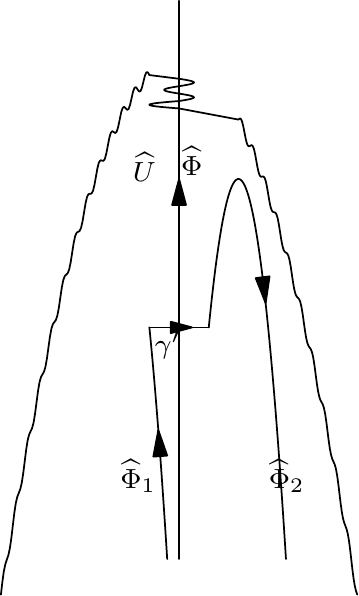}
    \end{minipage}
     \hspace{0.05\linewidth}
    \begin{minipage}[t]{0.2\linewidth}
      \centering
   \includegraphics[width=2.5cm]{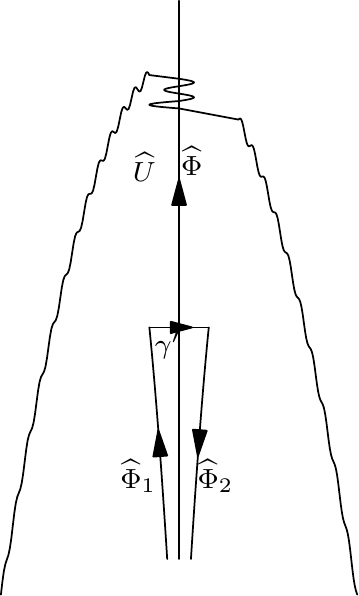}
    \end{minipage}
  \caption{Four possible cases in the proof of  Lemma \ref{L: chapter3-mainproof2-each leaf intersects only one component of M-K}}
\end{figure}
 $\Gamma=\widehat{\Phi}|_{(-\infty, s_1]}\gamma'\widehat{\Phi}_2|_{[s_2,\infty)}$ is an oriented proper arc and satisfies $R(\Gamma)\subset\widehat{U}$. We know that $\widehat{\Phi}$ intersects $\gamma'$. Let $t_0$ be a lower bound of the set $\{t:\widehat{\Phi}(t)\in\gamma'\}$.   We know that $\widehat{\Phi}|_{(-\infty,t_0]}\subset \widehat{U}$.
\end{proof}

Let $\delta:\mathbb{T}^1\rightarrow \widetilde{W}$ be an embedding that intersect $\widetilde{F}$ transversely, and $\widehat{\delta}:\mathbb{R}\rightarrow \widehat{W} $ be the lift of $\delta$. Then $\widehat{\delta}$ intersects every leaf in $\widehat{W}$, and intersects each leaf at only one point. Moreover, if $\widehat{\delta}$ intersects $\widehat{\Phi}$ and $\widehat{\Phi}'$ at $\widehat{\delta}(t)$ and $\widehat{\delta}(t')$ respectively, and if $t<t'$, then $\widehat{\Phi}$ is to the left of $\widehat{\Phi}'$, and $\widehat{\Phi}'$ is to the right of $\widehat{\Phi}$.
We define a map  $h:\mathbb{R}\rightarrow \widehat{\mathcal{F}}|_{\widehat{W}}$ by $h(t)=\widehat{\Phi}$ if $\widehat{\delta}(t)\in\widehat{\Phi}$.
\begin{lem}
The set of points $t\in\mathbb{R}$ such that $h(t)\cap\widehat{U}\ne\emptyset$ is open for each component $\widehat{U}$ of $\widehat{W}\setminus \widehat{K}$.
\end{lem}
\begin{proof}
 We fix a component $\widehat{U}$ of $\widehat{M}\setminus\widehat{K}$, and will first prove that the set $\{t: h(t)\cap\widehat{U}\ne\emptyset\}$ is open. Given a real number $t$ such that $h(t)$ intersects $\widehat{U}$ and $z\in h(t)\cap\widehat{U}$,  there is a trivialization neighborhood $V$ of $z$ such that $V\subset (\widehat{U}\cap \widehat{W})$. Moreover, $h^{-1}(V)$ is  an open interval  containing $t$.  So, the set $\{t: h(t)\cap\widehat{U}\ne\emptyset\}$  is open.
 \end{proof}

By Lemma \ref{L: chapter3-mainproof2-K doesnot conatin a leaf in W}, each leaf of $\widehat{\mathcal{F}}$ in $\widehat{W}$  intersects at least a component of $\widehat{M}\setminus\widehat{K}$. By lemma \ref{L: chapter3-mainproof2-each leaf intersects only one component of M-K}, each leaf of $\widehat{\mathcal{F}}$ in $\widehat{W}$  intersects at most one component of $\widehat{M}\setminus\widehat{K}$. So, each leaf of $\widehat{\mathcal{F}}$ in $\widehat{W}$  intersects exactly one component of $\widehat{M}\setminus\widehat{K}$.
Since $\widehat{M}\setminus\widehat{K}$ has countable components,   $$\mathbb{R}=\cup_{\widehat{U}}\{t: h(t)\cap\widehat{U}\ne\emptyset\}$$ is a disjoint union of countable many open sets. This is impossible.
\end{proof}

 \subsubsection{The case where  $M_0$ is a plane and $\rho_s(I,z_0)$ is reduced to a non-zero integer $k$}

Recall that $f_0$ is an area preserving homeomorphism of $M_0$,  that $z_0$ is an isolated fixed point of $f_0$ satisfying $i(f_0,z_0)=1$,  and that $I_0$ is a maximal identity isotopy that fixes only one point $z_0$ and satisfies $\rho_s(I_0,z_0)=\{k\}$.
In this case,  one can easily deduce that the result of Proposition \ref{P: chapter3-mainproof2-property p hold if the area is finite} is just a corollary of the result in the previous two cases. We will give a brief explanation.  Let $J$ be the identity isotopy of the identity map on $M_0$ fixing $z_0$ and satisfying $\rho_s(J,z_0)=1$. Write $I'_0=J^{-k}I_0$. It is  an identity isotopy of $f_0$ that satisfies $\rho_s(I'_0, z_0)=\{0\}$. By the result of Proposition \ref{P: chapter3-mainproof2-property p hold if the area is finite} in the two cases we have already proved, the property \textbf{P)} holds for $(f_0,I'_0,z_0)$.  A periodic orbit in  $M_0$ of type $(p,q)$ associated to $I'_0$ at $z_0$ is  a periodic orbit of type $(kq+p,q)$ associated to $I_0$ at $z_0$. So, the property \textbf{P)} holds for $(f_0,I_0,z_0)$.

 \subsubsection{The case where  $M_0$ is neither a sphere nor a plane and $\rho_s(I,z_0)$ is reduced to a non-zero integer $k$}

Recall that $f_0$ is an area preserving homeomorphism of $M_0$,  that $z_0$ is an isolated fixed point of $f_0$ satisfying $i(f_0,z_0)=1$,  that $I_0$ is a maximal identity isotopy that fixes only one point $z_0$ and satisfies $\rho_s(I_0,z_0)=\{k\}$.

As in Section \ref{S: pre-annulus covering projection}, let  $\pi:\widetilde{M}\rightarrow M_0\setminus \{z_0\}$ be the annulus covering projection, $\widetilde{I}$ be the natural lift of $I_0$ to $\widetilde{M}\cup\{\star\}$, and $\widetilde{f}$ be the lift of $f_0$ associated to $I_0$ to $\widetilde{M}\cup\{\star\}$. Then $\widetilde{I}$ is a maximal identity isotopy and $\mathrm{Fix}(\widetilde{I})$ is reduced to $\star$.
As before,  if the property \textbf{P)} holds for $(\widetilde{f},\widetilde{I},\star)$, it holds for $(f_0,I_0,z_0)$.

Let $\mathcal{F}$ be a transverse foliation of $I_0$, and $\widetilde{\mathcal{F}}$ be the lift of $\mathcal{F}$. Since $\rho_s(I_0,z_0)$ is reduced to a non-zero integer, by the assertion iv) of Proposition \ref{P: pre-rotation set}, $z_0$ is a sink or a source of $\mathcal{F}$ and $\star$ is a sink or a source of $\widetilde{\mathcal{F}}$. Let $W$ be the attracting or repelling basin of $z_0$ for $\mathcal{F}$, and  $\widetilde{W}$ be the  attracting or repelling basin of $\star$ for $\widetilde{\mathcal{F}}$. Recall that $\pi|_{\widetilde{W}\setminus\{\star\}}$ is a homeomorphism be tween $\widetilde{W}\setminus \{\star\}$ and $W\setminus\{z_0\}$. So, $\widetilde{W}$ is a strict subset of $\widetilde{M}\cup\{\star\}$, and its boundary is  the union of some proper leaves. By Remark \ref{R: pre-blow-up}, one knows that $\widetilde{f}$ can be blown-up at $\infty$ and $\rho(\widetilde{I},\infty)$ is equal to $0$.

Let $J$ be the identity isotopy of the identity map of $\widetilde{M}\cup\{\star\}$ fixing $\star$ and satisfying $\rho_s(J,\star)=1$. Write $\widetilde{I}'=J^{-k}\widetilde{I}$. We know that  $\rho_s(\widetilde{I}',\star)$ is reduced to $0$, and that  the blow-up rotation number $\rho(\widetilde{I}',\infty)$ is equal to $k$. One deduces by the assertion ii) of Proposition \ref{P: pre-rotation set} that there exists a neighborhood of $\infty$ that does not contain any contractible fixed points of $\widetilde{f}|_{\widetilde{M}}$  associated to  $\widetilde{I}'|_{\widetilde{M}}$. Let $(Y, \widetilde{I}_Y)$ be a maximal extension of $(\{\star\},\widetilde{I}')$ (see Section \ref{S: pre-Jaulent's preorder}). One knows that $Y$ is  a closed subset of the union of $\{\star\}$ and the set of contractible fixed points of $\widetilde{f}|_{\widetilde{M}}$ associated to $\widetilde{I}'|_{\widetilde{M}}$.  So, there is a neighborhood of $\infty$ that  does not intersect $Y$, and hence $Y$ is a compact set in $\widetilde{M}\cup\{\star\}$. One knows also that  $\rho_s(\widetilde{I}_Y,\star)$ is reduced to $0$, and that  the blow-up rotation number $\rho(\widetilde{I}_Y,\infty)$ is equal to $k$.  As in the previous subsection, in order to prove the result of Proposition \ref{P: chapter3-mainproof2-property p hold if the area is finite}, we only need to prove that the property \textbf{P)} holds for $(\widetilde{f},\widetilde{I}_Y,\star)$, which is the aim of this subsection.

\begin{prop}\label{P: chapter3-mainproof3}
Under the previous assumptions, the property \textbf{P)} holds for $(\widetilde{f},\widetilde{I}_Y,\star)$.
\end{prop}

\begin{proof}
To get this result, one has to consider two cases: $Y$ is reduced to a single point $\star$ or it contains at least two points. In the first case,  the proposition  is a corollary of Lemma \ref{L: chapter3-mainproof1-plane1}. Now, we will prove the proposition   in the second case.

Suppose that $Y$ contains at least two points and write $Y_0=Y\setminus\{\star\}$.  Let $\widetilde{M}_0$ be the connected component of $\widetilde{M}\cup\{\star\}\setminus Y_0$ containing $\star$. Recall that $Y$ is a compact subset of $\widetilde{M}\cup\{\star\}$. So, one has to consider the following two cases:
\begin{itemize}
\item[-]$\widetilde{M}_0$ is a bounded plane,
 \item[-]$\widetilde{M}_0$ is neither a sphere nor a plane.
 \end{itemize}
 In the first case, the area of $\widetilde{M}_0$ is finite, and the problem is reduced to the case of Section \ref{S: chapter3-plane finite area rotation 0}; while in the second case, we will prove the  result of Proposition \ref{P: chapter3-mainproof2-property p hold if the area is finite} like in Section \ref{S: rotation 0 and not plane}.

Now, we suppose that $\widetilde{M}_0$ is neither a sphere nor a plane.  Let $\pi'': \breve{M}\rightarrow \widetilde{M}_0$ be an annulus covering map,  $\breve{I}$ be the natural lift of $\widetilde{I}_Y|_{\widetilde{M}_0}$ to $\breve{M}\cup\{\breve{\star}\}$, and $\breve{f}$ be the lift of $\widetilde{f}|_{\widetilde{M}_0}$ to $\breve{M}\cup\{\breve{\star}\}$ associated to $\widetilde{I}_Y|_{\widetilde{M}_0}$.  As before, if the Property \textbf{P)} holds for $(\breve{f},\breve{I},\breve{\star})$, then it holds also for $(\widetilde{f},\widetilde{I}_Y,\star)$. So, the proposition is a corollary of Theorem \ref{T: chapter3-plane case} and the following Lemma \ref{L: chapter3-mainproof2-not plane+rotation non 0}.
\end{proof}

\begin{lem}\label{L: chapter3-mainproof2-not plane+rotation non 0}
 There exists a periodic orbit of $\breve{f}$ besides $\breve{\star}$.
 \end{lem}

\begin{proof}
The proof is similar to the proof of Proposition \ref{P: chanpter3-mainproof2-exist periodic orbit in W}.

 Let $\widetilde{\mathcal{F}}_Y$ be a transverse foliation of $\widetilde{I}_Y$, and $\breve{\mathcal F}$ be the lift of $\mathcal{F}_Y|_{\widetilde{M}_0}$ to $\breve{M}\cup\{\breve{\star}\}$. Recall the assumption in Remark \ref{R: asume positive}, one knows that $\star$ is a sink of $\widetilde{\mathcal{F}}_Y$ and that $\breve{\star}$ is a sink of $\breve{\mathcal{F}}$. Let $\widetilde{W}^*$ be the attracting basin of $\star$ for $\widetilde{\mathcal{F}}_Y$ and $\breve{W}$ be the attracting basin of $\breve{\star}$ for $\breve{\mathcal{F}}$. Recall that $\pi''|_{\breve{W}\setminus \{\breve{\star}\}}$ is a homeomorphism between $\breve{W}\setminus \{\breve{\star}\}$ and $\widetilde{W}^*\setminus\{\star\}$.

When $k\ge 1$, one deduces  by Proposition \ref{P: pre-rotation set} that the end $\infty$ is sink  of $\widetilde{\mathcal{F}}_Y$.  In this case $\widetilde{W}^*$ is a bounded subset of $\widetilde{M}\cup\{\star\}$, and hence the area of  both $\widetilde{W}^*$ and $\breve{W}$ are finite. We can repeat the proof of Proposition \ref{P: chanpter3-mainproof2-exist periodic orbit in W}, and get the result.

Now, we suppose that $k\le -1$. In this case, the end $\infty$ is a source of of $\widetilde{\mathcal{F}}_Y$.

 \begin{sublem}\label{L: chapter3-mainproof2-each leaf in breveW is from infinite to brevestar}
 Each leaf in $\breve{W}$ is an arc from infinite to $\breve{\star}$.
 \end{sublem}

 \begin{proof}
 When the area of $\breve{W}$ is finite, we deduces the result as in  Lemma \ref{L: chapter3-mainsproof2-arc in tildeW dived tildeM}.
Now suppose that  the area of $\breve{W}$ is infinite. We consider the compactification of $\widetilde{M}\cup\{\star\}$ by adding a point $\infty$ at infinite,  the added point $\infty$ is a source of $\widetilde{F}_Y$ and is at the boundary of $\widetilde{W}^*$.  So, there exists a leaf in $\widetilde{W}^*$ from the singularity $\infty$ to $\star$ whose  lift  in $\breve{W}$ is a leaf from  infinite to $\breve{\star}$. Therefore, each leaf in $\breve{W}$ is an arc from infinite to $\breve{\star}$.
 \end{proof}

 The  difference between our case and the case of Proposition \ref{P: chanpter3-mainproof2-exist periodic orbit in W} is that the area of $\breve{W}$ may be infinite. But we did not use this condition except in the proof of Lemma \ref{L: chapter3-mainsproof2-arc in tildeW dived tildeM} and Lemma \ref{L: chapter3-mainproof2-K doesnot conatin a leaf in W}.
 We have proven  Sublemma \ref{L: chapter3-mainproof2-each leaf in breveW is from infinite to brevestar} corresponding to Lemma \ref{L: chapter3-mainsproof2-arc in tildeW dived tildeM}. We will prove that the area of $K$ is finite, so the result of Lemma \ref{L: chapter3-mainproof2-K doesnot conatin a leaf in W} is still valid.

  Formally,  suppose that there does not exist any periodic orbits besides $\breve{\star}$. Let $(D_n)_{n\in\mathbb{}N}$ be an increasing sequence of closed disks containing $\breve{\star}$ such that $D_{n}$ is contained in the interior of $D_{n+1}$  for all $n\in\mathbb{N} $ and $\cup_{n\in \mathbb{N}} D_n=\breve{W}$. Let $K_n$ be the connected component of $\cap_{k\in \mathbb{Z}} \breve{f}^{-k}(D_n)$ containing $\breve{\star}$ and  $K=\overline{\cup_{n\in\mathbb{N}} K_n}\subset \breve{M}\cup\{\star\}$. We will prove that the area of $K$ is finite.

 Let $K'_n=\pi''(K_n)$, and $K'=\overline{\cup_{n}K'_n}\subset\widetilde{M}\cup\{\star\}\cup S_{\infty}$, where $S_{\infty}$ is the circle we added when blowing-up $\widetilde{f}$ at $\infty$.  As before, we can deduce that $K\subset \breve{W}$. Recall that $\pi''|_{\breve{W}\setminus \{\breve{\star}\}}$ is a homeomorphism between $\breve{W}\setminus \{\breve{\star}\}$ and $\widetilde{W}^*\setminus\{\star\}$.  Therefore, we know that $\pi''(K)\subset K'$ , and that the area of  $K$ is not bigger than the area of $K'$. So, we only need to prove that the area of $K'$ is finite.

We will prove it by contradiction. Suppose that the area of $K'$ is infinite. One deduces that $K'\cap S_{\infty}\ne \emptyset$.  As was proven in Section \ref{S: rotation 0 and not plane}, one knows that $\rho(\breve{I},K_n)=0$, and so $\rho(\widetilde{I}_Y, K'_n)=0$ for all $n\in\mathbb{N}$.  Since $(Y,\widetilde{I}_Y)$ is a maximal extension of $(\{\star\},J^{-k}\widetilde{I})$, one deduces that  $\rho(J^{-k}\widetilde{I}, K'_n)$ is equal to $0$ and that $\rho(\widetilde{I},K'_n)$ is equal to $k$, for all $n\in\mathbb{N}$.

 Since $K'\cap S_{\infty}$ is invariant by $\widetilde{f}$ and the blow-up rotation number $\rho(\widetilde{I},\infty)=0$, there exists a fixed point $\widetilde{z}_1\in K'\cap S_{\infty}$, and the rotation number of $\widetilde{z}_1$ (associated to $\widetilde{I}$) in the annulus $\widetilde{M}\cup S_{\infty}$ is $0$.

 Let $\pi':\widehat{M}\rightarrow \widetilde{M}\cup S_{\infty}$ be the universal cover,  $T$ be a generator of the group of covering automorphism, and  $\widehat{f}$ the lift of $\widetilde{f}$ associated to $\widetilde{I}$. Fix one $\widehat{z}_1\in \pi'^{-1}(\widetilde{z}_1)$. It is a fixed point of $\widehat{f}$.
 Let $U$ be a small neighborhood of $\widehat{z}_1$ such that $T^n(U)\cap U=\emptyset$ for all $n\ne 0$. Let $V\subset U$ be a neighborhood of $\widehat{z}_1$ such that $\widehat{f}^2(V)\subset U$. Fix $n$ large enough such that $K'_n\cap V\ne \emptyset$, and choose an arc $\gamma$ in $V$ connecting $\widehat{z}_1$ and an accessible point of $K'_n$ such that $\gamma\cap K'_n$ has exactly one point. By choosing a sub-arc of $\gamma\cup \widehat{f}^{-2}(\gamma)$, we get a cross-cut $\gamma'$. On one hand, $T(\gamma')\cap \gamma'=\emptyset $ because $\gamma'\subset V$.  On the other hand, we consider the prime-ends compactification of $\widetilde{M}\cup S_{\infty}\setminus K'_n$ at the end $K'_n$, and denote by $\widetilde{f}_{K'_n}$ the extension of $\widetilde{f}|_{\widetilde{M}\setminus K'_n}$. As was in Section \ref{S: pre-prime-ends Compactification}, let  $\pi'_{K'_n}:\pi'^{-1}(\widetilde{M}\cup S_{\infty}\setminus K'_n)\cup \mathbb{R}\rightarrow (\widetilde{M}\cup S_{\infty}\setminus K'_n)\cup S^1$ be the universal cover, and $\widehat{f}_{K'_n}$ the lift of $\widetilde{f}_{K'_n}$ whose restriction to $\pi'^{-1}(\widetilde{M}\cup S_{\infty}\setminus K'_n)$ is equal to $\widehat{f}$.
  \begin{figure}[h]
 \centering
 \includegraphics{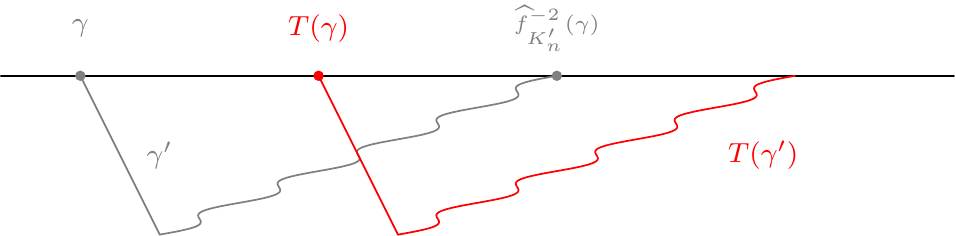}
 \end{figure}
 Recall that $\rho(\widetilde{I}, K'_n)$ is equal to $k\le -1$.  So, the end-cut $\widehat{f}_{K'_n}^{-2}(\gamma)>T^{-2k-1}(\gamma)\ge T(\gamma)$, which means $\gamma'\cap T(\gamma')\ne \emptyset$. We get a contradiction.
 \end{proof}

\section{The  case of diffeomorphisms}\label{S: chapter3-diffeomorphism case}

%

\subsection{The index at a degenerate  fixed point that is an extremum of a generating function }\label{S: index at a degenerate fixed point}

Let $f$ be a diffeomorphism of $\mathbb{R}^2$ and $g: \mathbb{R}^2\rightarrow \mathbb{R}$ be a $\mathcal{C}^2$ function, we call $g$ a \emph{generating function} of $f$ if  $\partial^2_{12}g<1$, and if
\begin{equation}\label{Eq: generating function}
 f(x,y)=(X,Y)\Leftrightarrow\left\{\begin{aligned} X-x & = & \partial_2 g(X,y),\\ Y-y & = & -\partial_1 g(X,y) . \end{aligned}\right.
\end{equation}

Every $\mathcal{C}^2$ function $g: \mathbb{R}^2\rightarrow \mathbb{R}$ satisfiying $\partial^2_{12} g\le c<1$ defines a diffeomorphism $f$ of $\mathbb{R}^2$ by the previous equations.  Moreover,  the Jacobian matrix $J_{f}$ of $f$ is  equal to
\begin{equation*}
\frac{1}{1-\partial^2_{12}g(X,y)}
\begin{pmatrix} 1 &    \partial^2_{22}g(X,y)\\
-\partial^2_{11}g(X,y)   & -\partial^2_{11}g(X,y) \partial^2_{22}g(X,y) +(1-\partial^2_{12} g(X,y))^2
\end{pmatrix}.
\end{equation*}
Since  $\det J_f=1$, the diffeomorphism  $f$  is orientation and area preserving.
On the other hand, every orientation and area preserving diffeomorphism $f$ of $\mathbb{R}^2$ satisfying $0<\varepsilon\le\partial_1(p_1\circ f)\le M<\infty $ can be generated by a generating function, where $p_1$ is the projection onto the first factor.

Moreover, we can naturally define an identity isotopy $I_0=(f_t)_{t\in[0,1]}$ of $f$ such that $f_t$ is generated by $tg$. Precisely, the diffeomorphisms $f_t$ are defined by the following equations:
 \begin{equation}\label{Eq: generate an isotopy}
 f_t(x,y)=(X^t,Y^t)\Leftrightarrow\left\{\begin{aligned} X^t-x & = & t \partial_2 g(X^t,y),\\ Y^t-y & = &  -t \partial_1 g(X^t,y) . \end{aligned}\right.
\end{equation}

  A point $(x,y)$ is a fixed point of $f$ if and only if it is a critical point of $g$. We say that a fixed point $(x,y)$ of $f$ is \emph{degenerate} if $1$ is an eigenvalue of $J_f(x,y)$. We will see later that a fixed point $(x,y)$ of $f$ is degenerate if and only if the Hessian matrix of $g$ at $(x,y)$ is degenerate.

 We can also define a local generating function. Precisely, if $(x,y)$ is a critical point of a  $\mathcal{C}^2$ function $g$ such that $\partial_{12} g(x,y)<1$, then one can define an orientation and area preserving local diffeomorphism $f$ at   $(x,y)$ by  the equations (\ref{Eq: generating function}). On the other side, if $(x,y)$ is a fixed point of  an orientation and area preserving diffeomorphism $f$ such that  $\partial_1(p_1\circ f)(x,y)>0$, where $p_1$ is the projection to the first factor, then one can find a $\mathcal{C}^2$ function $g$ defined in a neighborhood of $(x,y)$, that  defines the germ of $f$ at $(x,y)$ by the equations (\ref{Eq: generating function}). Moreover, in both cases, we can define a local isotopy of $f$ at $(x,y)$ by the equations (\ref{Eq: generate an isotopy}), and will call it  the \emph{local isotopy induced by $g$}.

\bigskip

In this section, suppose that $f: (W,0)\rightarrow (W',0)$ is a local diffeomorphism at  $0\in \mathbb{R}^2$, and that $g$ is a local generating function of $f$. We will prove the following Proposition \ref{P: chapter3-diffeo-the index of degenerate fixed point that is an extremum of genrating function}, and deduce Corollary \ref{C: main-generating function with a degenrate maximal point means accumlated by periodic points} as an immediate consequence of Theorem \ref{T: main-index 1 and rotation 0 implies  accumulated by periodic points} and Proposition \ref{P: chapter3-diffeo-the index of degenerate fixed point that is an extremum of genrating function}.

\begin{prop}\label{P: chapter3-diffeo-the index of degenerate fixed point that is an extremum of genrating function}
If $0$ is an isolated critical point of $g$ and  a local extremum  of $g$, and if the Hessian matrix of $g$ at $0$ is degenerate,  then  $i(f,0)$ is equal to $1$.
\end{prop}

\begin{proof}
The idea is to  compute the indices of the local isotopies, so that we can deduce the Lefschetz index by Proposition \ref{P: pre-relation of indices between isotopy and homeo}.

 We denote by $I_0=(f_t)_{t\in[0,1]}$  the local isotopy induced by $g$. We have the following lemma
\begin{lem}\label{L: chapter3-generating fuction-rho(I,0)=0}
 The blow-up rotation number $\rho(I_0,0)$ is equal to $0$.
 \end{lem}

 \begin{proof}
 Since $\mathrm{Hess}(g)(0)$ is degenerate, one deduces that $0$ is an eigenvalue of $\mathrm{Hess}(g)(0)$. Let $v$ be an eigenvector of $\mathrm{Hess}(g)(0)$ corresponding to the eigenvalue $0$. We will prove that $v$ is a common eigenvector of $J_{f_t}(0)$ corresponding to the eigenvalue $1$ for $t\in[0,1]$, and hence the blow-up rotation number $\rho(I_0,0)$ is equal to $0$.

Write
 $$\mathrm{Hess}(g)(0)=\begin{pmatrix} \varrho & \sigma\\ \sigma & \tau\end{pmatrix} \quad \text{and} \quad v=\begin{pmatrix} a\\  b\end{pmatrix}.$$
  Then, one deduces that
$$\varrho\tau-\sigma^2=0, \quad \varrho a+\sigma b=0, \quad \text{and } \sigma a+\tau b=0.$$
By a direct computation, one knows that for every $t\in[0,1]$,
$$J_{f_t}(0)=\frac{1}{1-t\sigma}\begin{pmatrix}1 & t\tau\\ -t\varrho & -t^2\varrho\tau+(1-t\sigma)^2 \end{pmatrix}
=\frac{1}{1-t\sigma}\begin{pmatrix}1 & t\tau\\ -t\varrho & 1-2t\sigma \end{pmatrix},$$
and then
$$
J_{f_t}(0)v=\frac{1}{1-t\sigma}
\begin{pmatrix}  a+t\tau b \\
-t\varrho a+ b-2t\sigma b
\end{pmatrix}
=\frac{1}{1-t\sigma}\begin{pmatrix} a-t\sigma a \\ t\sigma b+ b-2t\sigma b\end{pmatrix}
= \begin{pmatrix} a\\  b\end{pmatrix}.$$
\end{proof}

Since $f$ is area preserving, the rotation set at $0$ is not empty. By the assertion iv) of Proposition \ref{P: pre-rotation set}, and the previous lemma, one can deduce that $\rho_s(I_0,0)$ is reduced to $0$, and that for all local isotopy $I$ of $f$ that is not equivalent to $I_0$, the rotation  set $\rho_s(I,0)$ is reduced to a non-zero integer.

\begin{lem}
If $I$ is a local isotopy of $f$ that is not equivalent to $I_0$, then  $i(I,0)$ is equal to $0$.
\end{lem}
\begin{proof}
Let $\mathcal{F}$ be  foliation  locally transverse to  $I$. Since $\rho_s(I,0)$ is reduced to a non-zero integer, one can deduce by the assertion iii) of Proposition \ref{P: pre-rotation set} that $0$ is either a sink or a source of $\mathcal{F}$. By Proposition \ref{P: pre-relations of indices between  foliation and the others}, one deduces that $i(I,0)=i(\mathcal{F},0)-1=0$.
\end{proof}

In order to compute the index of $I_0$, we will construct an isotopy $I'$ that is equivalent to $ I_0$, and prove that  $i(I',0)=0$.

We  define  $I'=(f'_t)_{t\in[0,1]}$ in a neighborhood of $0$ by
 $$f'_t(x,y)=\left\{\begin{aligned}(x, y)+2t(X-x,0) & & \textrm{for } 0\leq t\leq 1/2 ,\\ (X,y)+(2t-1)(0, Y-y) & & \textrm{for } 1/2\leq t\leq 1,\end{aligned}\right.$$
where $(X,Y)=f(x,y)$.

\begin{lem}
The family $I'=(f'_t)_{t\in[0,1]}$ is  a local isotopy of $f$.
 \end{lem}
 \begin{proof}
 For every fixed $t\in[0,1]$, We will prove that $f'_t$ is a local diffeomorphism by  computing the determinant of the Jacobian matrices, and then get the result.

 Indeed, one knows
 $$\partial_1 X=1/(1-\partial_{12}g)>0.$$
  Then for $t\in[0,1/2]$,
\begin{equation*}
\det J_{f'_t}=\det\begin{pmatrix} 1+2t(\partial_1 X-1)) & 2t\partial_2 X\\  0 & 1\end{pmatrix}=2t\partial_1 X+(1-2t)>0;
\end{equation*}
and for $t\in[1/2,1]$,
$$
\det J_{f'_t}=\det\begin{pmatrix} \partial_1 X & \partial_2 X\\  (2t-1)\partial_1 Y & (2-2t)+(2t-1)\partial_2 Y\end{pmatrix}
=(2t-1)\det J_{f}+(2-2t)\partial_1 X>0.
$$
\end{proof}

\begin{lem}
The blow-up rotation number $\rho(I',0)$ is equal to $0$, and hence $I'$ is equivalent to $I_0$.
\end{lem}
\begin{proof}
As in the proof of Lemma \ref{L: chapter3-generating fuction-rho(I,0)=0}, we will prove that an eigenvector  of $\mathrm{Hess}(g)(0)$ corresponding to the eigenvalue $0$   is a common eigenvector of $J_{f'_t}(0)$ corresponding to the eigenvalue $1$ for $t\in[0,1]$, and hence deduce the lemma.

We keep the notations in the proof of Lemma \ref{L: chapter3-generating fuction-rho(I,0)=0}, and recall that
$$\varrho\tau-\sigma^2=0, \quad \varrho a+\sigma b=0, \quad \text{and } \sigma a+\tau b=0.$$
For $t\in[0,1/2]$,
$$J_{f'_t}(0)=\mathrm{Id}+2t\begin{pmatrix}\partial_1 X(0,0)-1 & \partial_2 X(0,0)\\ 0 & 0\end{pmatrix}=\mathrm{Id}+\frac{2t}{1-\sigma}\begin{pmatrix} \sigma & \tau \\  0 & 0 \end{pmatrix},$$
and
$$J_{f'_t}(0)v= v+\frac{2t}{1-\sigma}\begin{pmatrix}\sigma a+\tau b\\ 0\end{pmatrix}=v.$$
For $t\in[1/2,1]$,
$$J_{f'_t}(0)=J_{f}(0)-(2-2t)\begin{pmatrix}0 & 0 \\ \partial_1 Y(0,0) &\partial_2 Y(0,0)-1 \end{pmatrix}=J_{f}(0)-\frac{2-2t}{1-\sigma}\begin{pmatrix} 0 & 0 \\ -\varrho & -\sigma \end{pmatrix},$$
and
$$J_{f'_t}(0)v= J_{f}(0)v+\frac{2-2t}{1-\sigma}\begin{pmatrix} 0\\ \varrho a+\sigma b\end{pmatrix}=v.$$
We have verified that $v$ is a common eigenvector of $J_{f'_t}(0)$ corresponding to the eigenvalue $1$ for $t\in[0,1]$.
\end{proof}

To conclude, we will define a locally transverse foliation $\mathcal{F}_0$ of $I'$ such that $0$ is a sink or a source of $\mathcal{F}_0$, and then deduce by Proposition \ref{P: pre-relations of indices between  foliation and the others} that $i(I', 0)=i(\mathcal{F}_0,0 )-1=0$. Indeed, let  $\mathcal{F}_0$  be the foliation in a neighborhood of $0$  whose leaves are the integral curves of the gradient vector field\footnote{It means the vector field: $(x,y)\mapsto (\partial_1 g(x,y), \partial_2 g(x,y))$.} of $g$. One knows that $0$ is a sink of $\mathcal{F}_0$ if $0$ is a local maximum of $g$, and is a source of $\mathcal{F}_0$ if $0$ is a minimum of $g$.  We can finish our proof by the following lemma.

\begin{lem}
The foliation $\mathcal{F}_0$ is locally transverse to $I'$.
\end{lem}

\begin{figure}[h]
\center
   \includegraphics[width=3.5cm]{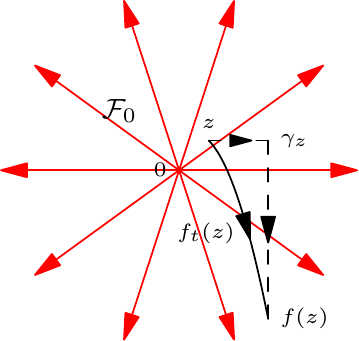}

  \caption{The dynamics and foliation generated by $g(x,y)=x^2+y^2$}
\end{figure}

\begin{proof}
Let $U$ be a sufficiently small Jordan domain containing $0$ such that $\mathcal{F}_0$ is well defined on $U$, and $V\subset U$ be a sufficiently small neighborhood of $0$ such  that  $f'_t$ is well defined on $V$ for $t\in[0,1]$, that $f$ does not have any other fixed point in $V$ except $0$, and that $\cup_{t\in[0,1]}f'_t(V)\subset U$. We will prove that  for every  $z=(x,y)\in V\setminus\{0\}$, the path $\gamma_z: t\mapsto f'_t(x,y)$ is   positively transverse to $\mathcal{F}_0$, and then deduce the lemma.

 Indeed, for $t\in[0,1/2]$,
\begin{eqnarray*}
&&\det\begin{pmatrix}
2(X-x) & \partial_1 g(f'_t(x,y))\\
0 &\partial_2 g(f'_t(x,y))
\end{pmatrix}\\
&=&2(X-x)\partial_2 g(f'_t(x,y))\\
&=&2(X-x)\partial_2 g(2tX+(1-2t)x,y)\\
&=&2(X-x)[\partial_2 g(X,y)+(2t-1)(X-x)\partial^2_{12}g(\xi, y)]  \\
&=&2(X-x)^2[1-(1-2t)\partial^2_{12}g(\xi, y)]\ge 0
\end{eqnarray*}
where $\xi$ is a real number between $x$ and $X$, and the inequality is strict if $X\ne x$.

 For $t\in[1/2,1]$,
\begin{eqnarray*}
&&\det\begin{pmatrix}
0 & \partial_1 g(f'_t(x,y))\\
2(Y-y) &\partial_2 g(f'_t(x,y))
\end{pmatrix}\\
&=&-2(Y-y)\partial_1 g(f'_t(x,y))\\
&=&-2(Y-y)\partial_1 g(X,(2-2t)y+(2t-1)Y)\\
&=&-2(Y-y)[\partial_1 g(X,y)+(2t-1)(Y-y)\partial^2_{12} g(X, \eta)]  \\
&=&2(Y-y)^2[1-(2t-1)\partial^2_{12} g (X ,\eta)]\ge 0
\end{eqnarray*}
where $\eta$ is a real number between $y$ and $Y$, and the inequality is strict if $Y\ne y$.

Since $z=(x,y)$ is not a fixed point, either $X\ne x$ or $Y\ne y$. If both of the inequalities are satisfied, $\gamma_z$ is positively transverse to $\mathcal{F}_0$; if $X\ne x$ and $Y= y$, $\gamma_z|_{t\in[0,\frac{1}{2}]}$  is positively transverse to $\mathcal{F}_0$, and $\gamma_z|_{t\in[\frac{1}{2},1]}$ is reduced to a point; if $X= x$ and $Y\ne y$, $\gamma_z|_{t\in[0,\frac{1}{2}]}$ is reduced to a point, and $\gamma_z|_{t\in[\frac{1}{2}, 1]}$ is positively transverse to $\mathcal{F}_0$.
\end{proof}
\end{proof}

\begin{rmq}
In the proof, we have indeed proven that $\mathcal{F}_0$ is locally transverse to any local isotopy  of $f$ that is equivalent to $I_0$.
\end{rmq}

\subsection{Discrete symplectic actions and symplectically degenerate extrema}\label{S: Discrete symplectic actions and symplectically degenerate extremum}

In this section, we will introduce  symplectically degenerate extrema.  More details can be found in \cite{mazzuchelli}.

We say that a diffeomorphism  $F:\mathbb{T}^2\rightarrow \mathbb{T}^2$ is  \emph{Hamiltonian} if  it is area preserving and  if there exists a lift $f$ satisfying
$$f(z+k)=f(z)+k \text{ for all } k\in\mathbb{Z}^2, \text{ and } \int_{\mathbb{T}^2}(f-\mathrm{Id})dxdy=0.$$
Refering to \cite{Mcduff}, this definition coincides with the usual definition of a Hamiltonian diffeomorphism of a symplectic manifold. More precisely, we call a time-dependent vector field $(X_t)_{t\in\mathbb{R}}$  a \emph{Hamiltonian vector field} if it is defined by the equation:
$$dH_t=\omega(X_t,\cdot),$$
where $(M, \omega)$ is a symplectic manifold and $H:\mathbb{R}\times M\rightarrow \mathbb{R}$ is a smooth function. The Hamiltonian vector field induces a \emph{Hamiltonian flow} $(\varphi_t)_{t\in\mathbb{R}}$ on $M$, which is the solution of  the following equation
$$\frac{\partial}{\partial t}\varphi_t(z)=X_t(\varphi_t(z)).$$
 We say that a diffeomorphism $F$ of $M$ is a \emph{Hamiltonian diffeomorphism} if it is the time-$1$ map of a Hamiltonian flow.
So, for a Hamiltonian diffeomorphism, there exists a natural identity isotopy $I$ which is defined by the Hamiltonian flow. We say that a fixed point  of a Hamiltonian diffeomorphism is \emph{contractible} if its trajectory along $I$ is a loop homotopic to zero in $M$, and that a $q$-periodic  point   of a Hamiltonian diffeomorphism is \emph{contractible} if it is a contractible fixed point of $F^q$.

\bigskip

Let $F: \mathbb{T}^2\rightarrow \mathbb{T}^2$ be a Hamiltonian diffeomorphism. Then $F$ is the time-$1$ map of a Hamiltonian flow, and we can factorize $F$ by
$$F=F_{k-1}\circ\cdots\circ F_0,$$
 where $F_j$ is $\mathcal{C}^1$-close to the identity, for $j=0,\cdots,k-1$. For every $j$,  let $f_j$ be the lift of $F_j$ that is $\mathcal{C}^1$-close to the identity, and $g_j$ be a generating function of $f_j$.  We define the \emph{discrete symplectic action}
 $$g:\mathbb{R}^{2k} \rightarrow \mathbb{R}$$ by
$$g(z):= \sum_{j\in \mathbb{Z}_k}(<y_j,x_j-x_{j+1}>+g_j(x_{j+1},y_j)),$$
where $z=(z_0,...,z_{k-1})$ and $z_j=(x_j,y_j)$.

By a direct computation, we  know that for every $j\in\mathbb{Z}_k$,
$$\frac{\partial}{\partial x_j}g(z)=y_j-y_{j-1}+\partial_1 g_{j-1}(x_j,y_{j-1}),\quad \text{and }\frac{\partial}{\partial y_j}g(z)=x_j-x_{j+1}+\partial_2 g_{j}(x_{j+1},y_j).$$
So, $z\in\mathbb{R}^{2k}$ is a critical point of $g$ if and only if $z_{j+1}=f_j(z_j)$ for every $j\in\mathbb{Z}_k$, and therefore if and only if $z_0\in \mathbb{R}^{2}$ is a  fixed point of $f=f_{k-1}\circ\cdots\circ f_0$.

 In particular, each $f_j$ commutes with the integer translation, and so $g$ is invariant by the diagonal action of $\mathbb{Z}^2$ on $\mathbb{R}^{2k}$ and descends to a function $$G:\mathbb{R}^{2 k}/\mathbb{Z}^2\rightarrow \mathbb{R}.$$
Moreover, $[z]\in\mathbb{R}^{2 k}/\mathbb{Z}^{2}$ is a critical point of $G$ if and only if $z_{j+1}=f_j(z_j)$ for every $j\in\mathbb{Z}_k$, and therefore if and only if $[z_0]\in \mathbb{T}^{2}$ is a contractible fixed point of $F$, where $F=F_{k-1}\circ\cdots\circ F_0$.  In particular, critical points of $G$   one-to-one correspond to contractible fixed points of  $F$. Moreover, for any period $q\in\mathbb{N}$, contractible $q$-periodic points of $F$ correspond to the equivalent classes in $\mathbb{R}^{2kp}/\mathbb{Z}^2$ of  critical points of the discrete symplectic action $g^{\times q}:\mathbb{R}^{2kp}/\mathbb{Z}^{2}\rightarrow\mathbb{R}$ defined by
$$g^{\times q}(z):=\sum_{j\in \mathbb{Z}_{kq}}(<y_j,x_j-x_{j+1}>+g_{ (j \mod k)}(x_{j+1},y_j)),$$
where $z=(z_0,...,z_{kp-1})$ and $z_j=(x_j,y_j)$.

 \medskip

Moreover, if $[z_0]\in\mathbb{T}^2$ is a contractible fixed point of $F$, then by a suitable shift one can suppose that $[z_0]$ is fixed along the Hamiltonian flow, and hence the factors $F_j$  fixes $[z_0]$ for $j=0,\cdots,k-1$. So, $z_0$ is a fixed point of each $f_j$ and a critical point of each $g_j$.
 We define the the graded group of local homology
$$C_{*}(z_0^{\times kn}):=H_{*}(\{g^{\times n}<g^{\times n}(z_0^{\times kn})\cup\{z_0^{\times kn}\},g^{\times n}<g^{\times n}(z_0^{\times kn})).$$
Then $C_j(z_0^{\times kn})$ are always trivial for $j<\textrm{mor}(z_0^{\times kn})$ and $j>\textrm{mor}(z_0^{\times kn})+\textrm{nul}(z_0^{\times kn})$, where $\textrm{mor}(z_0^{\times kn})$ is the dimension of   negative eigenvector space of Hessian matrix of $g^{\times n}$ at $z_0^{\times kn}$, and $\textrm{nul}(z_0^{\times kn})$ is the dimension of the kernel of Hessian matrix of $g^{\times n}$ at $z_0^{\times kn}$.

We say that $z_0$ is a \emph{symplectically degenerate maximum} (resp. \emph{symplectically degenerate minimum}) if $z_0$ is an isolated local maximum (resp. minumum) of the generating functions $g_0,\cdots, g_{k-1}$, and  the local homology $C_{kn+1}(z_0^{\times kn})$ is non-trivial for infinitely many $n\in\mathbb{N}$.

\begin{prop}[\cite{mazzuchelli}\cite{Ruelle}]
Let $z=z_0^{\times k}$ be a critical point of $g$ such that  $C_{kn+1}(z^{\times n})$ is non-trivial for infinitely $n\in\mathbb{N}$. Then $1$ is the only eigenvalue of $DF([z_0])$, and the blow-up rotation number $\rho(I,[z_0])$ is equal to $0$ for any identity isotopy of $F$ fixing $[z_0]$.
\end{prop}

\begin{rmq}
In particular, a symplectically degenerate extremum satisfies the condition of the proposition, and hence is a degenerate fixed point of $F$.
\end{rmq}

\subsection{The index at a symplectically degenerate extremum}\label{S: index at a symplectically degenerate extremum}

As in the previous subsection, let $F: \mathbb{T}^2\rightarrow \mathbb{T}^2$ be a Hamiltonian diffeomorphism, and $F=F_{k-1}\circ\cdots\circ F_0$ be a factorization by Hamiltonian diffeomorphisms $F_i$ which are $\mathcal{C}^1$-close to the identity.
Let $f_j$ be the lift of $F_j$ to  $\mathbb{R}^2$ that is $\mathcal{C}^1$-close to the identity, and $g_j$ be a generating function of $f_j$, for $j=0,\cdots, k-1$. As was recalled in the previous subsection, if $z_0$ is a symplectically degenerate extremum, then the blow-up rotation number $\rho(I,[z_0])$ is equal to $0$ for any identity isotopy of $F$ fixing $[z_0]$. We will prove the following Proposition \ref{P: chapter3-index of symplectically degenerate extremum}, and then can deduce Theorem \ref{T: main-symplectic degenerate maxima implies accumlated by periodic points} as an immediate corollary of Theorem \ref{T: main-index 1 and rotation 0 implies  accumulated by periodic points}.

\begin{prop}\label{P: chapter3-index of symplectically degenerate extremum}
If $z_0$ is a symplectically extremum, then $i(F, [z_0])$ is equal to $1$.
\end{prop}

We will only deal with the case where $0$ is a symplectically degenerated maximum, the other case can be treated similarly. Let us begin by some lemmas.

\begin{lem}
Suppose that $g$  is a (local or global) generating function of a diffeomorphism $f$, and that $0$ is a local maximum  of $g$ such that the Hessian matrix of $g$ at $0$ is degenerate. Let $I=(f_t)$ be the identity isotopy of $f$ induced by $g$ as in Section \ref{S: index at a degenerate fixed point},  and $\theta(t)$ be a continuous function such that
$$\frac{J_{f_t}(0)\begin{pmatrix}\cos \theta(0)\\ \sin\theta(0)\end{pmatrix}}{\|J_{f_t}(0)\begin{pmatrix}\cos \theta(0)\\ \sin\theta(0)\end{pmatrix}\|}=\begin{pmatrix}\cos \theta(t)\\\sin\theta(t)\end{pmatrix}.$$
Then,  one can deduce that $\theta(1)\ge \theta(0)$.
\end{lem}
\begin{proof}
As in Section \ref{S: index at a degenerate fixed point}, we  denote the Hessian of $g$ at $0$ by
$$\textrm{Hess}(g)(0)=\begin{pmatrix} \varrho &  \sigma\\ \sigma &  \tau\end{pmatrix}.$$
Since $0$ is a local maximal point of $g$, $\textrm{Hess}g(0)$ is  negative semi-definite. So, we knows that
 $$\varrho\le 0, \quad \tau\le 0, \quad\text{and } \varrho\tau-\sigma^2=0.$$
As  was proved in Section \ref{S: index at a degenerate fixed point}, if $( a, b)$ is a unit eigenvector of $\mathrm{Hess}(g)(0)$ corresponding to the eigenvalue $0$ , then it is a common  eigenvector of $J_{f_{t}}(0)$ corresponding to the eigenvalue $1$. Recall that
$$\varrho a+\sigma b=0,\quad \text{and } \sigma a+\tau b=0.$$
So,
$$J_{f_t}(0)\begin{pmatrix}- b\\  a\end{pmatrix}=
\frac{1}{1-t\sigma}\begin{pmatrix}1 & t\tau\\ -t\varrho & 1-2t\sigma\end{pmatrix}\begin{pmatrix}- b\\  a\end{pmatrix}\\
= \begin{pmatrix}- b\\  a\end{pmatrix}+\frac{t(\varrho+\tau)}{1-t\sigma}\begin{pmatrix} a\\  b\end{pmatrix}.
$$
Therefore
$$J_{f_t}(0)\Omega=\Omega\begin{pmatrix}1 & \frac{t(\varrho+\tau)}{1-t\sigma}\\ 0 & 1\end{pmatrix},$$
where $\Omega=\begin{pmatrix}  a& - b\\  b& a\end{pmatrix}$ is a normal matrix. Since $\frac{t(\varrho+\tau)}{1-t\sigma}\le 0$,  one can deduce that $\theta(1)\ge \theta(0)$.
\end{proof}
\begin{figure}[h]
\centering
 \includegraphics[width=8cm]{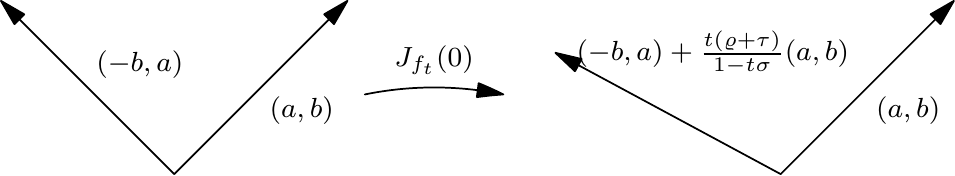}
\end{figure}

\begin{lem}\label{L: chapter3-symplectic-exist common conjugate}
If $0$ is a symplectically degenerate maximum,  then there exists a normal matrix $\Omega$ such that
$$\Omega^{-1}J_{f_j}(0)\Omega=\begin{pmatrix}1 & c_j\\ 0 & 1\end{pmatrix}$$
for $j=0,\cdots,k-1$, where $c_j$ are non-positive real numbers.
\end{lem}
\begin{proof}
Let $I_j=(f_{j,t})_{t\in[0,1]}$  be the local isotopy  of $f_j$ induced by $g_j$ as in section \ref{S: index at a degenerate fixed point}. Let $\mathcal{F}_j$ be the foliation whose leaves are the integral curves of the gradient vector field of $g_j$. As in Section \ref{S: index at a degenerate fixed point}, one can deduce that $0$ is a sink of $\mathcal{F}_j$ and that $\mathcal{F}_j$ is locally transverse to $I_j$. Therefore, one knows  that $\rho(I_j,0)\ge 0$, and that $\rho(I_j,0)=0$ if and only if $0$ is a degenerate fixed point of $f_j$.

  Let $\theta:[0,k]\rightarrow \mathbb{R}$ be a continuous function such that
$$\frac{J_{f'_{jt}}(0)\begin{pmatrix}\cos \theta(j)\\ \sin\theta(j)\end{pmatrix}}{\|J_{f'_{jt}}(j)\begin{pmatrix}\cos \theta(j)\\ \sin\theta(j)\end{pmatrix}\|}=\begin{pmatrix}\cos \theta(j+t)\\\sin\theta(j+t)\end{pmatrix}.$$
One knows that $\theta(j+1)> \theta(j)$ if $\rho(I_j,0)>0$, and $\theta(j+1)\ge \theta(j)$ if $\rho(I_j,0)=0$.
But we know that  $\rho(I_{k-1}\cdots I_0, z_0 )=\rho(I,[z_0])=0$, so there exists $\theta(0)\in\mathbb{R}$ and a continuous function $\theta$ as above such that $\theta(k)=\theta(0)$. Therefore, $\rho(I_j,z_0)=0$ for  $ j=0,\cdots, k-1$ and $\begin{pmatrix}\cos \theta(0)\\\sin \theta(0)\end{pmatrix}$ is a common eigenvector of $J_{f_j}(0)$ corresponding to the eigenvalue $1$. As in the proof of the previous lemma, we can prove this lemma by choosing
$$\Omega=\begin{pmatrix}\cos\theta(0) & -\sin\theta(0)\\ \sin\theta(0) & \cos\theta(0)\end{pmatrix}.$$
\end{proof}

\begin{lem}\label{L: chapter3-symplectic-conjugated does not change the maximality}
Suppose that $g$ is a (local or global) generating function of a diffeomorphism $f$,  that $0$ is a local maximum  of $g$, and that the Hessian matrix of $g$ at $0$ is degenerate. If $\Omega$ is a normal matrix, and if $f'=\Omega^{-1} f\Omega$ is  generated by $g'$ in a neighborhood of $0$, then $0$ is a local maximum of $g'$ and $\mathrm{Hess}(g')(0)$ is degenerate.
\end{lem}
\begin{proof}
Since $\mathrm{Hess}(g)(0)$ is degenerate, $1$ is an eigenvalue of $J_f(0)$ and hence an eigenvalue of $J_{f'}(0)$. So, $\mathrm{Hess}(g')(0)$ is degenerate.

Let $\mathcal{F}$ be the foliation whose leaves are integral curves of the gradient vector field of $g$, and $\mathcal{F}'$ be the foliation whose leaves are integral curves of the gradient  vector field of $g'$. Let $I_0$ be a local isotopy of $f$ satisfies $\rho(I_0,0)=0$,  and $I'_0$ be a local isotopy of $f$ satisfies $\rho(I'_0,0)=0$. As was proved in Section \ref{S: index at a degenerate fixed point}, $\mathcal{F}$ is locally transverse to $I_0$ and $\mathcal{F}'$ is locally transverse to $I'_0$.  Therefore, $\Omega\circ\mathcal{F}'$ is locally transverse to $I_0$. Since $0$ is a maximal point of $g$, it is a sink of $\mathcal{F}$. By the remark that follows Proposition \ref{P: pre-rotation type is unique},  one deduces that $0$ is a sink of $\Omega\circ \mathcal{F}'$, and hence a sink of $\mathcal{F}'$. Therefore, $0$ is a local maximum of $g'$.
\end{proof}

\begin{lem}\label{L: chapter3-symplectic-construct new generating function}
Let $g_0$ and $g_1$ be  local generating functions of  $f_0$ and $f_1$ respectively such that $0$ is a local maximal point of both $g_0$ and $g_1$, and that the Hessian matrices satisfy
$$\text{Hess}(g_i)(0)=\begin{pmatrix} 0 & 0\\ 0 & c_i\end{pmatrix},$$
where $c_i\le 0$ for $i=0,1.$
Then there exists  a function $g$  which is a generating function of $f=f_1\circ f_0$ in a neighborhood of  $0$. Moreover, $0$ is a local maximal point of $g$, and the Hessian matrix satisfies
$$\text{Hess}(g)(0)=\begin{pmatrix} 0 & 0\\ 0 & c_0+c_1\end{pmatrix}.$$
\end{lem}

\begin{proof}
Suppose that $g_0(0)=g_1(0)=0$. Since $0$ is a local maximal point of both $g_0$ and $g_1$, it is a critical point of both $g_0$ and $g_1$. So,
$$\partial_1 g_0(0,0)=\partial_2 g_0(0,0)=\partial_1 g_1(0,0)=\partial_2 g_1(0,0)=0.$$
Write $(x_1,y_1)=f_0(x_0,y_0)$ and  $(x_2,y_2)=f_1(x_1,y_1)$. By definition of generating functions, one knows that
 \begin{equation}\label{Eq: construct new generating fuction-1}
 y_1-y_0+\partial_1 g_0(x_1,y_0)=0, \quad \text{and } x_1-x_2+ \partial_2 g_1(x_2,y_1)=0.
 \end{equation}
Note that
$$\det\begin{pmatrix}\partial^2_{11} g_0 (0,0) & 1\\ 1& \partial^2_{22} g_1(0, 0)\end{pmatrix}=-1.$$
So, by implicit function theorem, there exists a $\mathcal{C}^1$ diffeomorphism $\varphi: W\rightarrow W'$ such that $(x_1,y_1)=\varphi(x_2,y_0)$, where $W$ and $W'$ are sufficiently small neighborhoods of $0$ in $\mathbb{R}^2$.
Moreover,
$$J_{\varphi}(0,0)=-\begin{pmatrix}\partial^2_{11}g_0(0,0) & 1\\ 1& \partial^2_{22}g_1(0,0)\end{pmatrix}^{-1}\begin{pmatrix}0 & \partial^2_{12}g_0(0,0)-1\\ \partial^2_{12}g_1(0,0)-1 & 0\end{pmatrix}=\begin{pmatrix}1 & -c_1\\ 0 & 1\end{pmatrix}.$$

Let
$$g(x_2, y_0)=g_0(x_1,y_0)+g_1(x_2,y_1)+(x_2-x_1)(y_0-y_1),$$
where $(x_1,y_1)=\varphi(x_2,y_0)$. We know that $g(0,0)=0$. In a neighborhood of $0$, by a direct computation and equations (\ref{Eq: construct new generating fuction-1}), one knows that
\begin{align*}
\partial_1 g(x_2,y_0)=& \partial_1 g_0(x_1,y_0)\partial_1 x_1(x_2,y_0)+\partial_1 g_1(x_2,y_1)+\partial_2 g_1(x_2,y_1)\partial_1 y_1(x_2,y_0)\\
&+(1-\partial_1 x_1(x_2,y_0))(y_0-y_1)-\partial_1 y_1(x_2,y_0)(x_2-x_1)\\
=&\partial_1 g_0(x_1,y_0)+\partial_1 g_1(x_2,y_1).
\end{align*}
Similarly, one gets
$$\partial_2 g(x_2,y_0)=\partial_2 g_0(x_1,y_0)+\partial_2 g_1(x_2,y_1).$$
 So, $g$ is a $\mathcal{C}^2$ function near $0$. Moreover,
$$\partial^2_{12} g(0,0)=\partial^2_{11}g_0(0,0)\partial_2 y_1(0,0)+\partial^2_{12}g_0(0,0)+\partial^2_{12}g_1(0,0)\partial_2  y_1(0,0)=0.$$
Because $g_0$ and $g_1$ locally generate $f_0$ and $f_1$ respectively, one deduces
$$\partial_1 g(x_2,y_0)=-(y_2-y_0) \quad \text{and} \quad
\partial_2 g(x_2,y_0)=x_2-x_0.$$
 Therefore, $g$ is a generating function of $f$ in a neighborhood of $0$.

By a direct computation, one gets
$$\partial^2_{11}g(0,0)=\partial^2_{11}g_0(0,0)\partial_1 x_1(0,0)+\partial^2_{11}g_1(0,0)+\partial^2_{12}g_1(0,0)\partial_1 y_1(0,0)=0,$$
and
$$\partial^2_{22}g(0,0)=\partial^2_{12}g_0(0,0)\partial_2 x_1(0,0)+\partial^2_{22}g_0(0,0)+\partial^2_{22}g_1(0,0)=c_0+c_1.$$
So,
 $$\mathrm{Hess}(g)(0)=\begin{pmatrix} 0 & 0\\ 0 & c_0+c_1\end{pmatrix}.$$

\bigskip

We will conclude by proving that $0$ is a locally maximum of $g$.  Let $\varepsilon>0$ be a small real number such that $|\varepsilon c_1|<1$.
We will prove that in a sufficiently small neighborhood of $0$,
$$ g(x_2,y_0)\le g_0(x_1+\frac{1}{\varepsilon}(y_0-y_1),y_0)+g_1(x_2, y_1+\varepsilon(x_2-x_1))\le 0,$$
and hence $0$ is a locally maximum of $g$ because the second inequality is strict for $(x_2,y_0)\ne 0$.
Indeed, by Taylor's theorem and equations (\ref{Eq: construct new generating fuction-1}),  one knows that in a sufficiently small neighborhood of $0$,
\begin{align*}
g_0(x_1+\frac{1}{\varepsilon}(y_0-y_1),y_0)&=g_0(x_1,y_0)+\frac{1}{\varepsilon}\partial_{1} g_0 (x_1,y_0)(y_0-y_1)+\frac{1}{2\varepsilon^2}\partial^2_{11}g_0(\xi,y_0)(y_0-y_1)^2\\
&=g_0(x_1,y_0)+\frac{1}{\varepsilon}(y_0-y_1)^2+\frac{1}{2\varepsilon^2}\partial^2_{11}g_0(\xi,y_0)(y_0-y_1)^2,
\end{align*}
where $\xi$ is a real number between $x_1$ and $x_1+\frac{1}{\varepsilon}(y_0-y_1)$. Similarly, one deduces that in sufficiently small neighborhood of $0$,
$$g_1(x_2,y_1+\varepsilon(x_2-x_1))=g_1(x_2,y_1)+\varepsilon(x_2-x_1)^2+\frac{\varepsilon^2}{2}\partial^2_{22}g_1(x_2,\eta )(x_2-x_1)^2,
$$
where $\eta$ is a real number between $y_1$ and $y_1+\varepsilon(x_2-x_1)$.
So,
\begin{align*}
 g(x_2,y_0)=&g_0(x_1+\frac{1}{\varepsilon}(y_0-y_1),y_0)+g_1(x_2, y_1+\varepsilon(x_2-x_1))\\
 &-\frac{1}{2\varepsilon}(y_0-y_1)^2-\frac{\varepsilon}{2}(x_2-x_1)^2+(x_2-x_1)(y_0-y_1)\\
&-\frac{1}{2\varepsilon}(1+\frac{1}{\varepsilon}\partial_{11}g_0(\xi,y_0))(y_0-y_1)^2-\frac{\varepsilon}{2}(1+\varepsilon\partial_{22}g_1(x_2,\eta))(x_2-x_1)^2.
\end{align*}
For $(x_2,y_0)\ne 0$ that is in a sufficiently small neighborhood of $0$, one can suppose that
$$|\frac{1}{\varepsilon^2}\partial_{11} g_0(\xi,y_0)|<1, \quad \text{and } |\varepsilon \partial_{22} g_1(x_2,\eta)|<1.$$
So,
$$ g(x_2,y_0)\le g_0(x_1+\frac{1}{\varepsilon}(y_0-y_1),y_0)+g_1(x_2, y_1+\varepsilon(x_2-x_1)).$$
\end{proof}

Now, we begin the proof of Proposition \ref{P: chapter3-index of symplectically degenerate extremum}.

\begin{proof}[Proof of Propostion \ref{P: chapter3-index of symplectically degenerate extremum}]
Suppose that $z_0$ is a symplectically degenerated maximum. By Lemma \ref{L: chapter3-symplectic-exist common conjugate}, there exists a coordinate transformation such that in the new coordinate system the Jacobian of each $f_j$ at $z_0$ has the form
 $$\begin{pmatrix} 1 & c_j \\ 0 & 1\end{pmatrix}$$
   where $c_j$ is a non-positive real number. We consider everything in the new coordinate system.   Each $f_j$ can be locally generated by a generating function $g'_j$, and the Hessian of $g'_j$ at $z_0$ has the form
    $$\begin{pmatrix} 0 & 0 \\ 0 & c_j\end{pmatrix}.$$
    By Lemma \ref{L: chapter3-symplectic-conjugated does not change the maximality}, $z_0$ is a local maximum of each $g'_j$. So, by Lemma \ref{L: chapter3-symplectic-construct new generating function}, we can construct  a generating function $g'$ such that
 \begin{itemize}
 \item[-] $z_0$ is a local maximum  of $g'$,
 \item[-]  $\mathrm{Hess}(g')(z_0)$ is degenerate,
\item[-] $g'$ generates $f=f_{k-1}\cdots f_0 $ in a neighborhood of $z_0$.
\end{itemize}
So, by Proposition \ref{P: chapter3-diffeo-the index of degenerate fixed point that is an extremum of genrating function}, we know  $i(f,z_0)$ is equal to $1$, and hence $i(F,[z_0])$ is equal to $1$.
\end{proof}

\section*{Acknowledgements}

I wish to thank Patrice Le Calvez for  proposing me the subject.

\bibliographystyle{alpha}

\bibliography{reference}

\begin{thebibliography}{KLCN14}

\bibitem[Bir26]{birkhoff}
George~D. Birkhoff.
\newblock An extension of {P}oincar\'e's last geometric theorem.
\newblock {\em Acta Math.}, 47(4):297--311, 1926.

\bibitem[Boy92]{Boyland}
Philip Boyland.
\newblock Rotation sets and monotone periodic orbits for annulus
  homeomorphisms.
\newblock {\em Comment. Math. Helv.}, 67(2):203--213, 1992.

\bibitem[Bro12]{Brouwer}
L.E.J. Brouwer.
\newblock Beweis des ebenen translationssatzes.
\newblock {\em Mathematische Annalen}, 72(1):37--54, 1912.

\bibitem[CB88]{Casson}
Andrew~J. Casson and Steven~A. Bleiler.
\newblock {\em Automorphisms of surfaces after {N}ielsen and {T}hurston},
  volume~9 of {\em London Mathematical Society Student Texts}.
\newblock Cambridge University Press, Cambridge, 1988.

\bibitem[CZ86]{cznondegenarate}
Charles~C. Conley and Eduard Zehnder.
\newblock A global fixed point theorem for symplectic maps and subharmonic
  solutions of {H}amiltonian equations on tori.
\newblock In {\em Nonlinear functional analysis and its applications, {P}art 1
  ({B}erkeley, {C}alif., 1983)}, volume~45 of {\em Proc. Sympos. Pure Math.},
  pages 283--299. Amer. Math. Soc., Providence, RI, 1986.

\bibitem[Fra88]{Franks}
John Franks.
\newblock Generalizations of the {P}oincar\'e-{B}irkhoff theorem.
\newblock {\em Ann. of Math. (2)}, 128(1):139--151, 1988.

\bibitem[Fra92]{FranksBrouwertranslationtheorem}
John Franks.
\newblock A new proof of the {B}rouwer plane translation theorem.
\newblock {\em Ergodic Theory Dynam. Systems}, 12(2):217--226, 1992.

\bibitem[Fra96]{Franksrotationvector}
John Franks.
\newblock Rotation vectors and fixed points of area preserving surface
  diffeomorphisms.
\newblock {\em Trans. Amer. Math. Soc.}, 348(7):2637--2662, 1996.

\bibitem[GLCP96]{LecalvezthmNaishul}
Jean-Marc Gambaudo, Patrice Le~Calvez, and {\'E}lisabeth P{\'e}cou.
\newblock Une g\'en\'eralisation d'un th\'eor\`eme de {N}aishul.
\newblock {\em C. R. Acad. Sci. Paris S\'er. I Math.}, 323(4):397--402, 1996.

\bibitem[Gui94]{Guillou}
Lucien Guillou.
\newblock Th\'eor\`eme de translation plane de {B}rouwer et g\'en\'eralisations
  du th\'eor\`eme de {P}oincar\'e-{B}irkhoff.
\newblock {\em Topology}, 33(2):331--351, 1994.

\bibitem[Han99]{Handel}
Michael Handel.
\newblock A fixed-point theorem for planar homeomorphisms.
\newblock {\em Topology}, 38(2):235--264, 1999.

\bibitem[Hin09]{hinstondegenerate}
Nancy Hingston.
\newblock Subharmonic solutions of {H}amiltonian equations on tori.
\newblock {\em Ann. of Math. (2)}, 170(2):529--560, 2009.

\bibitem[Jau14]{Jaulent}
Olivier Jaulent.
\newblock Existence d'un feuilletage positivement transverse\' a un
  hom\'eomorphisme de surface.
\newblock {\em Ann. Inst. Fourier}, 2014.
\newblock (accepted).

\bibitem[KLCN14]{Lecalvezprimeendsrotationnumberandperiodicpoints}
Andres Koropecki, Patrice Le~Calvez, and Meysam Nassiri.
\newblock Prime ends rotation numbers and periodic points.
\newblock {\em Duke Math. J.}, 2014.
\newblock (accepted).

\bibitem[LC99]{lecalvezindicesup1}
Patrice Le~Calvez.
\newblock Une propri\'et\'e dynamique des hom\'eomorphismes du plan au
  voisinage d'un point fixe d'indice {$>1$}.
\newblock {\em Topology}, 38(1):23--35, 1999.

\bibitem[LC03]{lecalvezindices}
Patrice Le~Calvez.
\newblock Dynamique des hom\'eomorphismes du plan au voisinage d'un point fixe.
\newblock {\em Ann. Sci. \'Ecole Norm. Sup. (4)}, 36(1):139--171, 2003.

\bibitem[LC05]{lecalvezfeuilletage}
Patrice Le~Calvez.
\newblock Une version feuillet\'ee \'equivariante du th\'eor\`eme de
  translation de {B}rouwer.
\newblock {\em Publ. Math. Inst. Hautes \'Etudes Sci.}, (102):1--98, 2005.

\bibitem[LC08]{lecalveztourner}
Patrice Le~Calvez.
\newblock Pourquoi les points p\'eriodiques des hom\'eomorphismes du plan
  tournent-ils autour de certains points fixes?
\newblock {\em Ann. Sci. \'Ec. Norm. Sup\'er. (4)}, 41(1):141--176, 2008.

\bibitem[LR08]{lerouxparabolic}
Fr{\'e}d{\'e}ric Le~Roux.
\newblock A topological characterization of holomorphic parabolic germs in the
  plane.
\newblock {\em Fund. Math.}, 198(1):77--94, 2008.

\bibitem[LR13]{lerouxrotation}
Fr{\'e}d{\'e}ric Le~Roux.
\newblock L'ensemble de rotation autour d'un point fixe.
\newblock {\em Ast\'erisque}, (350):x+109, 2013.

\bibitem[Mat01]{Matsumoto}
Shigenori Matsumoto.
\newblock Types of fixed points of index one of surface homeomorphisms.
\newblock {\em Ergodic Theory Dynam. Systems}, 21(4):1181--1211, 2001.

\bibitem[Maz13]{mazzuchelli}
Marco Mazzucchelli.
\newblock Symplectically degenerate maxima via generating functions.
\newblock {\em Math. Z.}, 275(3-4):715--739, 2013.

\bibitem[Mil06]{milnor}
John Milnor.
\newblock {\em Dynamics in one complex variable}, volume 160 of {\em Annals of
  Mathematics Studies}.
\newblock Princeton University Press, Princeton, NJ, third edition, 2006.

\bibitem[MS98]{Mcduff}
Dusa McDuff and Dietmar Salamon.
\newblock {\em Introduction to symplectic topology}.
\newblock Oxford Mathematical Monographs. The Clarendon Press, Oxford
  University Press, New York, second edition, 1998.

\bibitem[Na{\u\i}82]{Naishul}
V.~A. Na{\u\i}shul$'$.
\newblock Topological invariants of analytic and area-preserving mappings and
  their application to analytic differential equations in {${\bf C}^{2}$} and
  {${\bf C}P^{2}$}.
\newblock {\em Trudy Moskov. Mat. Obshch.}, 44:235--245, 1982.

\bibitem[PS87]{Slaminka}
Stephan Pelikan and Edward~E. Slaminka.
\newblock A bound for the fixed point index of area-preserving homeomorphisms
  of two-manifolds.
\newblock {\em Ergodic Theory Dynam. Systems}, 7(3):463--479, 1987.

\bibitem[Rue85]{Ruelle}
David Ruelle.
\newblock Rotation numbers for diffeomorphisms and flows.
\newblock {\em Ann. Inst. H. Poincar\'e Phys. Th\'eor.}, 42(1):109--115, 1985.

\bibitem[Spa66]{spanier}
Edwin~H. Spanier.
\newblock {\em Algebraic topology}.
\newblock Springer-Verlag, New York-Berlin, 1966.

\bibitem[SZ92]{szweaklynondegenerate}
Dietmar Salamon and Eduard Zehnder.
\newblock Morse theory for periodic solutions of {H}amiltonian systems and the
  {M}aslov index.
\newblock {\em Comm. Pure Appl. Math.}, 45(10):1303--1360, 1992.

\end{thebibliography}

\end{document}